\documentclass[10pt,reqno]{amsart}
\usepackage[colorlinks=true,
pdfstartview=FitV, linkcolor=cyan, citecolor=magenta,
urlcolor=blue]{hyperref}
\usepackage{amsmath,amsfonts,latexsym,amssymb}
\usepackage{mathrsfs}
\usepackage[latin1]{inputenc}
\usepackage[T1]{fontenc}
\usepackage{ae,aecompl}
\usepackage{braket}
\usepackage{color}
\usepackage{graphicx}
\usepackage{subfig}
\newtheorem{theorem}{Theorem}[section]
\newtheorem{lemma}[theorem]{Lemma}
\newtheorem{proposition}[theorem]{Proposition}
\newtheorem{corollary}[theorem]{Corollary}
\newtheorem{definition}[theorem]{Definition}
\renewcommand{\leq}{\leqslant}
\renewcommand{\geq}{\geqslant}
\newcommand{\rmks}{\vskip 0.5 truecm\noindent{\sc Remarks: }}

\newcommand{\ad}{\operatorname{Ad}}
\newcommand{\interior}{\operatorname{int}}

\newcommand{\seq}[1]{
\{{#1}_m\}_{m\in\mathbb N}}

\renewcommand{\d}{{\rm d}}

\newcommand{\rpn}{\mathbb{RP}(m)}
\newcommand{\sln}{\mathsf{SL}_m(\mathbb R)}
\newcommand{\psln}{\mathsf{PSL}_m(\mathbb R)}

\newcommand{\Var }{\operatorname{Var }}
\newcommand{\Cov }{\operatorname{Cov }}

\newcommand{\Cg}{{\mathcal C}_g(\Gamma, \mathsf G)}
\newcommand{\Cm}{{\mathcal C}(\Gamma,m)}
\newcommand{\Cc}{{\mathcal C}_c}

\newcommand{\tr}{\operatorname{Tr}}
\newcommand{\LL}{\mathsf L}

\newcommand{\TT}{\mathsf{T}}
\newcommand{\II}{{\bf I}}
\newcommand{\JJ}{{\bf J}}
\newcommand{\PP}{{\bf P}}
\newcommand{\p}{{\bf{p}}}
\newcommand{\q}{{\bf{q}}}
\newcommand{\m}{{\bf{m}}}
\newcommand{\CC}{\mathbb C}

\newcommand{\Real}{\mathbb R}
\newcommand{\bg}{\partial_\infty\Gamma}
\newcommand{\Gg}{{\mathsf U_0}\Gamma}
\newcommand{\Ug}{{\mathsf U}_\rho\Gamma}

\DeclareMathOperator{\tope}{top}

\newcommand{\hsf}{\mathsf}
\newcommand{\hdelta}{\widehat{\delta}}
\newcommand{\eps}{\varepsilon}
\newcommand{\ms}{\mathsf}
\newcommand{\mk}{\mathfrak}
\newcommand{\bVert}{\big\Vert}
\newcommand{\im}{\operatorname{Im}}
\newcommand{\tR}{_{t\in\mathbb R}}

\renewcommand{\hom}{\operatorname{Hom}}

\title{The pressure metric for Anosov representations}
\author[Bridgeman]{Martin Bridgeman}
\address{Boston College, Chestnut Hill, MA 02467 USA}
\author[Canary]{Richard Canary}
\address{University of Michigan, Ann Arbor, MI 41809 USA}
\author[Labourie]{Fran\c cois Labourie}
\address{ Univ. Nice Sophia-Antipolis, Laboratoire Jean Dieudonn\'e,  UMR 7351, Nice F-06000; FRANCE}
\author[Sambarino]{Andres Sambarino}
\address{ Univ. Paris-Sud, Laboratoire de Math\'ematiques, Orsay F-91405 Cedex; CNRS, Orsay cedex, F-91405 FRANCE}
\thanks{Canary was partially supported by NSF grant DMS - 1006298.
Labourie and Sambarino were partially supported by
the European Research Council under the {\em European Community}'s seventh Framework Programme (FP7/2007-2013)/ERC {\em grant agreement} ${\rm n}^{\tiny o}$ FP7-246918, as well as by the  ANR program ETTT (ANR-09-BLAN-0116-01) The authors  also
acknowledge  support from U.S. National Science Foundation grants 
DMS 1107452, 1107263, 1107367 "RNMS: GEometric structures And Representation varieties" (the GEAR Network)}

\begin{document}
\maketitle

\begin{abstract}
Using the thermodynamic formalism, we introduce a notion of intersection for projective Anosov representations, 
show analyticity results for the intersection and the entropy, and rigidity results for the intersection.
We use the renormalized intersection to produce an ${\rm Out}(\Gamma)$-invariant  Riemannian metric
on the  smooth points  of the deformation space of irreducible, generic, projective Anosov representations of a word 
hyperbolic group $\Gamma$  into $\sln$.
In particular, we produce  mapping class group invariant Riemannian metrics on  Hitchin components 
which restrict to the Weil--Petersson metric on the Fuchsian loci. 
Moreover, we produce ${\rm Out}(\Gamma)$-invariant metrics on deformation
spaces of convex cocompact representations into $\mathsf{PSL}_2(\mathbb C)$ and show that the Hausdorff dimension
of the limit set varies analytically over analytic families of convex cocompact representations into any rank 1
semi-simple Lie group.
\end{abstract}

\section{Introduction}

In this paper we produce a mapping class group invariant Riemannian metric on a Hitchin component of the
character variety of representations of a closed surface group into $\sln$
whose restriction to the Fuchsian locus is a multiple of the Weil-Petersson metric. More generally, we produce 
a ${\rm Out}(\Gamma)$-invariant  Riemannian metric on the  smooth generic points  of the deformation space of 
irreducible, projective Anosov
representations of a word hyperbolic group $\Gamma$ into $\sln$.  We use
Pl\"ucker representations to produce metrics on deformation spaces of convex cocompact
representations into $\ms{PSL}_2(\mathbb C)$ and  on the smooth points of deformation spaces of
Zariski dense Anosov representations into an arbitrary semi-simple Lie group. 

Our metric is produced using the thermodynamic formalism
developed by Bowen \cite{bowen1,bowen2}, Parry--Pollicott \cite{parry-pollicott}, Ruelle  \cite{ruelle}
and others. It generalizes earlier work done in the Fuchsian
and quasifuchsian cases by McMullen \cite{mcmullen} and Bridgeman \cite{bridgeman-pressure}.
In order to use the thermodynamic formalism, we associate a natural flow $\Ug$ to any projective Anosov 
representation $\rho,$
and show that it is a topologically transitive metric Anosov flow and is a H\"older reparameterization
of the geodesic flow $\Gg$ of  $\Gamma$ as defined by Gromov.
We then see that entropy varies analytically over any smooth analytic family of projective
Anosov homomorphisms of $\Gamma$ into $\sln$. 
As a consequence, again
using the Pl\"ucker embedding, we see that the Hausdorff dimension of the limit set varies analytically over
analytic families of convex cocompact representations into a rank one semi-simple Lie group. We also
introduce a renormalized intersection  $\JJ$ on the space of projective Anosov representations. Our metric
is given by the Hessian of  this renormalised intersection $\JJ$.
 
\bigskip

We now introduce the notation necessary to give more careful statements of our results.
Let $\Gamma$ be a word hyperbolic group with Gromov boundary $\bg$. 
Loosely speaking, a representation $\rho:\Gamma\to \sln$ is {\em projective Anosov} if it
has transverse projective limit maps, the image of every infinite order element is proximal,
and the proximality ``spreads uniformly'' (see Section \ref{convex-anosov} for a careful definition).
An element $A\in\sln$ is {\em proximal} if its action on $\rpn$ has an attracting fixed point.
A representation $\rho:\Gamma\to \sln$ is  said to have {\em transverse projective limit maps}
if there exist continuous $\rho$-equivariant maps 
\hbox{$\xi:\bg\to\rpn$} and \hbox{$\theta:\bg\to \rpn^*$} such that if $x$ and $y$ are distinct points in $\bg$, then
$$
\xi(x)\oplus\theta(y)=\mathbb R^m
$$
(where we identify $\rpn^*$ with the Grassmanian of $(m-1)$-dimensional vector subspaces of $\mathbb R^m$).
If $\gamma\in\Gamma$ has infinite order, $\rho$ is projective Anosov and $\gamma^+$ is the attracting fixed
point of the action of $\gamma$ on $\partial\Gamma$, then $\xi(\gamma^+)$ is the attracting fixed point for
the action of $\rho(\gamma)$ on $\rpn$.
Moreover, Guichard and Wienhard \cite[Proposition 4.10]{guichard-wienhard} proved that
every irreducible representation $\rho:\Gamma\to \sln$ with transverse projective limit maps is projective Anosov.

If $\rho$ is a projective Anosov representation, we can associate to every conjugacy class 
$[\gamma]$ of $\gamma\in\Gamma$ its {\em spectral radius}  $\Lambda(\gamma)(\rho).$
The collection of these radii form the {\em radius spectrum} of $\rho$. 
For every positive real number $T$ we define
$$
R_T(\rho)=\{[\gamma]\mid \log(\Lambda(\gamma)(\rho))\leq T\}.
$$
We will see that $R_T(\rho)$ is finite  (Proposition \ref{finiteclass}).
We also define  the {\em entropy} of a representation by
$$
h(\rho)=\lim_{T\to\infty}\frac{1}{T}\log\sharp(R_T(\rho)).
$$ 

If $\rho_1$ and $\rho_2$ are two projective Anosov representations, we  define their {\em intersection} by
$$
\II(\rho_1,\rho_2)=\lim_{T\to\infty}\left(\frac{1}{\sharp (R_T(\rho_1))}\sum_{[\gamma]\in R_T(\rho_0)}\frac{\log(\Lambda(\gamma)(\rho_2))}{\log(\Lambda(\gamma)(\rho_1))}\right).
$$
We also define the {\em renormalised intersection} by
$$
\JJ(\rho_1,\rho_2)=\frac{h(\rho_2)}{h(\rho_1)}\ \II(\rho_1,\rho_2).
$$
We prove, see Theorem \ref{analytic}, that all these quantities are well defined and obtain 
the following inequality and rigidity result for the renormalised intersection.
Let $\pi_m:\sln\to\ms{PSL}_m(\mathbb R)$ be the projection map. 
If $\rho:\Gamma\to \sln$ is a representation, let $\ms G_\rho$ be the Zariski closure of $\rho(\Gamma)$.

\begin{theorem}{\sc[Intersection]}\label{intersection}
If $\Gamma$ is a word hyperbolic group and  \hbox{$\rho_1:\Gamma\to \ms{SL}_{m_1}(\mathbb R)$} and 
$\rho_2:\Gamma\to \ms{SL}_{m_2}(\mathbb R)$ are projective Anosov representations, then
$$ {\JJ}(\rho_1,\rho_2)\geq 1. $$ 

Moreover, if $\rho_1$ and $\rho_2$ are irreducible, $\ms G_{\rho_1}$ and  $\ms G_{\rho_2}$ are connected
and \hbox{${\JJ}(\rho_1,\rho_2)= 1$}, then there exists an isomorphism 
$\phi:\pi_{m_1}(\ms G_{\rho_1})\to \pi_{m_2}(\ms G_{\rho_2})$ such that 
$$\phi\circ\pi_{m_1}\circ\rho_1= \pi_{m_2}\circ\rho_2.$$
\end{theorem}

We also establish a spectral rigidity result.  If $\rho:\Gamma\to\sln$ is projective Anosov and
$\gamma\in\Gamma$, then let
$
\LL(\gamma)(\rho)
$
denote the eigenvalue of maximal absolute value of $\rho(\gamma)$,
so 
$$\Lambda(\gamma)(\rho)=|\ms L(\gamma)(\rho)|.$$

\begin{theorem}{\sc [Spectral rigidity]}\label{theorem:length} 
Let $\Gamma$ be a word hyperbolic group and let $ \rho_1:\Gamma\to\sln$  and $ \rho_2:\Gamma\to\sln$ 
be  projective Anosov representations with limit maps $\xi_1$ and $\xi_2$ such that
$$\LL(\gamma)(\rho_1)=\LL(\gamma)(\rho_2)$$
for every $\gamma$ in $\Gamma.$ 
Then there exists $g\in\ms{GL}_m(\mathbb R)$ such that $g\xi_1=\xi_2.$

Moreover, if $\rho_1$ is irreducible, then $g\rho_1 g^{-1}= \rho_2.$
\end{theorem}

We now introduce the deformation spaces which occur in our work. In section \ref{MCR}, 
we will see that each of these deformation spaces is  a real analytic manifold.
Let us introduce some terminology. 
If $\ms G$ is a reductive subgroup of $\sln$, 
we say that an element of $\ms G$ is  {\em generic} if its centralizer is
a maximal torus in $\ms G$. For example, an element of $\sln$ is generic if and only if it is diagonalizable over $\mathbb C$ with
distinct eigenvalues. We say that a representation $\rho:\Gamma\to \ms G$ is \hbox{\em $\ms G$-generic} if the
Zariski closure of $\rho(\Gamma)$ contains a generic element of $\ms G$. 
Finally, we say that $\rho\in\hom(\Gamma,\ms G)$ is {\em regular} if it is a smooth point of the algebraic variety $\hom(\Gamma,\ms G)$. 
\begin{itemize}
\item Let $\Cm$ denote the space of (conjugacy classes of) regular, irreducible, projective Anosov  representations
of $\Gamma$ into $\sln$.
\item Let $\Cg$ denote the space of  (conjugacy classes of) $\ms G$-generic, regular, irreducible, projective Anosov representations.
\end{itemize}

We show that the entropy and the renormalised intersection vary analytically over our deformation spaces.
Moreover, we obtain analyticity on analytic families of projective Anosov homomorphisms. An analytic family of
projective Anosov homomorphisms is a continuous map $\beta:M\to \operatorname{Hom}(\Gamma,\sln)$ such
that $M$ is an analytic manifold, $\beta_m=\beta(m)$ is projective Anosov for all $m\in M$, and 
$m\to \beta_m(\gamma)$ is an analytic map of $M$ into $\sln$ for all $\gamma\in\Gamma$.

\begin{theorem}{\sc [Analyticity]}\label{analytic}
If $\Gamma$ is a word hyperbolic group, then
the entropy $h$  and the renormalised intersection $\JJ$ are well-defined positive, 
${\rm Out}(\Gamma)$-invariant analytic functions on the spaces $\Cm$ and \hbox{$\Cm\times\Cm$} respectively. More generally,
they are analytic functions on any analytic family of projective Anosov homomorphisms.

Moreover, let \hbox{$\gamma:(-1,1) \to \Cm$}  be  any analytic path with values in the deformation space,  let 
\hbox{${\bf J}_\gamma(t)={\bf J}(\gamma(0),\gamma(t))$} then
\begin{equation}
\left.\frac{\rm d}{{\d}t}\right\vert_{t=0}{\bf J}_\gamma=0 \  \hbox{and} \ \left.\frac{\rm d^2}{{\d}t^2}\right\vert_{t=0}{\bf J}_\gamma\geq 0.
\end{equation}
\end{theorem}

Theorem \ref{analytic} allows us to define a non-negative  analytic 2-tensor on $\Cg$. The pressure form  is defined to be the Hessian of the restriction 
of the renormalised intersection ${\bf J}$.
Our main result is the following.

\begin{theorem}{\sc[Pressure metric]}
\label{path metric}
Let $\Gamma$ be a word hyperbolic group and let  $\ms G$ be a reductive subgroup of $\sln$.
The pressure form is an analytic  ${\rm Out}(\Gamma)$-invariant Riemannian metric
on  $\Cg$.
\end{theorem}

If $S$ is a closed, connected,  orientable, hyperbolic surface, 
Hitchin \cite{hitchin} exhibited a component $\mathcal H_m(S)$ of
$\operatorname{Hom}(\pi_1(S),\ms{PSL}_m(\mathbb R))/\ms{PGL}_m(\mathbb R)$
now called the {\em Hitchin component},   which is an analytic manifold diffeomorphic to a ball. 
Each Hitchin component contains a Fuchsian locus which consists of representations
obtained by composing Fuchsian representations   of $\pi_1(S)$ into $\ms{PSL}_2(\mathbb R)$ with the
irreducible representation  $\tau_m:\ms{PSL}_2(\mathbb R)\to\ms{PSL}_m(\mathbb R)$.
The representations in a Hitchin component are called {\em Hitchin representations} and can be lifted to
representations into $\ms{SL}_m(\mathbb R)$. 
Labourie \cite{labourie-anosov} showed that  lifts of Hitchin representations are projective Anosov, irreducible
and $\sln$-generic. In particular,  if $\rho_i:\pi_1(S)\to \ms{PSL}_m(\mathbb R)$ are
Hitchin representations, then one can define $h(\rho_i)$, $\II(\rho_1,\rho_2)$ and $\JJ(\rho_1,\rho_2)$ just as for
projective Anosov representations. 
Guichard has recently announced a classification of the possible  Zariski closures of Hitchin representations,
see Section \ref{hitchin rigidity} for a statement. As a corollary of Theorem \ref{intersection} and Guichard's work
we obtain a stronger rigidity result for Hitchin representations.

\begin{corollary}{\sc[Hitchin rigidity]}\label{hitch-rigid}
Let $S$ be a closed, orientable surface and let
$\rho_1\in \mathcal H_{m_1}(S)$ and $\rho_2\in \mathcal H_{m_2}(S)$ 
be two Hitchin representations such that  
$$
\JJ(\rho_1,\rho_2)=1.
$$
Then, either
\begin{itemize}
\item
$m_1=m_2$ and $\rho_1=\rho_2$ in $\mathcal H_{m_1}(S)$, or
\item
there exists an element $\rho$ of the Teichm\"uller space $\mathcal T(S)$ so that
\hbox{$\rho_1=\tau_{m_1}(\rho)$} and  \hbox{$\rho_2=\tau_{m_2}(\rho)$}.
\end{itemize}
 \end{corollary}
 
In section \ref{convex projective} we use work of Benoist \cite{benoist-divisible0,benoist-divisible1} to obtain 
a similar rigidity result for representations which arise as monodromies of strictly convex projective structures
on compact manifolds with word hyperbolic fundamental group. We will call such representations Benoist representations.

Each Hitchin component lifts to a component
of $\mathcal{C}_g(\pi_1(S),\sln)$.
As a corollary of Theorem \ref{path metric} and work of Wolpert \cite{wolpert} we obtain:

\begin{corollary}{\sc [Hitchin component]}
\label{hitchin metric} 
The pressure form  on the Hitchin component is
an analytic Riemannian metric which 
is invariant under the mapping class group and restricts to the Weil-Pe\-ter\-sson metric on the Fuchsian locus.
\end{corollary}

The same naturally holds for Hitchin components of representations into $\ms{PSp}(n,\mathbb R)$, 
$\ms{S0}(n,n+1)$ and $\ms{G}_{2,0}$,  since they embed in Hitchin components of representations into
$\ms{PSL}(n,\mathbb R)$.  Labourie and Wentworth \cite{labourie-wentworth} have announced an explicit formula 
(in term of the Hitchin parametrisation) for the pressure metric along the Fuchsian locus.

Li \cite{li} has used the work of Loftin \cite{loftin} and Labourie \cite{labourie-cubic} to exhibit a metric on 
$\mathcal{H}_3(S)$, which she calls the Loftin metric,
which is invariant with respect to the mapping class group, restricts to a multiple of the Weil-Petersson metric on the Fuchsian
locus and such that the Fuchsian locus is totally geodesic. She further shows that a metric on $\mathcal{H}_3(S)$ constructed
earlier by  Darvishzadeh and Goldman \cite{goldman-darvishzadeh} restricts to a multiple of the Weil-Petersson metric on
the Fuchsian locus.
Kim and Zhang \cite{kim-zhang} introduced a mapping class group invariant 
K\"ahler metric on the Hitchin component $H_3(S)$ for 
$\mathsf{SL}(3,R)$, which Labourie \cite{labourie-cyclic} generalized to the
Hitchin components associated to all real split simple Lie groups of rank 2.

\medskip

If $\Gamma$ is a word hyperbolic group, we let ${\Cc}(\Gamma,\ms{PSL}_2(\mathbb C))$ denote the space of (conjugacy classes of)
convex cocompact representations of $\Gamma$ into $\ms{PSL}_2(\mathbb C)$.
In Section \ref{plucker-sec} we produce a representation, called the Pl\"ucker
representation,
\hbox{$\alpha:\ms{PSL}_2(\mathbb C)\to \sln$} (for some $m$), so that if $\rho\in{\Cc}(\Gamma,\ms{PSL}_2(\mathbb C))$,
then $\alpha\circ\rho$ is projective Anosov. 
The deformation space ${\Cc}(\Gamma,\ms{PSL}_2(\mathbb C))$ is an analytic manifold and 
we may define a renormalised intersection $\JJ$  and thus 
a pressure form on ${\Cc}(\Gamma,\ms{PSL}_2(\mathbb C))$.
The following corollary is a direct generalization of Bridgeman's pressure metric on quasifuchsian space
(see \cite{bridgeman-pressure}).

\begin{corollary}{\sc [Kleinian groups]}\label{psl2c}
Let $\Gamma$ be a torsion-free
word hyperbolic group. The pressure form gives rise to a ${\rm Out}(\Gamma)$-invariant metric on 
the analytic manifold ${\Cc}(\Gamma,\ms{PSL}_2(\mathbb C))$ which is Riemannian on the open subset
consisting of Zariski dense representations.
Moreover,
\begin{enumerate}
\item
If $\Gamma$ does not have a finite index subgroup which is
either a free group or a surface group, 
then the metric is Riemannian at all points in ${\Cc}(\Gamma,\ms{PSL}_2(\mathbb C))$.
\item
If $\Gamma$ is the fundamental group of a closed, connected,  orientable surface, then 
the metric is Riemannian off of the Fuchsian locus in ${\Cc}(\Gamma,\ms{PSL}_2(\mathbb C))$
and restricts to a multiple of the Weil-Petersson metric on the Fuchsian locus.
\end{enumerate}
\end{corollary}

If $\ms G$ is a rank one semi-simple Lie group, then work of Patterson \cite{patterson}, Sullivan \cite{sullivan}, Yue \cite{yue} and 
Corlette-Iozzi \cite{corlette-iozzi} shows that the entropy  of a convex
cocompact representation $\rho:\Gamma\to \ms G$ agrees with the Hausdorff dimension of the limit set of 
$\rho(\Gamma)$.
We may then apply Theorem \ref{analytic} and the Pl\"ucker representation to conclude that
that the Hausdorff dimension of the limit set varies analytically over analytic families of convex cocompact representations
into rank one semi-simple Lie groups.

\begin{corollary}
\label{rank 1}
{\sc[Analyticity of Hausdorff Dimension]}
If $\Gamma$ is a finitely generated group and $\mathsf G$ is a rank one semi-simple Lie group, then  
the Hausdorff dimension of the limit set
varies analytically on any analytic family of convex cocompact representations of $\Gamma$ into $\ms G$.
In particular, the Hausdorff dimension varies analytically over ${\Cc}(\Gamma,\ms{PSL}_2(\mathbb C))$
\end{corollary}

One may further generalize our construction into the setting of virtually Zariski dense Anosov representations into an
arbitrary semi-simple Lie group $\ms G$. A representation $\rho:\Gamma\to\ms G$ is {\em virtually Zariski dense}
if the Zariski closure of $\rho(\Gamma)$ is a finite index subgroup of $\ms G$.  
If $\Gamma$ is a word hyperbolic group, $\ms G$ is a semi-simple Lie group
with finite center and $\ms P$ is a non-degenerate parabolic subgroup, then we let $\mathcal{Z}(\Gamma;\ms G,\ms P)$
denote the space of (conjugacy classes of) regular virtually Zariski dense 
$(\ms G,\ms P)$-Anosov representations of $\Gamma$ into $\ms G$. 
The space $\mathcal{Z}(\Gamma;\ms G,\ms P)$ is an analytic orbifold, see
Proposition \ref{zd variety}, and
we can again use a Pl\"ucker representation to define a pressure metric on $\mathcal{Z}(\Gamma;\ms G,\ms P)$.
If $\ms G$ is connected, then $\mathcal{Z}(\Gamma;\ms G,\ms P)$ is an analytic manifold.

\begin{corollary}{\sc [Anosov representations]}
\label{metric on anosov} 
Suppose that $\Gamma$ is a word hyperbolic group, 
$\ms G$ is a semi-simple Lie group with finite center and $\ms P$ is a  non-degenerate parabolic subgroup of $\ms G$.
Then there exists an ${\rm Out}(\Gamma)$-invariant analytic Riemannian metric on
the orbifold $\mathcal{Z}(\Gamma;\ms G,\ms P)$.
\end{corollary}

A key tool in our proof is the introduction of a flow  $\Ug$ associated to a projective Anosov 
representation $\rho$.
Let $\rho:\Gamma\to\sln$ be a projective Anosov representation with limit maps $\xi$ and $\theta$.
Let $F$ be the total space of the principal $\mathbb R$-bundle over $\rpn\times\rpn^*$ whose fiber at
the point $(x,y)$ is the space of  norms on the line  $\xi(x)$. 
There is a natural $\mathbb R$-action on $F$ which takes a norm $u$ on $x$ to the norm $e^{-t}u$.
Let $F_\rho$ be $\mathbb R$-principal bundle  over 
$$\bg^{(2)} =\bg\times\bg\setminus\{(x,x)\mid x\in\bg\}.$$
which is the pull back of $F$ by $(\xi,\theta)$. The $\mathbb R$-action on $F$ gives rise to a flow on
$F_\rho$. 
 (An analogue of this flow was first introduced by Sambarino \cite{sambarino-quantitative,sambarino-thesis}
in the setting of projective Anosov irreducible representations of fundamental groups of closed negatively curved manifolds.)

We then show that this flow is metric Anosov and is a H\"older reparameterization of  the
{\em Gromov geodesic flow} $\Gg$ of $\Gamma$.
Moreover, this flow encodes the spectral radii of elements of $\rho(\Gamma)$, i.e.
the period of the flow associated
to (the conjugacy class of ) an element $\gamma\in\Gamma$ is $\log \Lambda(\gamma)(\rho)$.
(Metric Anosov flows are a natural
generalization of Anosov flows in the setting of compact metric spaces and were studied
by Pollicott \cite{pollicottcoding}.)

\begin{theorem}{\sc [geodesic flow]}
\label{geodesic flow}
The action of $\Gamma$ on $F_\rho$ is proper and cocompact. Moreover, the $\mathbb R$ action on 
$\Ug=F_\rho/\Gamma$ is a topologically transitive metric Anosov flow which is 
H\"older orbit equivalent to the geodesic flow $\Gg$. 
\end{theorem}

Theorem \ref{geodesic flow} allows us to make use 
of the thermodynamic formalism. We show that if $f_\rho$ is the H\"older function regulating the change of speed of $\Ug$ and $\Gg$, 
then $\Phi_\rho=-h(\rho) f_\rho$ is a pressure zero function on $\Gg$.
Therefore, we get a  mapping
$$\mathfrak{T}:\Cm\to {\mathcal{H}}(\Gg),$$
called the {\em thermodynamic mapping}, from $\Cm$ into the space
$\mathcal{H}(\Gg)$ of Liv\v sic cohomology classes of pressure zero H\"older functions on $\Gg$.
Given any \hbox{$[\rho]\in\Cm$}, there exists an open neighborhood $U$ of $[\rho]$ and a lift of $\mathfrak{T}|_U$ 
to an analytic map of $U$ into the space  ${\mathcal{P}}(\Gg)$ of pressure zero H\"older functions on $\Gg$. Our pressure form is
obtained as a pullback of the pressure 2-tensor on ${\mathcal{P}}(\Gg)$ with respect to this lift.

\medskip\noindent
{\bf Remarks and references: }
Anosov representations were introduced by Labourie \cite{labourie-anosov} in his study of
Hitchin representations, and their theory was further developed by Guichard and Wienhard  \cite{guichard-wienhard}.
Benoist \cite{benoist-divisible0,benoist-divisible1,benoist-divisible3} studied holonomy maps 
of  strictly convex projective structures on closed manifolds which he showed were irreducible
representations with transverse projective limit maps, hence projective Anosov. 
Sambarino \cite{sambarino-thesis,sambarino-quantitative,growth} introduced a flow, closely related to
our flow, associated to
a representation with transverse projective limit maps and used it to prove the continuity of the
associated entropy on a Hitchin component.
Pollicott and Sharp \cite{pollicott-sharp} applied the thermodynamic formalism and work of Dreyer \cite{dreyer}
to show that a closely related entropy gives rise to an analytic function on any Hitchin component.

Our metric generalizes Thurston's Riemannian metric on Teichm\"uller space which he defined to
be the Hessian of the length of a random geodesic. Wolpert \cite{wolpert} proved that Thurston's Riemannian
metric was a multiple of the more classical Weil-Petersson metric. Bonahon \cite{bonahon} gave an
interpretation of Thurston's metric in terms of the Hessian of an intersection function. 
Burger \cite{burger} previously studied the intersection number for convex cocompact subgroups of rank 1 simple Lie groups
and proved a strong version of Theorem \ref{intersection} in this setting (see also Kim \cite{kim-ergodic}).
The study of geometric properties of surfaces using the thermodynamic formalism originated in Bowen \cite{quasicircles}. 
Using a Bowen-Series coding and building on work of Bridgeman and Taylor \cite{bridgeman-taylor},
McMullen \cite{mcmullen} gave a pressure metric formulation of the Weil--Petersson metric on Teichm\"uller space. 
Bridgeman \cite{bridgeman-pressure} developed a pressure metric
on quasifuchsian space which restricts to the Weil--Petersson metric on the Fuchsian locus. 
Our Theorem \ref{path metric} is a natural generalization
of Bridgeman's work into the setting of projective Anosov representations, while Corollary \ref{psl2c} is a generalization
into the setting of general deformation spaces of convex cocompact representations into $\ms{PSL}_2(\mathbb C)$.

Corollary \ref{rank 1} was established by Ruelle \cite{ruelle-hd} for
quasifuchsian representations, {\it i.e.} when $\Gamma=\pi_1(S)$ and $\ms G=\mathsf{PSL}_2(\mathbb{C})$, 
and by Anderson and Rocha \cite{anderson-rocha} for function groups, {\it i.e.} when  
$\Gamma$ is a free product of surface groups and free groups and $\ms G=\mathsf{PSL}_2(\mathbb{C})$.
Previous work of Tapie \cite{tapie} implies that the Hausdorff dimension of the limit set is  a $C^1$ function on
$C^1$-families of convex cocompact representations of $\Gamma$ into a rank one Lie group $\ms G$. 
Tapie's work was inspired by work of Katok, Knieper, Pollicott and Weiss \cite{KKPW,KKW}
who established analytic variation of the entropy for analytically varying families of  Anosov flows on closed 
Riemannian manifolds.
Our Theorem
\ref{theorem:length} is related to the marked length spectrum rigidity theorem of Dal'Bo-Kim \cite{dalbo-kim}.

Coornaert--Papadopoulos \cite{coornaert-papadopoulos} showed  that if $\Gamma$ is word hyperbolic, then
there is a symbolic coding of its geodesic flow $\Gg$.
However, this coding is not necessarily one-to-one on a large enough set to apply the thermodynamic formalism.
Therefore, word hyperbolic groups admitting  projective  Anosov representations represent an interesting class 
of groups from the point of view of symbolic dynamics.

\medskip\noindent
{\bf Acknowledgements:}
We thank Bill Goldman,  Alex Lubotzky, Fran\c cois Ledrappier, Olivier Guichard, Fr\'ed\'eric Paulin, 
Jean-Fran\c cois Quint, Hans-Henrik Rugh,  Ralf Spatzier,
Matthew Stover and Amie Wilkinson for helpful discussions. We thank the referee for many useful
comments which improved the exposition.
This research was begun while the authors were participating in the program on {\em Geometry
and Analysis of Surface Group Representations} held at the Institut Henri Poincar\'e in Winter 2012.

\tableofcontents

\section{Anosov representations}

In this section, we recall the theory of  Anosov representations. 
We  begin by defining projective Anosov representations and developing their basic properties. 
In section \ref{plucker-sec}, we will see that any Anosov representation can be transformed, via post-composition
with a Pl\"ucker representation, into a projective Anosov representation, while in section \ref{irreducible reps}
we will study properties of irreducible projective Anosov representations.

\subsection{Projective Anosov representations}
\label{convex-anosov}

A  representation  $\rho:\Gamma\to\sln$ is projective Anosov if it has transverse projective limit maps
and  the associated flat bundle over its 
Gromov geodesic flow has a contraction property we will define carefully below.

\begin{definition}
Let $\Gamma$ be a  word hyperbolic group and $\rho$ be a representation of $\Gamma$ in $\sln$. 
We say $\rho$ has {\em transverse projective limit maps} if there exist $\rho$-equivariant continuous maps
\hbox{$\xi:\bg\to\rpn$} and 
\hbox{$\theta:\bg\to\rpn^*$} such that if $x\not=y$, then
$$
\xi(x)\oplus\theta(y)=\mathbb R^m.
$$ 

\end{definition}

\medskip\noindent
{\bf Conventions:} 
Denote by $\rpn$ the projective space of $\mathbb R^m.$  We will often identify $\rpn^*$ with the Grassmannian 
${\rm Gr}_{m-1}(\mathbb R^m)$ of $(m-1)$-dimensional subspaces of $\mathbb R^m,$ 
via $\varphi\mapsto\ker\varphi.$ The action of $\sln$ on $\rpn^*$ consistent wth this identification is 
$$g\cdot\varphi=\varphi\circ g^{-1}.$$

We will also assume throughout this paper that our word hyperbolic group does not
have a finite index cyclic subgroup. 
Since all the word hyperbolic groups we study are linear,  Selberg's Lemma
implies that they contain finite index torsion-free subgroups.

\medskip

Gromov \cite{gromov} defined a geodesic flow $\Gg$ for a word hyperbolic group -- that we shall call the {\em Gromov geodesic flow} -- (see Champetier \cite{champetier} and Mineyev \cite{mineyev}
for details). He defines a proper cocompact action of $\Gamma$ on
$\bg^{(2)}\times\mathbb R$ which commutes with the action of $\mathbb R$ by translation on the final factor.
The action of $\Gamma$ restricted to
$\bg^{(2)}$ is the diagonal action arising from the standard action of $\Gamma$ on $\partial_\infty\Gamma$.
There is a metric on $\bg^{(2)}\times\mathbb R$, well-defined up to H\"older equivalence, so that $\Gamma$ acts by isometries, every orbit of the $\mathbb R$ action gives a quasi-isometric embedding and the geodesic flow acts by Lipschitz homeomorphisms.
The flow on 
$$\widetilde{\Gg}=\bg^{(2)}\times\mathbb R$$
descends to a flow on the quotient
$$\Gg=\bg^{(2)}\times\mathbb R/\Gamma.$$
In the case that $M$ is a closed negatively curved manifold and $\Gamma=\pi_1(M)$,  $\Gg$ may be identified with $\ms T^1M$ in such a way that the flow on $\Gg$ is identified with the geodesic flow on $\ms T^1M$. Since the action of $\Gamma$ on $\bg^2$ is topologically transitive, the Gromov geodesic flow is topologically transitive.

If $\rho:\Gamma\to\sln$ is a representation, we let $E_\rho$ be the associated flat bundle over the geodesic 
flow of the word hyperbolic group $\Gg$.
Recall that 
$$E_\rho=\widetilde{\Gg}\times \mathbb R^m/\Gamma$$
where the action of $\gamma\in\Gamma$ on $\mathbb R^m$ is given
by $\rho(\gamma)$. If $\rho$ has transverse projective limit maps
$\xi$ and $\theta$, there is an induced splitting of $E_\rho$  as 
$$
E_\rho=\Xi\oplus\Theta
$$ 
where $\Xi$ and $\Theta$ are  sub-bundles, parallel along the geodesic flow,  of rank  $1$ and $m-1$ respectively.
Explicitly, if we lift $\Xi$ and $\Theta$ to sub-bundles $\tilde\Xi$ and $\tilde\Theta$ of the bundle
$\widetilde{\Gg}\times \mathbb R^m$ over $\widetilde{\Gg}$, then the
fiber of $\tilde\Xi$ above $(x,y,t)$ is simply $\xi(x)$ and the fiber of $\tilde\Theta$ is $\theta(y)$.

The $\mathbb R$-action on $\widetilde{\Gg}$ extends to a flow $\{\tilde \psi_t\}_{t\in\mathbb R}$ on
\hbox{$\widetilde{\Gg}\times \mathbb R^m$} (which acts trivially on the $\mathbb R^m$ factor).
The flow $\{\tilde \psi_t\}_{t\in\mathbb R}$ descends to a flow 
$\{\psi_t\}_{t\in\mathbb R}$ on $E_\rho$ which is a lift of the geodesic flow 
on $\Gg$. In particular, the flow respects the splitting
$E_\rho=\Xi\oplus\Theta$. 

In general, we say that a vector bundle $E$ over a compact topological space whose total space is equipped with
a flow $\{\phi_t\}_{t\in\mathbb R}$ of bundle automorphisms is {\em contracted} by the flow if for any metric 
$\Vert.\Vert$ on $E$, there exists $t_0>0$ such that if $v\in E$, then
$$
\Vert \phi_{t_0} (v)\Vert\leq\frac{1}{2}\Vert v\Vert.
$$
Observe that if bundle is contracted by a flow, its dual is contracted by the inverse flow.
Moreover, if the flow is contracting, it is also {\em uniformly contracting}, i.e. given any metric,
there exists  positive constants $A$ and $c$ such that
$$\Vert \phi_{t} (v)\Vert\leq A e^{-ct}\Vert v\Vert$$
for any $v\in E$.

\begin{definition}\label{def:convano}
A representation $\rho:\Gamma\to\sln$ with transverse projective limit maps is {\em projective Anosov} if  the 
bundle $\operatorname{Hom}(\Theta,\Xi)$ is contracted
by the flow $\{\psi_t\}_{t\in\mathbb R}$.
\end{definition}

In the sequel, we will use the notation $\Theta^*=\operatorname{Hom}(\Theta,\mathbb R)$.  The following alternative description will be useful.

\begin{proposition}
A representation $\rho:\Gamma\to\sln$  with transverse projective limit maps $\xi$ and $\theta$ is projective Anosov
if and only if  there exists
$t_0>0$ such that for all $Z\in \Gg$, $v\in \Xi_Z\setminus\{0\}$ and $w\in\Theta_Z\setminus\{0\}$,
\begin{equation}
\label{anosovineq}
\frac{\Vert\psi_{t_0}(v)\Vert}{\Vert\psi_{t_0}(w)\Vert}\le\frac{1}{2}\frac{\Vert v\Vert}{\Vert w\Vert}.
\end{equation}
\end{proposition}

\begin{proof}
Given  a projective Anosov representation $\rho:\Gamma\to\sln$ and a metric $\Vert.\Vert$ on $E_\rho$, 
let $t_0>0$ be chosen so that  
$$
\Vert \psi_{t_0} (\eta)\Vert\leq\frac{1}{2}\Vert \eta\Vert.
$$
for all $\eta\in\Xi\otimes\Theta^*$. If $Z\in \Gg$, $v\in \Xi_Z\setminus\{0\}$ and $w\in\Theta_Z\setminus\{0\}$, then there
exists $\eta\in \operatorname{Hom}(\Theta_Z,\Xi_Z)=(\Xi\otimes\Theta^*)_Z$ such that
$\eta(w)=v$ and $\Vert\eta\Vert=\Vert v\Vert/\Vert w \Vert$. Then,
\begin{equation*}
\frac{\Vert\psi_{t_0}(v)\Vert}{\Vert\psi_{t_0}(w)\Vert}\le\Vert\psi_{t_0}(\eta)\Vert\leq \frac{1}{2}\Vert\eta\Vert=
\frac{\Vert v\Vert}{\Vert w\Vert}.
\end{equation*}
The converse is immediate. \end{proof}

Furthermore, projective Anosov representations are contracting on $\Xi$.

\begin{lemma}
\label{contracting on Xi}
If $\rho:\Gamma\to \sln$ is projective Anosov,
then $\{\psi_t\}_{t\in\mathbb R}$ is contracting on $\Xi$.
\end{lemma}

\begin{proof}  
Since the bundle $\Xi\otimes\Theta^*$ is contracted, so is 
$$
\Omega=\det(\Xi\otimes\Theta^*)=\Xi^{\otimes(m-1)}\otimes\det(\Theta^*).
$$ 
One may define an isomorphism from $\Xi$ to $\det(\Theta)^*$ by taking $u$ to the map
$\alpha\to {\rm Vol}(u\wedge\alpha)$. Since $\det(\Theta)^*$ is isomorphic to
$\det(\Theta^*)$,
it follows that $\Omega$ is isomorphic to $\Xi^{\otimes m}$. Thus $\Xi$ is contracted.\end{proof}

It follows from standard  techniques in hyperbolic dynamics that our limit maps are H\"older.
We will give a proof of a more general statement in Section \ref{anvar} (see 
\cite[Proposition 3.2]{labourie-anosov} for a proof in a special case).

\begin{lemma}
\label{anosov-holder} 
Let $\rho$ be a projective Anosov representation,  then the limit maps $\xi$ and $\theta$ are H\"older.
\end{lemma}

If $\gamma$ is an infinite order element of $\Gamma$, then there is a periodic orbit of $\Gg$ associated to $\gamma$.
If $\gamma^+$ is the attracting fixed point of $\gamma$ on $\bg$ and $\gamma^-$ is
its other fixed point, then this periodic orbit is the  image of $(\gamma^+,\gamma^-)\times\mathbb R$.
Inequality (\ref{anosovineq}) and Lemma \ref{contracting on Xi} 
applied to the periodic orbit of $\Gg$ associated to $\gamma$ imply that 
$\rho(\gamma)$ is proximal and that $\xi(\gamma^+)$ is the eigenspace associated to the largest
modulus eigenvalue of $\rho(\gamma)$. Similarly, $\xi(\gamma^-)$ is the repelling hyperplane of  $\rho(\gamma)$.
It follows that the limit maps $\xi$ and $\theta$ are uniquely determined by $\rho$
(see also \cite[Lemmas 3.1 and 3.3]{guichard-wienhard}).

Let  $\LL(\gamma)(\rho)$ denote the eigenvalue of $\rho(\gamma)$ of maximal absolute value 
and let $\Lambda(\gamma)(\rho)$ denote the spectral radius of  $\rho(\gamma)$,
so $\Lambda(\gamma)(\rho)=|\LL(\gamma)(\rho)|$.
If $S$ is a fixed  generating set for $\Gamma$ and $\gamma\in \Gamma$, then we let $l(\gamma)$ 
denote the translation length of
the action of $\gamma$ on the Cayley graph of $\Gamma$ with respect to $S$; more explicitly, $l(\gamma)$ is the minimal word
length of any element conjugate to $\gamma$.
Since the contraction is uniform and the length of the periodic orbit of $\Gg$ associated to $\gamma$ 
is comparable to $l(\gamma)$,
we obtain the following uniform estimates:

\begin{proposition}\label{proxi2}
If $\rho:\Gamma\to\sln$ is a projective Anosov representation, 
then there exists $\delta\in (0,1)$ such that if $\gamma\in\Gamma$
has infinite order, then
$\LL(\gamma)(\rho)$ and $(\ms L(\gamma^{-1})(\rho))^{-1}$ are both eigenvalues of $\rho(\gamma)$ of multiplicity one and
$$
\rho(\gamma)=\ms L(\gamma)(\rho)\mathsf p_\gamma+ \mathsf m_\gamma +\frac{1}{\LL(\gamma^{-1})(\rho)}\mathsf q_\gamma
$$
where  \begin{itemize}
\item $\mathsf p_\gamma$ is the projection on $\xi(\gamma^+)$ parallel to $\theta(\gamma^-)$,
\item $\mathsf q_\gamma= \mathsf p_{\gamma^{-1}}$,
\item $\mathsf m_\gamma=A\circ (1-\mathsf q_\gamma- \mathsf p_{\gamma})$ and
$A$ is an endomorphism of $\theta(\gamma^-)\cap\theta(\gamma+)$ whose spectral radius is less than
$$\delta^{\ell(\gamma)} \Lambda(\gamma)(\rho).$$
\end{itemize}
\end{proposition}

Moreover, we see that $\rho$ is well-displacing in the following sense:

\begin{proposition}{\sc [Displacing property]}\label{displac}
If $\rho:\Gamma\to\sln$ is a projective Anosov representation, then there exists constants $K>0$ and $C>0$, and a neighborhood $U$ of $\rho_0$ in 
${\rm Hom}(\Gamma,\sln)$ such that that for every $\gamma\in\Gamma$ and $\rho\in U$  we have
\begin{eqnarray}
\frac{1}{K} \ell(\gamma)-C\leq &\log(\Lambda(\gamma)(\rho))&\leq K \ell(\gamma)+C,
\end{eqnarray}
\end{proposition}

Proposition \ref{displac} immmediately implies:

\begin{proposition}\label{finiteclass}
For every real number $T$, the set
$$
R_T(\rho)=\{[\gamma]\mid \log(\Lambda(\gamma)(\rho))\leq T\}
$$
is finite.
\end{proposition}

\medskip\noindent
{\bf Remark:} Proposition \ref{proxi2} is a generalization of results of Labourie \cite[Proposition 3.4]{labourie-anosov},
Sambarino \cite[Lemma 5.1]{sambarino-quantitative} and
Guichard-Wien\-hard \cite[Lemma 3.1]{guichard-wienhard}. Proposition \ref{displac} is a generalization of
a result of Labourie \cite[Theorem 1.0.1]{labourie-ens} and a special case of a result of
Guichard-Wienhard \cite[Theorem 5.14]{guichard-wienhard}. See \cite{delzant} for a discussion of well-displacing
representations and their relationship with quasi-isometric embeddings.

\subsection{Anosov representations}
\label{anosov-rep}

We now recall the general definition of an Anosov representation and note that projective Anosov
representations are examples of Anosov representations.

We first recall some notation and definitions.
Let $\mathsf G$ be a semi-simple Lie group with finite center  and Lie algebra $\mk g$. Let $\ms K$ be a maximal compact subgroup of $\sf G$ and let $\tau$ be the Cartan involution on $\mk g$ whose fixed point set is the Lie algebra of $ \ms K.$ Let $\mk a=\mk a_{{\ms G}}$ be a maximal abelian subspace contained in $\{v\in\frak g: \tau v=-v\}.$

For $a\in\mk a$, let $\ms M$ be the connected component of the centralizer of $\exp a$ which contains the identity,
and let $\mk m$ denote its Lie algebra. Let $E_\lambda$ be the eigenspace of the action of $a$ on $\mk g$ with eigenvalue $\lambda$ and consider \begin{eqnarray*}
\mk n^+&=&\bigoplus_{\lambda>0}E_\lambda, \cr
\mk n^-&=&\bigoplus_{\lambda<0}E_\lambda, 
\end{eqnarray*}
so that
\begin{eqnarray}
\mk g=\mk m\oplus\mk n^+\oplus \mk n^-.\label{split}
\end{eqnarray}
Then $\mk n^+$ and $\mk n^-$ are Lie algebras normalized by $\ms M$. 
Let $\ms P^\pm$ the connected Lie subgroups of $\mathsf G$ whose Lie algebras are $\mk p^\pm=\mk m\oplus \mk n^\pm$. 
Then $\ms P^+$ and $\ms P^-$ are {\em opposite parabolic subgroups}.
We will say that $\ms P^+$ is {\em non-degenerate} if
$\mk p^+$ does not contain a simple factor of $\mk g$.

We may identify a point  $([X],[Y])$ in $\ms G/\ms P^+\times \ms G/\ms P^-$ with the pair  
$(\ad(X)\ms P^+,\ad(Y)\ms P^-)$ of parabolic subgroups. 
The pair $(\ad(X)\ms P^+,\ad(Y)\ms P^-)$ is {\em transverse} if
their intersection $\ad(X)\ms P^+\cap\ad(Y)\ms P^-$ is conjugate to $\ms M$.

We now suppose that $\rho:\Gamma \to \ms G$ is a representation of  word hyperbolic group $\Gamma$ and 
$\xi^+:\bg \to\ms G/\ms P^+$ and $\xi^-:\Gamma\to \ms G/\ms P^-$  are continuous
$\rho$-equivariant maps. We say that $\xi^+$ and $\xi^-$ are {\em transverse} if given any two distinct points $x,y\in\bg$,
$\xi^+(x)$ and $\xi^-(y)$ are transverse.
The $\ms G$-invariant splitting described by Equation \eqref{split} then gives rise to bundles  over $\Gg$. 
Let $\widetilde{\mathcal N}_\rho^+$ and $\widetilde{\mathcal N}_\rho^-$ be the bundles over $\widetilde{\Gg}$ whose fibers
over the point $(x,y,t)$ are 
$$
\ad(\xi^-(y))\mk n^+\ \ {\rm and}\ \ \ad(\xi^+(x))\mk n^-.
$$
There is a natural action of $\Gamma$ on $\widetilde{\mathcal N}_\rho^+$ and $\widetilde{\mathcal N}_\rho^-$,
where the action on the fiber is given by $\rho(\Gamma)$, and we denote the quotient bundles over $\Gg$ by
$\mathcal N_\rho^+$ and $\mathcal N_\rho^-$.  We may lift the geodesic flow to a flow on
the bundles $\mathcal N_\rho^+$ and $\mathcal N_\rho^-$ which acts trivially on the fibers.

\begin{definition}
Suppose that $\ms G$ is a semi-simple Lie group with finite center, $\ms P^+$ is a parabolic
subgroup of $\ms G$ and $\Gamma$ is a word hyperbolic group.
A representation $\rho:\Gamma\to \ms G$
is {\em  $(\ms{G},\ms P^+)$-Anosov} if there exist transverse
$\rho$-equivariant maps $$\xi^+:\bg \to \ms G/\ms P^+\textrm{ and }\xi^-:\bg \to \ms G/\ms P^-$$  so that the geodesic flow 
is contracting on the associated bundle $\mathcal N_\rho^+$ and the inverse flow is contracting on the bundle 
$\mathcal N_\rho^-$.
\end{definition}

We now recall some basic properties of Anosov representations which were established
by Labourie, \cite[Proposition 3.4]{labourie-anosov} and
\cite[Theorem 6.1.3]{labourie-ens}, and Guichard-Wienhard \cite[Theorem 5.3 and Lemma 3.1]{guichard-wienhard}.
We recall that an element $g\in\ms G$ is {\em proximal} relative to $\ms P^+$ if $g$ has fixed points
$x^+\in \ms G/\ms P^+$ and $x^-\in\ms G/\ms P^-$ so that $x^+$ is transverse to $x^-$ and if
$x\in \ms G/\ms P^+$ is transverse to $x^-$ then $\lim_{n\to\infty} g^n(x)=x^+$.

\begin{theorem}
\label{anosov properties}
Let $\ms G$ be a semi-simple Lie group, $\ms P^+$ a parabolic subgroup, $\Gamma$ a word
hyperbolic group and $\rho:\Gamma\to \ms G$ a $(\ms G,\ms P^+)$-Anosov representation.
\begin{enumerate}
\item
$\rho$ has finite kernel, so $\Gamma$ is virtually torsion-free.
\item
$\rho$ is well-displacing, so $\rho(\Gamma)$ is discrete.
\item
If $\gamma\in\Gamma$ has infinite order, then $\rho(\gamma)$ is proximal relative to $\ms P^+$
\end{enumerate}
\end{theorem}

In this language, projective Anosov representations are exactly the same as 
$(\sln, \mathsf P^+)$-Anosov  representations where $\mathsf P^+$ is the stabilizer of a line in $\mathbb R^m$.

\begin{proposition}
Let $\ms P^+$ be the stabilizer of a line in $\mathbb R^m$. A representation 
$\rho:\Gamma\to \sln$ is projective Anosov if and only if it is $(\sln, \mathsf P^+)$-Anosov.
Moreover, 
the limit maps $\xi$ and $\theta$ in the definition of projective Anosov representation agree with the limit maps $\xi^+$ and
$\xi^-$ in the definition of a
$(\sln, \mathsf P^+)$-Anosov  representation.  
\end{proposition}

\begin{proof} 
If $\rho$ is projective Anosov with limit maps $\xi$ and $\theta$,
one may identify $\sln/P^+$ with $\rpn$
and $\sln/P^-$ with $\rpn^*$ so that,
after letting $\xi^+=\xi$ and $\xi^-=\theta$,
$N_\rho^+$ is identified with ${\rm Hom}(\Theta,\Xi)$ and $N_\rho^-$ is identified with
${\rm Hom}(\Xi,\Theta)$. 

The same identification holds if $\rho$ is $(\sln, \mathsf P^+)$-Anosov with limit maps $\xi^+$ and $\xi^-$.
\end{proof}

\subsection{Pl\"ucker representations} 
\label{plucker-sec}

Guichard and Wienhard \cite{guichard-wienhard} showed how to obtain
a projective Anosov representation from any Anosov representation by post-composing with a Pl\"ucker representation.
We first recall the following general result.

\begin{theorem}\label{theorem:convexAnosov} {\sc [Guichard-Wienhard \cite[Prop. 4.3]{guichard-wienhard}]}
Let $\phi:\ms G\to\ms{SL}(V)$ be a finite dimensional irreducible representation. Let $x\in\mathbb P(V)$ and 
assume that  $$\ms P=\{g\in\ms G:\phi( g)(x)=x\}$$
is  a parabolic subgroup of $\ms G$ with opposite parabolic $\ms Q$.
If $\Gamma$ is a word hyperbolic group, then a representation $\rho:\Gamma\to\ms G$ is $(\ms G,\ms P)$-Anosov if and
only if $\phi\circ \rho$ is projective Anosov.

Furthermore, if $\rho$ is $(\ms G,\ms P)$-Anosov with limit maps
$\xi^+$ and $\xi^-$, then the limit maps of $\phi\circ\rho$ are given
by $\xi=\beta\circ\xi^+$ and
$\theta=\beta^*\circ\xi^-$ where $\beta:\ms G/\ms P\to\mathbb P(V)$ and $\beta^*:\ms G/\ms Q\to\mathbb P(V^*)$ 
are the maps induced by $\phi.$
\end{theorem}

The following corollary is observed by Guichard-Wienhard \cite[Remark 4.12]{guichard-wienhard}.
We provide a proof here for the reader's convenience.
The representation given in the proof will be called the \emph{Pl\"ucker representation}
of $\ms G$ with respect to $\ms P .$

\begin{corollary}{\sc [Guichard-Wienhard]}
\label{plucker} For any parabolic subgroup $\ms P$ of a semi-simple Lie group $\ms G$
with finite center, there exists a finite dimensional
irreducible representation \hbox{$\alpha:\ms G\to\ms{SL}(V)$} such that if $\Gamma$ is a word hyperbolic group
and \hbox{$\rho:\Gamma \to\ms G$} is a  $(\ms G,\ms P)$-Anosov representation,
then  $\alpha\circ\rho$ is projective Anosov.  

Moreover, if $\ms P$ is non-degenerate, then $\ker(\alpha)=Z(\ms G)$ and
$\alpha$ is an immersion.

\end{corollary}

\begin{proof} 
In view of Theorem \ref{theorem:convexAnosov} it suffices to find a finite dimensional irreducible representation 
$\alpha:\ms G\to\ms{SL}(V)$ such that $\alpha(\ms P )$ is the stabilizer (in $\alpha(\ms G)$) of a line in $V.$ 

Let $\Lambda^kW$ denote the $k$-th exterior power of the vector space $W.$  
Let $n=\dim \mk n^+=\dim \mk n^-$ and consider $\alpha:\ms G\to\ms{SL}(\Lambda^n\mk g)$ 
given by 
$$\alpha(g)=\Lambda^n\ad(g).$$ 
One may readily check that the restriction of $\alpha$ to  $V=\braket{\alpha(\ms G) \cdot \Lambda^n\mk n^+}$ works.

If $\ms P $ is non-degenerate, then $\ker  (\alpha|_V)$ is a normal subgroup  of $\ms G$ which is contained in $\ms P$,
so $\ker  (\alpha|_V)$ is contained in $Z(\ms G)$ (see \cite{razovin}).
Since $Z(\ms G)$ is in the kernel of the adjoint representation, we see that  
$\ker  (\alpha|_V)=Z(\ms G)$. Since $\alpha|_V$ is algebraic and $Z(\ms G)$ is finite, 
it follows that $\alpha|_V$ is an immersion.
\end{proof}

If $\ms G$ has rank one, then it contains a unique conjugacy class of parabolic subgroups. 
A representation $\rho:\Gamma\to \ms G$ is Anosov if and only if it is convex 
cocompact (see \cite[Theorem 5.15]{guichard-wienhard}). We then get the following.

We recall that the topological entropy of a convex cocompact representation $\rho:\Gamma\to\ms G$
of a word hyperbolic group into a rank one semi-simple Lie group is given by
$$h(\rho)=\lim_{T\to\infty}\frac{1}{T}\log\left( \sharp\{[\gamma]\ |\d(\rho(\gamma))\leq T\}\right),$$ where 
$\d(\rho(\gamma))$ denotes the translation length of $\rho(\gamma).$ We obtain the following immediate corollary.

\begin{corollary}\label{entropy:rk1}
Let $\ms G$ be a rank one semi-simple  Lie group, let $\Gamma$ be a word hyperbolic group and
let $\alpha:\ms G\to \ms{SL}(V)$ be the Pl\"ucker representation. There exists $K>0$, such that if
$\rho:\Gamma\to \ms G$ is convex cocompact, then $\alpha\circ\rho$ is projective Anosov and 
$$h(\alpha\circ\rho)=\frac{h(\rho)}{K}.$$
\end{corollary}

\begin{proof} Let $\lambda_{\ms G}:\ms G\to \mathfrak a_{\ms G}$ be the Jordan projection of $\ms G.$
Since $\mathfrak a_{\ms G}$ is one dimensional, we can identify it with $\Real$ by setting $\lambda_{\ms G}(g)=\d(g).$

Denote by $\chi_\alpha\in\mathfrak a_{\ms G}$ the highest (restricted) weight of the representation $\alpha$
(see, for example, Humphreys \cite{humphreys}). By definition, 
one has $\Lambda(\alpha(g))=\chi_\alpha(\d(g)),$ for every $g\in \ms G.$ 
Hence, since $\mathfrak a_{\ms G}$ is one dimensional, one has
\begin{equation}\label{pro:rk1}
\Lambda(\alpha(\rho(\gamma)))=K\d(\rho(\gamma))
\end{equation} 
for every $\gamma \in\Gamma.$ 

It follows immediately that
$$h(\alpha\circ\rho)=\frac{h(\rho)}{K}.$$
\end{proof}

\subsection{Irreducible representations}
\label{irreducible reps}

Guichard and Wienhard \cite[Proposition 4.10]{guichard-wienhard} proved that
irreducible representations  with transverse projective limit maps are projective Anosov
(see also \cite{labourie-anosov} for hyperconvex representations).

\begin{proposition}\label{irredano}{\sc[Guichard--Wienhard]}
If $\Gamma$ is a word hyperbolic group, then
every irreducible  representation $\rho:\Gamma\to\sln$ with transverse projective limit maps
is projective Anosov.
\end{proposition}

It will be useful to note that if $\rho:\Gamma\to\sln$ is projective Anosov and irreducible, then 
$\xi(\bg)$ contains a projective frame for $\rpn$. We recall that a collection of $m+1$ elements in $\rpn$ is a
{\em projective frame} if every subset containing $m$ elements spans $\mathbb R^m$. We first prove the following lemma.

\begin{lemma} Let $\rho:\Gamma\to\sln$ be a representation with a continuous \hbox{$\rho$-equivariant} map 
$\xi:\bg\to\rpn,$ then the preimage $\xi^{-1}(V)$ of a vector subspace $V\subset \Real^m$ is either 
$\bg$ or has empty interior on $\bg$.
\end{lemma}

\begin{proof} Choose $\{x_1,\ldots,x_p\}\subset\bg$ so that $\{\xi(x_1),\ldots,\xi(x_p)\}$ spans 
the vector subspace $\braket{\xi(\bg)}$  spanned by $\xi(\bg)$.

Suppose that  $\xi^{-1}(V)=\{x\in\bg:\xi(x)\in V\}$ has non-empty interior in $\bg$.
Choose $\gamma\in\Gamma$ so that $\gamma^-\notin\{x_1,\ldots,x_p\}$ and $\gamma^+$ belongs to the interior of $\xi^{-1}(V).$ 

Since $ \gamma^n(x_i)\to\gamma^+$ for every $i\in\{1,\ldots,p\},$ if we choose $n$ large enough,
then $\gamma^n(x_i)$ is contained in the interior of $\xi^{-1}(V)$, so
$\xi(\gamma^n x_i)\in V.$ Since $\{\xi(\gamma^n (x_1)),\ldots,\xi(\gamma^n(x_p))\}$ still spans $\braket{\xi(\bg)}$,
we see that \hbox{$\braket{\xi(\bg)}\subset V$}, in which case $\xi^{-1}(V)=\bg$.
\end{proof}

The following generalization of the fact that every irreducible projective Anosov representation admits a projective frame will
be useful in Section \ref{rigidity}.

\begin{lemma}
\label{mutual frame}
Let $\rho_1:\Gamma\to\sln$ and $\rho_2:\Gamma\to\sln$ be representations with  continuous equivariant limit
maps $\xi_1$ and $\xi_2$ such that
$\dim\braket{\xi_1(\bg)}=\dim\braket{\xi_2(\bg)}=p.$ Then there exist $p+1$ distinct points
$\{x_0,\ldots,x_p\}$ in $\bg$ such that $$\{\xi_1(x_0),\ldots,\xi_1(x_p)\}\textrm{ and }\{\xi_2(x_0),\ldots,\xi_2(x_p)\}$$ are projective frames of $\braket{\xi_1(\bg)}$ and $\braket{\xi_2(\bg)}$ respectively.
\end{lemma}

\begin{proof}
We first proceed by iteration to produce $\{x_1,\ldots,x_p\}$ so that
$\{\xi_1(x_1),\ldots,\xi_1(x_p)\}$ and $\{\xi_2(x_1),\ldots,\xi_2(x_p)\}$  generate
$$V=\braket{\xi_1(\bg)}\ \ \textrm{and}\ \ W=\braket{\xi_2(\bg)}.$$
Assume we have found $\{x_1,\ldots,x_k\}$ so that $\{\xi_1(x_1),\ldots,\xi_1(x_k)\}$ and $\{\xi_2(x_1),\ldots,\xi_2(x_k)\}$  are both
linearly independent. Define 
$$V_k=\braket{\{\xi_1(x_1),\ldots,\xi_1(x_k)\}}\textrm{ and }W_k=\braket{\{\xi_2(x_1),\ldots,\xi_2(x_k)\}}.$$
By the previous lemma,  if $k< p$, then $\xi_1^{-1}(V_k)$ and $\xi_2^{-1}(W_k)$ have empty interior, so
their complements must intersect. 
Pick $$x_{k+1}\in\xi_1^{-1}(V_k)^c\cap\xi_2^{-1}(W_k)^c.$$ 
This process is complete when $k=p$.

It remains to find $x_0$. For each $i=1,\ldots,p$, let
$$U_i^1=\braket{\{\xi_1(x_1),\ldots,\xi_1(x_p)\}\setminus\{\xi_1(x_i)\}}$$  and $$U^2_i=\braket{\{\xi_2(x_1),\ldots,\xi_2(x_p)\}\setminus\{\xi_2(x_i)\}}.$$ 
Then, choose
$$x_0\in\bigcap_i \xi_1^{-1}(U^1_i)^c\cap\xi_2^{-1}(U^2_i)^c.$$
One easily sees that $\{x_0,\ldots,x_p\}$ has the claimed properties.
\end{proof}

If $\rho:\Gamma\to \sln$ is projective Anosov and irreducible, then $\braket{\xi(\bg)}=\mathbb R^m$ (since $\braket{\xi(\bg)}$ is $\rho(\Gamma)$-invariant),
so Lemma \ref{mutual frame} immediately gives:

\begin{lemma}
\label{projective frame}
If $\rho:\Gamma\to\sln$ is an irreducible projective Anosov representation with limit maps $\xi$ and $\theta$, 
then then there exist 
$\{x_0,\ldots,x_m\}\subset \bg$ so that $\{\xi(x_0),\ldots,\xi(x_m)\}$ is a
projective frame for $\rpn$.
\end{lemma}

We will  also need the following lemma which was explained to us by J.-F. Quint.

\begin{lemma}{\sc[Quint]}\label{compactfactors} 
If $\Delta$ is an irreducible subgroup of $\sln$ that contains a proximal element,
then the Zariski closure $\ms G$ of $\Delta$ is a semi-simple Lie group without compact factors whose center
$Z(\ms G)\subset\{\pm I\}.$
\end{lemma}

\begin{proof}  Since $\ms G$ acts irreducibly on $\mathbb R^m$, it is a reductive group. 
Moreover, since $\ms G$ contains a proximal matrix, one easily sees
that attracting lines of proximal matrices in $\ms G$ span $\Real^m$, and that each attracting line
of a proximal matrix in $\ms G$ is invariant under $Z(\ms G).$
Therefore, $Z(\ms G)\subset\{\pm I\}$, so $\ms G$ is a semi-simple Lie group.

Let $\ms K$ be the maximal normal connected compact subgroup of $\ms G,$ 
and let $\ms H$ be the product of the non-compact Zariski connected, simple factors of $\ms G.$ 
Then $\ms H$ and $\ms K$ commute and $\ms H\ms K$ has finite index in $\ms G.$

Consider now a proximal element $g\in \ms G.$ Replacing $g$ by a large enough power, we can assume that 
$g=hk$ for some $h\in \ms H$ and $k\in \ms K.$ Since all eigenvalues of $k$ have modulus $1$
and $k$ and $h$ commute, we conclude that $h$ is proximal. So we can assume that $g\in \ms H.$ 

Since $g$ and $\ms K$ commute, the attracting line of $g$ is fixed by $\ms K,$ 
and, since $\ms K$ is connected, each vector of this attracting line is fixed by $\ms K.$ 
Let $W$ be the vector space of $\ms K$-fixed vectors on $\Real^m,$ then $W$ is $\ms G$-invariant ,
since $\ms K$ is normal in $\ms G$, and nonzero. Since $\ms G$ is irreducible, $W=\Real^m$ and so $\ms K=\{I\}.$
\end{proof}

Proposition \ref{proxi2} and Lemma \ref{compactfactors} together have the following immediate consequence.

\begin{corollary}\label{Zclosureconvex} Let $\rho:\Gamma\to\sln$ be an  irreducible projective Anosov representation, 
then the Zariski closure $\ms G_\rho$ of $\rho(\Gamma)$ is a semi-simple Lie group without compact factors 
such that $Z(\ms G_\rho)\subset\{\pm I\}.$
\end{corollary}

\subsection{$\mathsf G$-generic representations}
\label{gg}

Let $\mathsf G$ be a reductive subgroup of $\sln$. We recall that an element in $\mathsf G$ is {\em generic}
if its centralizer is a maximal torus in $\mathsf G$. 
We say that a representation $\rho:\Gamma \to\sln$ of $\Gamma$ is {\em $\mathsf G$-generic} if  
$\rho(\Gamma)\subset \ms G$ and the Zariski closure $\overline{\rho(\Gamma)}^Z$ of $\rho(\Gamma)$ contains a 
$\mathsf G$-generic element.

We will need the following observation.

\begin{lemma}
\label{pro:gen}
If $\ms G$ is a reductive subgroup of $\sln$ and
$\rho:\Gamma\to\ms G$ is a $\mathsf G$-generic representation, then 
there exists $\gamma\in\Gamma$ such  that $\rho(\gamma)$ is  a generic element of $\ms G$.
\end{lemma}

\begin{proof}
We first note that the set of non-generic elements of $\ms G$ is Zariski closed in $\ms G$, so the set of generic elements
is Zariski open in $\ms G$. Therefore, if the Zariski closure of $\rho(\Gamma)$ contains generic elements of $\ms G$,  then
$\rho(\Gamma)$ must itself contain generic elements of $\ms G$.
\end{proof}

\section{Thermodynamic  formalism}\label{PR}

In this section, we recall facts from the thermodynamic formalism, as developed by
developed by Bowen \cite{bowen1,bowen2}, Parry--Pollicott \cite{parry-pollicott}, Ruelle  \cite{ruelle}
and others, which we will need in our work. In section \ref{pressure form}, we will describe a variation
of  a construction of McMullen \cite{mcmullen}, which produces a pressure form on
the space of pressure zero functions on a flow space. Our pressure metric will be a pull-back of
this form.

\subsection{H\"older flows on compact spaces}

Let $X$ be a compact metric space with a H\"older continuous flow $\phi=\{\phi_t\}_{t\in\Real}$ without fixed points.

\subsubsection{Flows and parametrisations}

Let $f:X\to\Real$ be a positive H\"older continuous function. Then, since $X$ is compact, $f$ has a positive minimum and for every $x\in X,$ the function $\kappa_f:X\times\Real\to\Real,$ defined by $\kappa_f(x,t)=\int_0^tf(\phi_sx)ds,$ is an increasing homeomorphism of $\Real.$ We then have a map $\alpha_f:X\times\Real\to\Real$ that verifies \begin{equation}\label{equation:inversa} \alpha_f(x,\kappa_f(x,t))=\kappa_f(x,\alpha_f(x,t))=t,\end{equation} for every $(x,t)\in X\times\Real.$ 

The \emph{reparametrization} of $\phi$ by $f,$ is the flow $\phi^f=\{\phi^f_t\}_{t\in\Real}$ on $X,$ defined by $\phi^f_t(x)=\phi_{\alpha_f(x,t)}(x),$ for all $t\in\Real$ and $x\in X.$

\subsubsection{Liv\v sic-cohomology classes}

Two H\"older functions $f,g:X\to\Real$ are \emph{Liv\v sic-cohomologous} if there exists $V:X\to\Real$ of class
$\textrm{C}^1$ in the flow's direction such that
$$f(x)-g(x)=\left.\frac{\partial}{\partial t}\right|_{t=0}V(\phi_t(x)).$$

Then one easily notices that:

\begin{enumerate}
\item  
If $f$ and $g$ are Liv\v sic
cohomologous then they have the same integral over any $\phi$-invariant measure,  and
\item
If $f$ and $g$ are both positive and Liv\v sic cohomologous, then the flows $\phi^f$ and $\phi^g$ are H\"older conjugate.
\end{enumerate}

\subsubsection{Periods and  measures} 
Let $O$ be the set of periodic orbits of $\phi$. 
If $a\in O$ then its {\em period} as a $\{\phi_t^f\}$ periodic orbit is
$$\int_0^{p(a)}f(\phi_s(x))\d s$$ 
where $p(a)$ is the period of $a$ for $\phi$ and $x\in a$.  In particular, if $\hdelta_a$ is the probability measure invariant by the flow and supported by the orbit $a$,  and if $$
\hdelta_a=\frac{\delta_a}{\braket{\delta_a|1}},
$$
then
$$\braket{\delta_a|f}=\int_0^{p(a)}f(\phi_s(x))\d s\ \ \textrm{and}\ \ p(a)= \braket{\delta_a|1}.$$
In general, if $\mu$ is a $\phi$-invariant measure on $X$ and $f:X\to\mathbb R$ is a H\"older function,
we will use the notation
$$\braket{\mu|f} =\int_X f \d\mu.$$

Let $\mu$ be a $\phi$-invariant probability measure on $X$ and let $\phi^f$ be the reparametrization of $\phi$ by $f$.
We  define $\widehat{f. \mu}$ by
$$\widehat{f. \mu}=\frac{1}{\braket{\mu|f}}f.\mu.$$ 
The map $\mu\mapsto \widehat{f. \mu}$ induces a bijection between $\phi$-invariant probability measures and $\phi^f$-invariant probability measures. If
$\hdelta^f_a$ is the unique $\phi_f$ invariant probability measure supported by $a$, 
then $\hdelta^f_a=\widehat{f.\delta_a}$. 
In particular, we have
\begin{equation}
\braket{\hdelta_a^f|g}=\frac{\braket{\delta_a|f.g}}{\braket{\delta_a|f}}\label{delta}
\end{equation}

\subsubsection{Entropy, pressure and equilibrium states}If $\mu$ is a $\phi$-invariant probability
measure on $X,$ then we denote by $h(\phi,\mu),$ its metric entropy. The Abramov formula \cite{abramov} relates the metric entropies of a flow and its reparameterization:
\begin{equation}
\label{eq:abramov}
h(\phi^f,\widehat{f.\mu})=\frac{1}{\int f\,\d \mu}h(\phi,\mu).
\end{equation}

Let $\mathcal M^{\phi}$ denote the set of $\phi$-invariant probability measures.
The \emph{pressure} of  a function $f:X\to\Real$ is defined by 
\begin{equation}\PP(\phi,f)=\sup_{m\in\mathcal M^{\phi}}\left(h(\phi,m)+\int_X f\,\d m\right)\label{def:press}.\end{equation}
In particular,
$$h_{\tope}(\phi)=\PP(\phi,0)$$
is the {\em topological entropy} of the flow $\phi$.

A measure $m\in\mathcal M^\phi$  on $X$ such that 
$$\PP(\phi,f)=h(\phi,m)+\int_X f\d m,$$
is called an \emph{equilibrium state} of $f$. 

An equilibrium state for the function $f\equiv0$ is called a {\em measure of maximal entropy}.

\begin{rmk}
The pressure $\PP(\phi,f)$ only depends on the  Liv\v sic cohomology class of $f$.
\end{rmk}

The following lemma from Sambarino \cite{sambarino-quantitative} is a consequence of the definition and the Abramov formula.

\begin{lemma}
\label{lemma:entropia2}{\rm (Sambarino \cite[Lemma 2.4]{sambarino-quantitative})}  
If $\phi$ is a H\"older continuous flow on a  compact metric space $X$ and $f:X\to\Real$
is a positive H\"older continuous function, then
$$P(\phi,-h f)=0$$ 
if and only if $h=h_{\tope}(\phi^f)$. 

Moreover, if $h=h_{\tope}(\phi^f)$ and  $m$ is an equilibrium state of $-hf$, 
then $\widehat{f.m}$ is a measure of maximal entropy for  the reparameterized flow  $\phi^f$. 
\end{lemma}

\subsection{Metric Anosov flows}

We shall assume from now on that the flow $\{\phi_t\}_{t\in\mathbb R}$ is a topologically transitive metric Anosov flow
on $X$. 

We recall that a flow $\{\phi\}_{t\in\mathbb R}$  on a metric space $X$ is {\em topologically transitive} if given any two
open sets $U$ and $V$ in $X$, there exists $t\in\mathbb R$ so that $\phi_t(U)\cap V$ is non-empty.

Let $X$ be metric space. Let $\mathcal L$ be an equivalence relation on $X$.
We denote by $\mathcal L_x$ the equivalence class of $x$ and call it the {\em leaf} through $x$, 
so that we have a partition of $X$ into leaves
$$
X=\bigsqcup_{y\in Y}\mathcal L_y,
$$ 
where $Y$ is the collection of equivalence classes of $\mathcal L$.
Such a partition  is a {\em lamination} if we can for every $x$ in $X$, an open neighbourhood $O_x$ of $x$,  two topological spaces  $U$ and $K$,  a homeomorphism $\nu_x =(\nu^1_x,\nu^2_x)$ called a {\em chart}  from $O_x$ to $U\times K$ satisfying the following conditions
\begin{itemize}
\item for all $z,w\in O_x\cap O_y$, 
$$
\nu^1_x(w)=\nu^1_x(z)\iff \nu^1_y(w)=\nu^1_y(z) ,
$$
\item we have that $w~{\mathcal L}~z$ if and only if  there exists a sequence $w_i$, 
$i\in\{1,\ldots n\}$ with $w_1=w$ and $w_n=z$, such that $w_{i+1}\in O_{w_i}$ and $\nu^1_{w_i}(w_i)=\nu^1_{w_i}(w_{i+1})$.
\end{itemize}

A {\em plaque open set} in the chart corresponding to $\nu$ is a set of the form 
$\nu\left(O\times\{z_0\}\right)$ where $x=\nu(y_0,z_0)$ and $O$ is an open set in $U$ containing $y_0$.
The {\em plaque topology} on $\mathcal L_x$ is the topology generated  by the plaque open sets. 
A {\em plaque neighborhood} of $x$ is a neighborhood for the plaque topology on $\mathcal L_x$. 

We say that two laminations $\mathcal L$ and $\mathcal L'$ define a {\em local product structure}, 
if for any point $x$ in $X$ there exist plaque neighborhoods $U$ and $U'$ of $x$ in $\mathcal L$ and $\mathcal L'$ 
respectively, and a map  $\nu:U\times U'\to X,$ which is an homeomorphism onto an open set of $X,$ such that 
$\nu$ is both a chart for $\mathcal L$ and for $\mathcal L'.$

Assume now we have a flow $\{\phi_t\}_{t\in\mathbb R}$ on $X.$ If $\mathcal L$ is a
lamination invariant by $\{\phi_t\},$ we say that $\mathcal L$ is {\em transverse to the flow}, 
if for every $x$ in $X$, there exists a plaque  neighborhood $U$ of $x$ in $\mathcal L_x$, 
a topological space $K$,  $\epsilon>0$, and  a chart 
$$
\nu: U\times K\times (-\epsilon,\epsilon)\to X,
$$   
such that
$$
\phi_t(\nu(u,k,s))=\nu(u,k,s+t).
$$
If $\mathcal L$ is tranverse to the flow, we define a new lamination, called the {\em central lamination} with respect to $\mathcal L$, denoted by $\mathcal L^c,$ by
letting $x\, \mathcal L^c\, y$ if and only if  there exists $s$ such that $\phi_s(x)\,\mathcal L \,y.$

Finally, a $\{\phi_t\}$ invariant lamination $\mathcal L$ is {\em contracted by the flow},
if there exists $t_0>0$ such that for all $x\in X$, there exists a chart $\nu_x:U\times K\to V$ 
of an open neighborhood $V$ of $x$, such that if 
$$z=\nu_x(u,k),\ \  {\rm and}\ \  y=\nu_x(v,k),$$
then for all $t>t_0$
$$  d(\phi_t(z),\phi_t(y))<\frac{1}{2}d(z,y).$$ 

\begin{definition}{\sc[Metric Anosov flow]}
A flow  $\{\phi_t\}_{t\in\mathbb R}$ on a compact metric space $X$  is {\em metric Anosov}, 
if there exist two laminations, $\mathcal L^+$ and $\mathcal L^-,$ transverse to the flow, such that
\begin{enumerate}
\item $(\mathcal L^+, \mathcal L^{-,c})$ defines a local product structure, 
\item $(\mathcal L^-, \mathcal L^{+,c})$ defines a local product structure,
\item $\mathcal L^+$ is contracted by the flow, and
\item $\mathcal L^-$ is contracted by the inverse flow.
\end{enumerate}
Then $\mathcal L^+$, $\mathcal L^-$, $\mathcal L^{+,c}$, $\mathcal L^{-,c}$ are respectively called the {\em stable}, {\em unstable}, { \em central stable} and {\em central unstable} laminations.
\end{definition}

\begin{rmk}\label{rmk:pollicott} In the language of Pollicott \cite{pollicottcoding}, a metric Anosov flow is a Smale flow:
the local product structure of $(\mathcal L^+, \mathcal L^{-,c})$ is what he calls the map
$$\left<\cdot,\cdot\right>:\{(x,y)\in X\times X:d(x,y)<\eps\}\to X.$$
\end{rmk}

\subsubsection{Liv\v sic's Theorem}

Liv\v sic \cite{livsic} shows that the Liv\v sic cohomology class of a H\"older function $f:X\to\Real$ is determined by its periods:

\begin{theorem}
\label{theorem:livsic}Let $f:X\to \Real$ be a H\"older continuous function, then $\braket{\delta_a|f}=0$ for every $a\in O$ if and only if $f$ is Liv\v sic cohomologous to zero.
\end{theorem}

\subsubsection{Coding}

We shall say that the triple $(\Sigma,\pi,r)$ is a \emph{Markov coding} for $\phi$ if $\Sigma$ is an irreducible two-sided subshift of finite type, the maps $\pi:\Sigma\to X$ and $r:\Sigma\to\Real_+^*$ are H\"older-continuous and verify the following conditions: Let $\sigma:\Sigma\to\Sigma$ be the shift, and let $\hat r:\Sigma\times\Real\to\Sigma\times\Real$ be the homeomorphism defined by $$\hat r(x,t)=(\sigma x, t-r(x)),$$ then

\begin{itemize}\item[i)] the map $\Pi:\Sigma\times\Real\to X$ defined by $\Pi(x,t)=\phi_t(\pi( x))$ is surjective and $\hat r$-invariant, \item[ii)] consider the suspension flow $\sigma^r=\{\sigma^r_t\}_{t\in\Real}$ on $(\Sigma\times\Real)/\hat r,$ then the induced map $\Pi:(\Sigma\times\Real)/\hat r\to X$ is bounded-to-one and, injective on a residual set which is of full measure for every ergodic invariant measure of total support of $\sigma^r.$ 
\end{itemize}

\begin{rmk} If a flow $\phi$ admits a Markov coding, then every reparametrization $\phi^f$ of $\phi$ also admits a Markov coding, simply by changing the roof function $r.$
\end{rmk}

We recall, see Remark \ref{rmk:pollicott}, that  a metric Anosov flow is a Smale flow. 
One then has the following theorem of Bowen \cite{bowen1,bowen2} and Pollicott \cite{pollicottcoding}.

\begin{theorem} A topologically transitive metric Anosov flow  on a compact metric space
admits a Markov coding.
\end{theorem}

\subsection{Entropy and pressure for Anosov flows}
The thermodynamic formalism of suspensions of subshifts of finite type extends thus to 
topologically transitive metric Anosov flows. For a positive H\"older 
function $f:X\to\Real_+$ and $T\in\Real$, we define
$$R_T(f)=\{a\in O\mid \braket{\delta_a| f}\leq T\}.$$
Observe that $R_T(f)$ only depends on the cohomology class of $f$.

\subsubsection{Entropy}

For a topologically transitive metric Anosov flow Bowen \cite{bowen1} (see also Pollicott \cite{pollicottcoding}) showed:

\begin{proposition}
\label{pos-fin} 
The topological entropy of a topologically transitive metric Anosov flow 
$\phi=\{\phi_t\}_{t\in\mathbb R}$ on a compact metric space $X$ is finite and positive. Moreover, 
$$h_{\tope}(\phi)=\lim_{T\to\infty}\frac{1}{T}\log\sharp\ \left\{a\in O\mid p(a)\leq T\right\}.$$ 
\end{proposition}

In particular, for a nowhere vanishing H\"older continuous function $f$,
$$h_f=\lim_{T\to\infty}\frac{1}{T}\log\sharp\left( R_T(f)\right)=h_{\tope}(\phi^f)$$
is finite and positive.

\subsubsection{Pressure}

The Markov coding may be used to show the pressure of a H\"older function $g:X\to\Real$ is finite  and that  there is a unique equilibrium state of $g$. We shall denote this equilibrium state as $m_g.$

\begin{theorem}{\sc [Bowen--Ruelle \cite{BowenRuelle},Pollicott \cite{pollicottcoding}]}
\label{theorem:uniqueness} 
Let $\phi=\{\phi_t\}_{t\in\mathbb R}$ be a topologically transitive metric Anosov flow  on a compact metric space $X$
and let $g:X\to\Real$ be a H\"older function,
then there exists a unique equilibrium state $m_g$ for $g.$ 
Moreover, if $f:X\to\Real$ is a H\"older function such that $m_f=m_g$, 
then $f-g$ is Liv\v sic cohomologous to a constant.
\end{theorem}

The pressure function has the following alternative formulation in this setting (see Bowen--Ruelle \cite{BowenRuelle}): 
\begin{eqnarray}
\PP(\phi,g)=\lim_{T\rightarrow\infty}\frac{1}{T}\log\left(\sum_{a\in R_T(1)}e^{\braket{\delta_a| g}}\right).\label{eq:press1}
\end{eqnarray}

\subsubsection{Measure of maximal entropy}
We have  the following equidistribution result of Bowen \cite{bowen1} (see also Pollicott \cite{pollicottcoding}).

\begin{theorem}\label{BM}  
A topologically transitive metric Anosov flow $\phi=\{\phi_t\}_{t\in\mathbb R}$ on a compact metric space $X$
has a unique probability measure $\mu_\phi$ of maximal entropy. Moreover,
\begin{equation}
\mu_\phi=\lim_{T\to\infty}\left(\frac{1}{\sharp R_T(1)}\sum_{a\in R_T(1)}\hdelta_a\right).
\end{equation}
\end{theorem}

The probability measure of maximal entropy for $\phi$ is called the {\em Bowen--Margulis measure} of $\phi$.

\subsection{Intersection and renormalised intersection}

\subsubsection{Intersection}

Let $\phi=\{\phi_t\}_{t\in\mathbb R}$ be a topologically transitive metric Anosov flow  on a compact metric space $X$.
Consider a positive H\"older function $f:X\to\Real_+$ and a continuous function $g:X\to\Real$. 
We define the {\em intersection of $f$ and $g$ as}
$$
\II(f,g)=\int \frac{g}{f}\d\mu_{\phi^f},
$$
where $\mu_{\phi^f}$ is the Bowen--Margulis measure  of the flow $\phi^f$.
We also have the following two alternative ways to define the intersection
\begin{eqnarray}
\II(f,g)&=&\lim_{T\to\infty}\left(\frac{1}{\sharp R_T(f)}\sum_{a\in R_T(f)}\frac{\braket{\delta_a| g}}{\braket{\delta_a|f}}\right)\label{inter1}\\
\II(f,g)&=&\frac{\int g\, \d m_{-h_f f}}{\int f\, \d m_{-h_f f}}
\label{inter2}
\end{eqnarray}
where $h_f$ is the topological entropy of $\phi^f$, and $m_{-h_f.f}$ is the equilibrium state of $-h_f.f$.
The first equality follows from Theorem \ref{BM} and Equation \eqref{delta}, the second equality follows
from the second part of Lemma \ref{lemma:entropia2}.

Since $\braket{\delta_a|f}$ depends only on the Liv\v sic cohomology class of $f$ and 
$\braket{\delta_a|g}$ depends only on the Liv\v sic cohomology class of $g$, the
intersection $\II(f,g)$  depends only  on  the  Liv\v sic cohomology classes of $f$ and $g$.

\subsubsection{A lower bound on the renormalized intersection}\label{subsubsection:bound}

For two positive H\"older functions $f,g:X\to\Real_+$ define the \emph{renormalized intersection} as
$$\JJ(f,g)=\frac{h_g}{h_f}\II(f,g),$$
where $h_f$ and $h_g$ are the topological entropies of $\phi^f$ and $\phi^g$. 
Uniqueness of equilibrium states together with the definition of the pressure imply the following proposition.

\begin{proposition}
\label{proposition:lowerbound} 
If $\phi=\{\phi_t\}_{t\in\mathbb R}$ is a topologically transitive metric Anosov flow on a compact metric space $X$, and
$f:X\to\mathbb R_+$  and $g:X\to\mathbb R_+$ are positive H\"older functions, then 
$$\JJ(f,g)\geq1.$$
Moreover, $\JJ(f,g)=1$ if and only if $h_f f$ and $h_g g$ are Liv\v sic cohomologous.
\end{proposition}

\begin{proof}
Since $\PP(\phi,-h_g g)=0$, 
$$h_g\int g\,\d m\geq h(\phi,m)$$ 
for all $m\in\mathcal{M}^\phi$ and, by Theorem \ref{theorem:uniqueness}, equality  holds only for $m=m_{-h_g g},$
the equilibrium state of $-h_g g.$
Applying the analogous inequality for $m_{-h_f f},$ together with Abramov's formula (\ref{eq:abramov}) and Lemma \ref{lemma:entropia2}, one sees that
$$h_g\int g\,\d m_{-h_f.f}\geq h(\phi,m_{-h_f.f})=h_f\int f\,\d m_{-h_f.f},$$ 
which implies that $\JJ(f,g)\geq1.$ 

If $\JJ(f,g)=1$, then $m_{-h_g g}=m_{-h_f f}$ and thus, applying theorem \ref{theorem:uniqueness}, one 
sees that $h_g g-h_f f$ is Liv\v sic cohomologous to a constant $c.$ Thus, 
$$0=\PP(\phi,-h_g g)=\PP(\phi,-h_f f-c)=\PP(\phi,-h_f f)-c=-c.$$ 
Therefore, $h_g g$ and $h_f f$ are Liv\v sic cohomologous.
\end{proof}

\subsection{Variation of the pressure and the pressure form}
\label{pressure form}

McMullen \cite{mcmullen} introduced a pressure metric on the space of Liv\v sic cohomology classes of
pressure zero H\"older functions on a shift space $\Sigma$. In this section, we use his construction
to produce a pressure form, and associated semi-norm, on the space of pressure zero H\"older functions
on our flow space $X$.

\subsubsection{First and second derivatives}
For $g$ a H\"older continuous function with mean zero (i.e. $\int g \ dm_f = 0$), we define the \emph{variance of  $g$} with respect to $f$ as 
$$\Var(g,m_f)=\lim_{T\to\infty}\frac1T\int \left(\int_0^Tg(\phi_s(x))\d s\right)^2\d m_f(x),$$
where $m_f$ is the equilibrium state of $f.$ 
Similarly, for two mean zero H\"older continuous functions $g$ and $h$, we define the {\em covariance} of $g$ and $h$
with respect to $f$ as
$$\Cov(g,h,m_f)=\lim_{T\to\infty}\frac1T\int \left(\int_0^Tg(\phi_s(x))\d s\right)\left(\int_0^Th(\phi_s(x))\d s\right)\d m_f(x).$$
Since $m_f$ is invariant with respect  to flow $\phi_t$, we may rewrite this as 
$$\Cov(g,h,m_f)=\lim_{T\to\infty} \int g(x)\left(\int_{-T}^Th(\phi_s(x))\d s\right)\d m_f(x).$$

We shall omit the background flow in the notation of the pressure function 
and simply write
$$\PP(\cdot)=\PP(\phi,\cdot).$$

\begin{proposition}{\sc(Parry-Pollicott \cite[Prop. 4.10,4.11]{parry-pollicott}, 
Ruelle \cite{ruelle})}
\label{proposition:deriv-presion}
Suppose that  $\phi=\{\phi_t\}_{t\in\mathbb R}$ is a topologically transitive metric 
Anosov flow on a compact metric space $X$, and
$f:X\to\mathbb R$  and $g:X\to \mathbb R$ are H\"older functions.
If $m_f$ is the equilibrium state of $f$, then
\begin{enumerate}
\item 
The function $t\mapsto \PP(f+tg)$ is  analytic, 
\item The first derivative is given by 
$$\left.\frac{\partial \PP(f+tg)}{\partial t}\right|_{t=0}=\int g\,\d m_f,$$
\item  If $\int g\, \d m_f=0$ then $$\left.\frac{\partial^2 \PP(f+tg)}{\partial t^2}\right|_{t=0}=\Var(g,m_f),$$ 
\item If $\Var(g,m_f)=0$ then $g$ is  Liv\v sic cohomologous to zero.
\end{enumerate}
\end{proposition}

\subsubsection{The pressure form}
Let $C^h(X)$ be the set of real valued H\"older continuous functions on $X$.
Define $\mathcal P(X)$ to be the set of  pressure zero H\"older functions on $X$, i.e.
$$\mathcal P(X)=\left\{\Phi\in C^h(X): \PP(\Phi)=0\right\}.$$ 
The tangent space of $\mathcal P(X)$ at $\Phi$ is the set 
$${\hsf T}_\Phi\mathcal P(X)=\ker \d_\Phi\PP=\left\{g\in C^h(X)\mid\int g\, \d m_\Phi=0\right\}$$ 
where $m_\Phi$ is the equilibrium state of $\Phi.$ 
Define the \emph{pressure semi-norm} of \hbox{$g\in \TT_\Phi\mathcal P(X)$} as
$$\|g\|_\PP^2=-\frac{\Var(g,m_\Phi)}{\int \Phi \,\d m_\Phi}.$$

One has the following computation.

\begin{lemma}
Let $\phi=\{\phi_t\}_{t\in\mathbb R}$ be a topologically transitive metric Anosov flow  on a compact metric space $X$.
If $\{ \Phi_t\}_{t\in(-1,1)}$ is a smooth one parameter family contained in $\mathcal P(X)$, then 
$$\|\dot{\Phi}_0\|_\PP^2=\frac{\int \ddot{\Phi}_0\,\d m_{\Phi_0}}{\int \Phi_0\,\d m_{\Phi_0}}.$$
\end{lemma}

\begin{proof}
As $\PP(\Phi_t) = 0$ by differentiating twice we get the equation
$${\rm D^2}\PP(\Phi_0)(\dot{\Phi}_0,\dot{\Phi}_0) + {\rm D}\PP(\Phi_0)(\ddot{\Phi}_0) = 0 = \Var (\dot{\Phi}_0, m_{\Phi_0}) + \int \ddot{\Phi}_0 \d m_{\Phi_0}.$$
Thus
$$\Vert \dot{\Phi}_0\Vert_\PP^2  =  -\frac{\Var (\dot{\Phi}_0,m_\Phi)}{\int \Phi_0 \,\d m_{\Phi_0}}  =  \frac{\int \ddot{\Phi}_0 \,\d m_{\Phi_0}}{\int \Phi_0 \,\d m_{\Phi_0}}.$$
\end{proof}

We then have the following relation, generalizing  Bonahon \cite{bonahon}, 
between  the renormalized intersection and the pressure metric.

\begin{proposition}
\label{prop:deriv2} 
Let $\phi=\{\phi_t\}_{t\in\mathbb R}$ be a topologically transitive metric Anosov flow  on a compact metric space $X$.
If $\{f_t:X\to\mathbb R_+\}_{t\in(-1,1)}$ is a one-parameter family
of positive H\"older functions  and  
$\Phi_t=-h_{f_t}f_t$ for all $t\in(-1,1)$, then 
$$\left.\frac{\partial^2}{\partial t^2}\right|_{t=0}\JJ(f_0,f_t)=\|\dot \Phi_0\|_\PP^2.$$
\end{proposition}

\begin{proof} 
By Equation  \eqref{inter2} and the definition of the renormalised intersection,  we see that
$$\JJ(f_0,f_t)=\frac{\int \Phi_t\,\d m_{\Phi_0}}{\int\Phi_0\,\d m_{\Phi_0}}.$$ 
Differentiating twice and applying the previous lemma, one obtains
$$\left.\frac{\partial^2}{\partial t^2}\right|_{t=0}\JJ(f_0,f_t)=\frac{\int \ddot{\Phi}_0\,\d m_{\Phi_0}}{\int\Phi_0\,\d m_{\Phi_0}}=\|\dot\Phi_0\|_\PP^2$$ 
which completes the proof.
\end{proof}

So, the pressure semi-norm arises naturally from the {\em pressure form} $p$ which is the symmetric 2-tensor 
on  ${\hsf T}_\Phi\mathcal P(X)$ given by the Hessian of $\JJ_\Phi=\JJ(\Phi,\cdot)$. One may compute that if
$f,g \in {\hsf T}_\Phi\mathcal P(X)$, then
$$p(f,g) = - \frac{\Cov(f,g, m_{\Phi})}{\int \Phi \,\d m_{\Phi}}.$$

\subsection{Analyticity of entropy, pressure and intersection}

We now show that pressure, entropy and intersection vary analytically for analytic families of positive H\"older functions.

\begin{proposition}\label{pro:ana0}
Let $\phi=\{\phi_t\}_{t\in\mathbb R}$ be a topologically transitive metric Anosov flow  on a compact metric space $X$.
Let $\{f_u:X\to\mathbb R\}_{u\in D}$ and $\{g_v:X\to\mathbb R\}_{v\in D}$ be  two  analytic  families of H\"older functions. Then the function 
$$u\mapsto \PP(f_u)$$
is analytic. Moreover, if  the family $\{f_u\}_{u\in D}$ consists of positive functions then the functions 
\begin{eqnarray}
u&\mapsto& h_u=h_{f_u},\\
(u,v)&\mapsto&\II(f_u,g_v).
\end{eqnarray}
are both analytic.
\end{proposition}

\begin{proof}
Since the pressure function is analytic on the space of H\"older functions (see  Parry-Pollicott \cite[Prop. 4.7]{parry-pollicott}
or Ruelle \cite[Cor. 5.27]{ruelle}) the function $u\mapsto \PP(f_u)$ is analytic.

Since the family $\{f_u\}_{u\in D}$ consists of positive functions,   Proposition \ref{proposition:deriv-presion}  implies
that
$$
\left.\frac{\d}{\d t}\right\vert_{t=h_u}\PP(-tf_u)=\left.\frac{\d}{\d t}\right\vert_{t=h_u}\PP(-h_uf_u-(t-h_u)f_u)=
-\int f_u\d m_{-h_uf_u}<0.
$$
Thus an application of the Implicit Function Theorem yields that $u\mapsto h_u$ is analytic.

We also get that 
$$
(u,v,t)\mapsto\left.\frac{\d}{\d t}\right\vert_{t=0}\PP(-h_{u}f_u+tg_v),
$$
is analytic.
But, applying Proposition \ref{proposition:deriv-presion} again,
$$
\left.\frac{\d}{\d t}\right\vert_{t=0}\PP(-h_{u}f_u+tg_v)=\int g_v \,\d m_{-h_uf_u}.
$$
Thus the function $(u,v)\mapsto \int g_v \,\d m_{-h_uf_u}$ is analytic. Similarly (taking $g_v=f_u$), the function
$u\mapsto \int f_u \,\d m_{-h_uf_u}$ is analytic. Thus, we get, by Equation \eqref{inter2} that
$$
(u,v)\mapsto \II(f_u,g_v)=\frac{\int g_v \d m_{-h_uf_u}}{\int f_u \d m_{-h_uf_u}},
$$
is analytic.

\end{proof}

\section{The geodesic flow of a projective Anosov representation}\label{sec:geodano}

In this section, we define a flow $(\Ug,\{\phi_t\}_{t\in\mathbb R})$ associated to a projective Anosov representation
$\rho:\Gamma\to\sln$.
We will show that $\Ug$ is a H\"older reparameterization of the geodesic flow $\Gg$ of the
domain group $\Gamma$, so it will make sense to refer to $\Ug$ as the {\em geodesic flow} of the
representation.

Let $F$ be the total space of the bundle  over
$$\rpn^{(2)}=\rpn\times\rpn^*\setminus\{(U,V)\mid U\not\subset V\},$$ 
whose fiber at the point $(U,V)$ is the space 
$$ \mathsf M (U,V)=\{(u,v)\mid u\in U,\ v\in V,\ \braket{v|u}=1\}/\sim,$$
where $(u,v)\sim(-u,-v)$ and $\rpn^*$ is identified with the projective space of the dual space $(\mathbb R^m)^*$.
Notice that $u$ determines $v$, so that $F$ is an $\mathbb R$-bundle. One may also identify $\mathsf M(U,V)$ with
the space of norms on $U$.

Then $F$ is equipped with a natural $\mathbb R$-action, given by
$$
\phi_t(U,V,(u,v))=(U,V,(e^tu,e^{-t}v)).
$$
If $\rho:\Gamma\to\sln$ is a projective Anosov representation and $\xi$ and $\theta$ are the associated limit maps,
we consider the associated pullback bundle
$$F_\rho=(\xi,\theta)^*F$$
over $\bg^{(2)}$ which inherits an $\mathbb R$ action from the action on $F$.
The action of $\Gamma$ on $\bg^{(2)}$ extends to an action on $F_\rho$.
If we let
$$\Ug=F_\rho/\Gamma,$$
then the $\mathbb R$-action on $F_\rho$ descends to a flow $\{\phi_t\}_{t\in\mathbb R}$ on $\Ug$, which
we call the {\em geodesic flow} of the representation.

\begin{proposition}{\sc[The geodesic flow]}
\label{geoflowreparam}
If $\rho:\Gamma\to\sln$ is a projective Anosov representation, then
the action of $\Gamma$ on $F_\rho$ is proper and cocompact. Moreover, the flow  $\{\phi_t\}_{t\in\mathbb R}$ on $\Ug$ 
is H\"older conjugate  to a H\"older reparameterization of the Gromov geodesic flow on $\Gg$ and the orbit
associated to $[\gamma]$, for any infinite order primitive element $\gamma\in\Gamma$,  has period
$\Lambda(\rho)(\gamma)$.
\end{proposition}

We produce a $\Gamma$-invariant H\"older orbit equivalence between $\widetilde{\Gg}$ and $F_\rho$ which
is a homeomorphism.
Recall that $\widetilde{\Gg}=\bg^{(2)}\times\mathbb R$ and  that  $\widetilde{\Gg}/\Gamma=\Gg$. 
Since the action of $\Gamma$ on $\widetilde{\Gg}$ is proper and cocompact,
it follows immediately that $\Ug$ is  H\"older conjugate  to a H\"older reparameterization of the Gromov geodesic flow on $\Gg$.

\begin{proposition}\label{geoflowreparam2}
If $\rho:\Gamma\to\sln$ is a projective Anosov representation,
there exists a $\Gamma$-equivariant H\"older orbit equivalence 
$$\tilde\nu: \widetilde{\Gg}\to F_\rho$$
which is a homeomorphism. 
\end{proposition}

Let $E_\rho$ be the flat bundle associated  to $\rho$ on $\Gg$. Recall that $E_\rho$ splits as 
$$
E_\rho=\Xi\oplus\Theta.
$$
Let $\{\psi_t\}_{t\in\mathbb R}$ be the lift of the geodesic flow on $\Gg$ to a flow on $E_\rho$.
We first observe that we may produce a H\"older metric on the bundle $\Xi$ which is contracting on all scales.

\begin{lemma}\label{XiHolder}
There exists a H\"older  metric $\tau^0$ on the bundle $\Xi$ and $\beta>0$ such that for all $t>0$ we have,
$$
\psi_t^*(\tau^0)<  e^{-\beta t}\tau^0.
$$
\end{lemma}

\begin{proof}
Let $\tau$ be any H\"older metric on $\Xi$. Since $\rho$ is projective Anosov, Lemma \ref{contracting on Xi}
implies that there
exists $t_0>0$ such that 
$$
\psi_{t_0}^*(\tau)\leq \frac{1}{4}\tau.
$$
Choose $\beta>0$ so that $2<e^{\beta t_0}<4$ and, for all $s$, let $\tau_s=\psi_s^* (\tau)$.
Let 
$$
\tau^0=\int_0^{t_0} e^{\beta s}\tau_s\,\d s.
$$ 
Notice that $\tau^0$ has the same regularity as $\tau$. If $t>0$, then
\begin{eqnarray}
\psi_t^*(\tau^0)&=&\int_{0}^{t_0} e^{\beta s}\tau_{t+s}\,\d s\cr
&=&e^{-\beta t}\int_{t}^{t+t_0} e^{\beta u}\tau_u\,\d u.
\label{eq:XiHolder1}
\end{eqnarray}
Now observe that
\begin{eqnarray}
\int_{t}^{t+t_0} e^{\beta u}G_u\,\d u
&=&\tau^0+\int_{t_0}^{t+t_0}e^{\beta u}\tau_u\d u-\int_{0}^{t}e^{\beta u}\tau_u\,\d u\cr
&=&\tau^0+\int_0^te^{\beta u}\psi_u^*\left(e^{\beta t_0}\psi_{t_0}^*(\tau)- \tau\right)\d u.
\end{eqnarray}
But
$$
e^{\beta t_0}\psi_{t_0}^*(\tau)\leq \frac{e^{\beta t_0}}{4} \tau< \tau.
$$
Thus 
$$
\int_{t}^{t+t_0} e^{\beta u}\tau_u \,\d u < \tau^0.
$$
and the result follows from Inequality \eqref{eq:XiHolder1}.
\end{proof}

{\em Proof of Proposition \ref{geoflowreparam2}}\label{proof:reparam}
Let $\tau^0$ be the metric provided by Lemma \ref{XiHolder} and let $\beta$ be the associated positive number.
Let $\tilde\Xi$ denote the line bundle over $\bg^{(2)}\times\mathbb R$
which is the lift of $\Xi$. Notice
that $\tau^0$ lifts to a H\"older metric $\tilde \tau^0$ on $\tilde\Xi$.
Our H\"older orbit equivalence  
$$\tilde\nu:\bg^{(2)}\times\mathbb R\to F_\rho$$ 
will be given by
$$
\tilde\nu(x,y,t)=\left(x,y,(u{(x,y,t)},v(x,y,t))\right),
$$
where $\tilde \tau^0_{(x,y,t)}(u(x,y,t))=1$ and   $\tilde \tau^0_{(x,y,t)}$ is the metric on the line $\xi(x)$ 
induced by the metric $\tilde G^0$ by
regarding $\xi(x)$ as the fiber of $\tilde\Xi$ over the point $(x,y,t)$.
The fact that $\psi_t^*\tau^0<\tau^0$ for all  $t>0$ implies that $\tilde\nu$ is injective. Since $\tilde \tau^0$ is H\"older
and $\Gamma$-equivariant,  $\tilde \nu$ is also H\"older and $\Gamma$-equivariant.

It remains to prove that $\tilde\nu$ is proper. We will argue by contradiction. If $\tilde\nu$ is not proper,
then there exists a sequence $\{(x_n,y_n,t_n)\}_{n\in\mathbb N}$ 
leaving every compact subset of \hbox{$\bg^{(2)}\times\mathbb R$}, such that 
$\{\tilde\nu(x_n,y_n,t_n)\}_{n\in\mathbb N}$ converges to $(x,y,(u,v))$ in $F_\rho$.
Letting $\tilde\nu(x_n,y_n,t_n)=(x_n,y_n,(u_n,v_n))$, 
we see immediately that 
$$
\lim_{n\to\infty} x_n=x,\  \lim_{n\to\infty} y_n=y,\ \textrm{and}\  
 \lim_{n\rightarrow\infty} (u_n,v_n)= (u,v).$$
Writing 
$\tilde\nu(x_n,y_n,0)=(x_n,y_n,(\hat u_n,\hat v_n))$ and $\tilde\nu(x,y,0)=(x,y,(\hat u,\hat v))$, we obtain, 
by the continuity of the map $\tilde\nu$, 
$$
\lim_{n\rightarrow\infty} (\hat u_n,\hat v_n)=(\hat u,\hat v).
$$
If $t>0$, then
\begin{equation*}
\frac
{\tilde \tau^0_{(x,y,t)}}{\tilde \tau^0_{(x,y,0)}}=\frac
{\psi_t^*\left(\tilde \tau^0_{(x,y,0)}\right)}{\tilde \tau^0_{(x,y,0)}}<e^{-\beta t}.
\end{equation*}
In particular,  
\begin{equation}
\left\vert \frac{\braket{v\mid u_n}}{\braket{v\mid\hat u_n}}\right\vert < e^{-\beta t_n}.\label{proper1}
\end{equation}
Without loss of generality, either $t_n\to\infty$ or $t_n\to -\infty$.
If $t_n\to\infty$, then by Inequality \eqref{proper1},
$$
0=\lim_{n\rightarrow\infty}\frac{\braket{v\mid u_n}}{\braket{v\mid \hat u_n}},$$
on the other hand,
$$
\lim_{t\rightarrow\infty}\frac{\braket{v\mid  u_n}}{\braket{v\mid\hat u_n}}=\frac{\braket{v\mid u}}{\braket{v\mid \hat u}}\not=0.
$$
We have thus obtained a contradiction. Symmetrically, if  $t_n\to -\infty$, then $$
0=\lim_{n\rightarrow\infty}\frac{\braket{v\mid \hat u_n}}{\braket{v\mid u_n}}=\frac{\braket{v\mid \hat u}}{\braket{v\mid u}}\not=0,
$$
which is again a contradiction.

The restriction of $\tilde\nu$ to each orbit  $\{(x,y)\}\times\mathbb R$ is a proper, continuous, injection into
the fiber of $F_\rho$ over $(x,y)$ (which is also homeomorphic to $\mathbb R$). It follows that the restriction of
$\tilde\nu$ to each orbit is a homeomorphism onto the image fiber.
We conclude that $\tilde\nu$ is surjective and hence a proper, continuous, bijection. Therefore, $\tilde\nu$ is
a homeomorphism. 
This completes the proof of Proposition \ref{geoflowreparam2}.

\medskip

In order to complete the proof of Proposition \ref{geoflowreparam},
it only remains to compute the period of the orbit associated to $[\gamma]$ for an infinite order primitive 
element $\gamma\in\Gamma$.
Since $\rho$ is projective Anosov, Proposition \ref{proxi2} implies that $\rho(\gamma)$ is proximal,
$\xi(\gamma^+)$ is the attracting line and $\theta(\gamma^-)$ is the repelling hyperplane.
If $u\in\xi(\gamma^+)$ and $v\in\theta(\gamma^-)$ one sees that 
$$\rho(\gamma)(u)=\LL(\gamma)(\rho)\ u\textrm{ and }\rho(\gamma)( v)=\frac1{\LL(\gamma)(\rho)}\ v.$$ 
Thus, $(\gamma^+,\gamma^-,(u,v))$ and 
$$(\gamma^+,\gamma^-,\LL(\gamma)(\rho)u,\frac1{\LL(\gamma)(\rho)}v)=\phi_{\log(\Lambda(\gamma)(\rho))}(\gamma^+,\gamma^-,(u,v))$$ 
project to the same point on $\Ug.$ (Recall that 
$$(\LL(\gamma)(\rho)u,\frac1{\LL(\gamma)(\rho)}v)\sim
(-\LL(\gamma)(\rho)u,\frac{-1}{\LL(\gamma)(\rho)}v)$$
in $M(\xi(\gamma^+),\theta(\gamma^-))$.)
Since $\gamma$ is primitive, this finishes the proof.
\qed

\section{The geodesic flow is a metric Anosov flow}

In this section, we prove that the
geodesic flow of a projective Anosov representation is a metric Anosov flow: 

\begin{proposition}{\sc[Anosov]}
\label{pro:geodanoprelim}
If  $\rho:\Gamma\to\sln$ is a projective Anosov representation, then
the geodesic flow $(\Ug,\{\phi_t\}_{t\in\mathbb R})$ is a topologically transitive metric Anosov flow.
\end{proposition}

The reader with a background in hyperbolic dynamics may be convinced by the following heuristic argument: 
essentially the splitting of an Anosov representation yields a section of some (product of) flag manifolds and the graph of this
section should be thought as a Smale locally maximal hyperbolic set; then the result follows from the ``fact" that the 
restriction of the flow on  such a set is a metric Anosov flow. However, the above idea does not exactly work, 
and moreover it is not easy to extricate it from  the existing literature in the present framework.
Therefore, we give a detailed and {\it ad-hoc} construction, although the result should be true in a rather general setting.

The topological transitivity of $(\Ug,\{\phi_t\}_{t\in\mathbb R})$
follows immediately  from the topological transitivity of the action of $\Gamma$ on $\bg^2$. 
We  define a metric on the geodesic flow in Section \ref{metano}, introduce the stable and unstable leaves 
in Section \ref{leafano}, explain how to control the metric along the unstable leaves in Section \ref{leaf lift } 
and finally proceed to the proof in Section \ref{geodano}.
A more precise version of Proposition \ref{pro:geodanoprelim} is given by Proposition \ref{pro:geodano}.

\subsection{The geodesic flow as a metric space}\label{metano}

Recall that $F$ is the total space of an $\mathbb R$-bundle over $\rpn^{(2)}$ whose fiber at the point $(U,V)$ 
is the space 
$$\mathsf M (U,V)=\{(u,v)\mid u\in U,\ v\in V,\ \braket{v|u}=1\}/\sim.$$
Since $\rpn^{(2)}\subset \rpn\times\rpn^*$, any Euclidean metric on $\mathbb R^m$ 
gives rise to a metric on $F$ which is a subset of 
$$ \rpn\times\rpn^*\times\left(\left(\mathbb R^m\times\left(\mathbb R^m\right)^*\right)/\pm 1\right).$$
The metric on $F$ pulls back to a metric on $F_\rho$.
A metric on $F_\rho$ obtained by this procedure is called a {\em linear metric}. 
Any two linear metrics are  bilipschitz equivalent.

The following lemma allows us to use a linear metric to study $F_\rho$.

\begin{lemma}\label{linear:metric}
There exists a $\Gamma$-invariant metric $d_0$ on $F_\rho$ which is locally bilipschitz equivalent to any linear metric. 
\end{lemma}

The $\Gamma$-invariant metric $d_0$ descends to a metric on $\Ug$  which we will also call $d_0$ and is defined for every $x$ and $y$ in $F_\rho$ by
$$
d_0(\pi(x),\pi(y))=\inf_{\gamma\in\Gamma}(x,\gamma(y)),
$$
where $\pi$ is the projection $F_\rho\to \Ug$.

\begin{proof}{}
We first notice that all linear metrics on $F_\rho$ are bilipschitz to one another, so that it suffices to construct a metric
which is locally bilipschitz to a fixed linear metric $d$.

Let $V$ be an open subset  of $F_\rho$ with compact closure which contains a closed fundamental
domain for the action of $\Gamma$ on $F_\rho$. Since the action of $\Gamma$ on $F_\rho$ is proper,
$\{ V_\gamma=\gamma(V)\}_{\gamma\in\Gamma}$ is a locally finite cover of $F_\rho$. Let 
$\{d_\gamma=\gamma^*d\}_{\gamma\in\Gamma}$
be the associated family of metrics on $F_\rho$. Since each element of $\Gamma$ acts as a  bilipschitz automorphism
with respect to any linear metric, any two metrics in the family
$\{d_\gamma=\gamma^*d\}_{\gamma\in\Gamma}$ are  bilipschitz equivalent.

We will use this cover and the associated family of metrics to construct a \hbox{$\Gamma$-invariant} metric on $F_\rho$.
A {\em path} joining two points $x$ and $y$ in $F_\rho$ is a pair of tuples
$$
\mathcal P= ((z_0,\ldots, z_n),(\gamma_0,\ldots,\gamma_n)),
$$ 
where  $(z_0,\ldots, z_{n})$  is  an $n$-tuple of points in $F_\rho$ and $(\gamma_0,\ldots,\gamma_n)$  is an $n$-tuple  of 
elements of $\Gamma$ such that
\begin{itemize}
\item $x=z_0\in V_{\gamma_0}$ and  $y=z_n\in V_{\gamma_n}$, 
\item for all $n>i>0$, $z_i\in V_{\gamma_{i-1}}\cap V_{\gamma_i}$.
\end{itemize}

The {\em length} of a path is given by
$$
\ell(\mathcal P)=\frac{1}{2}\left(\sum_{i=0}^{n-1}d_{\gamma_i}(z_i,z_{i+1})+d_{\gamma_{i+1}}(z_i,z_{i+1})\right)
$$ 
We then define
$$
d_0(x,y)=\inf\{\ell(\mathcal P)\mid \mathcal P\hbox{ joins $x$ and $y$}\}.
$$
It is clear that $d_0$ is a $\Gamma$-invariant pseudo metric. It remains to show
that $d_0$ is a metric which is locally bilipschitz to $d$.

Let $z$ be a point in $F_\rho$. Then there exists a neighborhood $Z$ of $z$ so that
$$
A=\{\gamma\mid V_\gamma\cap Z\not=\emptyset\},
$$
is a finite set. 
Choose $\alpha>0$ so that 
$$
\bigcup_{\gamma\in A}\{x\mid d_\gamma(z,x)\leq \alpha\}\subset Z.
$$
Let $K$ be chosen so that if $\alpha,\beta\in A $, then $d_\alpha$ and $d_\beta$ are $K$-bilipschitz.
Finally, let 
$$
W=\bigcap_{\gamma\in A}\left\{x\mid d_\gamma(z,x)\leq \frac{\alpha}{10K}\right\}.
$$
By construction, if $x$ and $y$ belong to $W$, then for all $\gamma\in A$,
\begin{equation}
d_\gamma(x,y)\leq \frac{\alpha}{5K}\label{ineq:control10}.
\end{equation}

Let $x$ be a point in $W$. Let 
$\mathcal P=((z_0,\ldots,z_n),(\gamma_0,\ldots,\gamma_n))$ be a path joining $x$ to a point  $y$.  

If there exists $j$ such that $\gamma_j\not\in A$, then
\begin{eqnarray}
\label{ineq:control1}
\ell(\mathcal P)&\geq& \frac{1}{2}\sum_{i=0}^{i=j-1}d_{\gamma_i}(z_{i-1},z_{i})\cr
&\geq& \frac{1}{2K}\left(\sum_{i=0}^{i=j-1}d_{\gamma_{j-1}}(z_i,z_{i+1})\right)\cr
&\geq& \frac{1}{2K}d_{\gamma_{j-1}}(z_0,z_{j})\geq \frac{1}{2K}\left(d_{\gamma_{j-1}}(z,z_{j}))-d_{\gamma_{j-1}}(z_0,z)\right)\cr
&\geq&\frac{1}{2K}\left(\alpha-\frac{\alpha}{10K}\right)\geq \frac{\alpha}{5K}.
\end{eqnarray}

If $\gamma_j\in A$ for all $j$, then the triangle inequality and the definition of $K$ immediately imply
that for all $\gamma\in A$,
\begin{equation}
\ell(\mathcal P)\geq \frac{1}{K}d_{\gamma}(x,y).
\label{ineq:control2}
\end{equation}
Inequalities (\ref{ineq:control1}) and (\ref{ineq:control2})  imply that 
\begin{equation}
d_0(x,y)\geq \frac{1}{K}\inf\left(\frac{\alpha}{5}, d_\gamma(x,y)\right)>0,\label{ineq:control4}
\end{equation}
hence $d_0$ is a metric.
Moreover, if $x,y\in W$, then by inequalities (\ref{ineq:control4})
and (\ref{ineq:control10}), 
\begin{equation}
d_0(x,y)\geq \frac{1}{K}d_\gamma(x,y).\label{ineq:control5}
\end{equation}
By construction, and taking the path $\mathcal P_0=\left((x,y),(\gamma,\gamma)\right)$ with $\gamma$ in $A$, we also get 
\begin{equation}
d_0(x,y)\leq \ell(\mathcal P_0)= d_\gamma(x,y).\label{ineq:control6}\end{equation}
As consequence of inequalities  \eqref{ineq:control5} and \eqref{ineq:control6}, $d_0$ is bilipschitz on $W$ to any $d_\gamma$ with $\gamma\in A$.

Since $d$ is bilipschitz to $d_\gamma$ for any $\gamma\in A$,  we see that $d_0$ is bilipschitz to $d$ on $W$.

Since $z$  was arbitrary, it follows that $d_0$ is locally bilipschitz to $d$.
\end{proof}

\subsection{Stable and unstable leaves}\label{leafano}

In this section, we define the stable and unstable  laminations of the geodesic flow $F_\rho$.
Let 
$$Z=(x_0,y_0,(u_0,v_0))$$ 
be a point in $F_\rho$. 

\begin{enumerate}
\item The {\em unstable leaf through $Z$} is
$$\mathcal L_Z^-=\{(x,y_0,(u,v_0))\mid x\in \bg, \ u\in\xi(x),\  \braket{v_0|u}=1\}.$$
The {\em central unstable leaf through $Z$} is 
\begin{eqnarray*}
\mathcal L_Z^{-,c}&=&\{(x,y_0,(u,v))\mid x\in \bg, (u,v)\in\mathsf M(\xi(x),\theta(y_0))\}\cr &=&\bigcup_{t\in\mathbb R}\phi_t(\mathcal L_Z+).
\end{eqnarray*}

\item 
The {\em stable leaf through $Z$} is 
$$\mathcal L_Z^+=\{(x_0,y,(u_0,v))\mid y\in \bg,\ v\in\theta(y),\ \braket{v|u_0}=1\}.$$
The {\em central stable leaf through $Z$} is 
\begin{eqnarray*}\mathcal L_Z^{+,c}&=&\{(x_0,y,(u,v))\mid y\in \bg,(u,v)\in\mathsf M(\xi(x_0),\theta(y))\}\cr &=&\bigcup_{t\in\mathbb R}\phi_t(\mathcal L_Z^-).\end{eqnarray*}
\end{enumerate}
Observe that $\mathcal L_Z^+$  is homeomorphic to $\bg\setminus\{x_0\}$ and
$\mathcal L^-_Z$ is homeomorphic to $\bg\setminus\{y_0\}$.

The following two propositions are immediate from our construction.

\begin{proposition}{\sc[Invariance]}
If $\gamma\in\Gamma$ and $t\in\mathbb R$, then
$$
\mathcal L_{\gamma(Z)}^\pm=\gamma\left(\mathcal L_{Z}^\pm\right)\ \ {\rm and}\ \ 
\mathcal L_{\phi_t(Z)}^\pm=\phi_t\left(\mathcal L_{Z}^\pm\right).
$$
\end{proposition}

\begin{proposition}{\sc[product structure]} \label{pro:ps}
The (two) pairs of lamination $(\mathcal L^\pm,\mathcal L^{\mp,c})$ define a local product structure
on $F_\rho$, and hence on $\Ug$.
\end{proposition}

\medskip\noindent
{\bf Remark:} Throughout this section, we abuse notation by allowing $\{\phi_t\}_{t\in\mathbb R}$ to denote both the 
flow on $\Ug$ and the flow on $F_\rho$ which covers it and letting $\mathcal L^\pm$ denote both
the lamination on $F_\rho$ and the induced lamination on $\Ug$.

\subsection{The leaf lift  and the distance}
\label{leaf lift }

In this section we introduce the {\em leaf lift } and show that it helps in controlling  distances in $F_\rho$.

We first define the leaf lift  for points in the bundle $F$.
Let $A=(U,V,(u_0,v_0))$ be a point in $F$. We observe that there exists a unique continuous map, called the {\em leaf lift } 
from
$$
O_{A}=\{w\in \rpn^*\mid U\cap  {\rm ker}(w)=\{0\}\}.
$$
to 
$\left((\mathbb R^m)^*\setminus\{0\}\right)/\pm 1$
such that $w$ is taken to  $\Omega_{w,A}$
such that
\begin{eqnarray}
\Omega_{w,A}\in w,& &
\braket{\Omega_{w,A}|u_0}=1.\label{def:omega}
\end{eqnarray}
In particular, $\Omega_{v_0,A}=v_0$. Observe that at this stage  the leaf lift  coincides with the classical notion of an  affine chart.

The following lemma records immediate properties of the leaf lift .

\begin{lemma}\label{ineqmetricP1}
Let $\Vert. \Vert_1$ be a Euclidean norm on $\mathbb R^n$ and $d_1$ the associated metric on $\rpn^*$. 
If $A=(x,y,(u,v))\in F$, then there exist constants $K_1 >0$ and $\alpha_1>0$ such that for $w_0, w_1\in\rpn^*$
\begin{itemize}
\item If  $d_1 (w_i,y)\leq\alpha_1$, for $i=0,1$, then
\begin{eqnarray*}
 \Vert \Omega_{w_0,A}- \Omega_{w_1,A}\Vert_1 \leq K_1 d_1\left(w_0,w_1\right),\label{inegmetric10} \end{eqnarray*}
\item If  $\Vert \Omega_{w_i,A}- \Omega_{y,A}\Vert_1\leq\alpha_1$ for $i=0,1$, then
\begin{eqnarray*}
d_1 (w_0,w_1)\leq K_1 \Vert \Omega_{w_0,A}- \Omega_{w_1,A}\Vert_1.\label{inegmetric11}
\end{eqnarray*}
\end{itemize}
\end{lemma}

If $Z=(x,y,(u_0,v_0))\in F_\rho$ and $W=(x,w,(u_0,v))\in\mathcal L^+_Z$, then  we define the {\em leaf lift}
$$
\omega_{W,Z}=\Omega_{\xi^*(w),(\xi(x),\xi^*(y),(u_0,v_0))}=v.
$$

The following result allows us to use the leaf lift  to bound distances in $F_\rho$

\begin{proposition}\label{sect-dist}
Let $d_0$ be a  $\Gamma$-invariant  metric on $F_\rho$ which is locally bilipschitz equivalent to a linear metric and let
\hbox{$Z\to\Vert.\Vert_Z$} be a  \hbox{$\Gamma$-invariant} map from $F_\rho$ into the space of Euclidean metrics on 
$\mathbb R^m$.
There exist positive constants $K$ and $\alpha$ such that for any $Z\in F_\rho$ and any $W\in\mathcal L^-_Z$,
\begin{itemize}
\item if $ d_0\left(W,Z\right)\leq\alpha$, then
\begin{eqnarray}
 \bVert \omega_{W,Z}- \omega_{Z,Z}\bVert_Z \leq K d_0\left(W,Z\right),\label{inegmetric10b}
 \end{eqnarray}
 \item if $\bVert \omega_{W,Z}- \omega_{Z,Z}\bVert_Z\leq\alpha$ then
 \begin{eqnarray}
d_0\left(W,Z\right)\leq K \bVert \omega_{W,Z}- \omega_{Z,Z}\bVert_Z.\label{inegmetric11b}
\end{eqnarray}
\end{itemize}
\end{proposition}

\begin{proof}
Since $\Gamma$ acts cocompactly on $F_\rho$ and  both $d_0$ and  the section $\Vert.\Vert$ are
\hbox{$\Gamma$-invariant}, it suffices to prove the previous assertion for $Z$ in a compact subset  $R$ of $F_\rho$. 
Observe first that $d_0$ is uniformly $C$-bilipschitz on $R$ to any of the linear metrics $d_Z$ coming from 
$\Vert.\Vert_Z$ for $Z$ in $R$ for some constant $C$.

Lemma \ref{ineqmetricP1} implies that,  for all $Z\in R$, there exist
positive constants $K_Z$ and $\alpha_Z$ such that if $W_0,W_1\in\mathcal L^-_Z\cap O$, then
\begin{itemize}
\item If  $d_0\left(W_i,Z\right)\leq\alpha_Z$ for $i=0,1$, then
\begin{eqnarray*}
 \bVert \omega_{W_0,Z}- \omega_{W_1,Z}\bVert_Z \leq K_Z d_0\left(W_0,W_1\right), \end{eqnarray*}
\item If  $\bVert \omega_{W_i,Z}- \omega_{Z,Z}\bVert_Z\leq\alpha_Z$ for $i=0,1$, then
\begin{eqnarray*}
d_0\left(W_0,W_1\right)\leq K_Z \bVert \omega_{W_0,Z}- \omega_{W_1,Z}\bVert_Z.
\end{eqnarray*}
\end{itemize}

Since $R$ is compact, one may apply the classical argument which establishes that continuous functions are uniformly
continuous on compact sets, to show that there are  positive constants $K$ and $\alpha$ which work for all $Z\in R$.
\end{proof}

\subsection{The geodesic flow is Anosov}
\label{geodano}

The following result completes the proof of Proposition \ref{pro:geodanoprelim}

\begin{proposition}{\sc[Anosov property]}\label{pro:geodano}
Let $\rho:\Gamma\to\sln$ be a projective  Anosov representation, and let
$\mathcal L^\pm$ be the laminations  on $\Ug$ defined above.
Then there exists a metric on $\Ug$,  H\"older equivalent to the H\"older structure on $\Ug$, 
such that
\begin{enumerate}
\item $\mathcal L^+$ is contracted by the flow,
\item  $\mathcal L^-$ is contracted by the inverse  flow,\end{enumerate}
\end{proposition}

We first show that the leaf lift  is contracted by the  flow.

\begin{lemma}\label{contract1}
There exists a $\Gamma$-invariant map $Z\mapsto \Vert.\Vert_Z$ from $F_\rho$ into the space of Euclidean metrics on $\mathbb R^m$, such that for every positive integer $n$,  there exists  $t_0>0$
such that  if $t>t_0$, $Z\in F_\rho$, and 
$W\in\mathcal L^+_Z$ then
\begin{equation}
\left\Vert\omega_{\phi_t(W),\phi_t(Z)}-\omega_{\phi_t(Z),\phi_t(Z)}\right\Vert_{\phi_t(Z)}\leq\frac{1}{2^n} \bVert \omega_{W,Z}-\omega_{Z,Z}\bVert_Z\label{ineganos1}.
\end{equation}
\end{lemma}

The following notation will be used in the proof.
\begin{itemize}
\item
For a vector space $A$ and a subspace $B\subset A$, let 
$$
B^\perp=\{\omega\in A^*\mid B\subset \ker(\omega)\}.
$$

\item 
We consider the $\Gamma$-invariant splitting of the trivial $\mathbb R^m$-bundle 
$$
F_\rho\times\mathbb R^m=\hat\Xi\oplus\hat\Theta
$$
\begin{itemize}
\item where  $\hat\Xi$ is the line bundle over $F_\rho$ such that the fiber above $(x,y,(u,v))$ is given by $\xi(x)$
and
\item  $\hat\Theta$ is a hyperplane bundle over $F_\rho$ with fiber $\theta(y)$ above the point $(x,y,(u,v))\in F_\rho$.
\end{itemize}
\end{itemize}

\begin{proof}
Suppose that $Z=(x,y,(u_0,v_0))$ and $W=(x,w,(u_0,v))\in\mathcal L^+_Z$, then by the definition of the leaf lift 
$$
(\omega_{W,Z}-\omega_{Z,Z})(u_0)=0,$$
and thus
$$
\omega_{W,Z}=\alpha_{W,Z}+\omega_{Z,Z},
$$  
where $\alpha_{W,Z} \in\xi(x)^\perp$.
Then
$$
\phi_{t}(\omega_{W,Z})=\phi_t(\alpha_{W,Z})+\phi_t(\omega_{Z,Z}).
$$
We choose a $\Gamma$-invariant map from $F_\rho$ into the space of Euclidean metrics on $\mathbb R^m$
so that for all $Y\in F_\rho$
$$
\bVert\omega_{Y,Y}\bVert_Y=1.
$$
Then
$$
\omega_{\phi_t(Z),\phi_t(Z)}=\frac{1}{\bVert \phi_t(\omega_{Z,Z})\bVert_{\phi_t(Z)}}\phi_t(\omega_{Z,Z}),$$
hence
$$
\omega_{\phi_t(W),\phi_t(Z)}=\frac{\phi_t(\alpha_{W,Z})}{\bVert \phi_t(\omega_{Z,Z})\bVert_{\phi_t(Z) }}+ \omega_{\phi_t(Z),\phi_t(Z)}.
$$
It follows that
$$
\bVert\omega_{\phi_t(W),\phi_t(Z)}-\omega_{\phi_t(Z    ),\phi_t(Z)}\bVert_{\phi_t(Z)} = \frac{\bVert \phi_t(\alpha_{W,Z})\bVert_{\phi_t(Z)}}{\bVert \phi_t(\omega_{Z,Z})\bVert_{\phi_t(Z)} } 
$$
Since $\rho$ is projective Anosov,  and $(\Ug,\{\phi_t\}_{t\in\mathbb R})$ is a H\"older reparameterization of
$(\Gg,\{\psi_t\}_{t\in\mathbb R})$, there exists $t_1>0$ so that for all $Z\in F_\rho$ and
for all $t>t_1$, if  $v\in \hat\Xi_Z^\perp$ and $w\in\hat\Theta_Z^\perp$ , then 
$$
\frac{\bVert \phi_{t}(v)\bVert_{\phi_t(Z)} }{\bVert \phi_{t}(w)\bVert_{\phi_t(Z)} }\leq \frac{1}{2} \frac{\bVert v\bVert_Z}{\bVert w\bVert_Z}.
$$
Thus, since $\alpha_{W,Z} \in\hat\Xi^\perp_Z$ and
$\omega_{Z,Z}\in\hat\Theta_Z^\perp$, 
for all $n\in\mathbb N$ and  $t>nt_1$, we have
$$
\bVert\omega_{\phi_t(W),\phi_t(Z)}-\omega_{\phi_t(Z),\phi_t(Z)}\bVert_{\phi_t(Z)}\leq \frac{1}{2^n} \frac{\bVert \alpha_{W,Z}\bVert_Z}{\bVert \omega_{Z,Z}\bVert_Z}.$$
Since $\alpha_{W,Z}=\omega_{W,Z}-\omega_{Z,Z}$ and  $\bVert \omega_{Z,Z}\bVert_Z=1$, the previous assertion yields the result
with $t_0=nt_1$.
\end{proof}

We are now ready to establish Proposition \ref{pro:geodano}.

\medskip\noindent
{\em Proof of Proposition \ref{pro:geodano}:}
Let $K$ and $\alpha$ be as in Proposition \ref{sect-dist}.
Choose $n\in\mathbb N$ so that 
\begin{equation}
\frac{K}{2^n}\leq 1\ \ \ {\rm and}\ \ \ \frac{K^2}{2^n}\leq \frac{1}{2}.\label{kn}
\end{equation}
Let $t_0$ be the constant  from Lemma \ref{contract1} with our choice of $n$.

Suppose that $Z\in F_\rho$, $W\in \mathcal L^+_Z$, $t>t_0$ and $d_0(W,Z)\leq \alpha$.
Then, by Inequality \eqref{inegmetric10b},
\begin{equation}
\bVert \omega_{W,Z}-\omega_{Z,Z}\bVert\leq K d_0(W,Z).\label{proof:geodan1}
\end{equation}
By Lemma \ref{contract1}, 
\begin{equation}
\bVert\omega_{\phi_t(W),\phi_t(Z)}-\omega_{\phi_t(Z),\phi_t(Z)}\bVert_{\phi_t(Z)}\leq{\frac{1}{2^n}}\bVert \omega_{W,Z}-\omega_{Z,Z}\bVert_Z.\label{proof:geodan2}
\end{equation}
In particular, combining Equations \eqref{proof:geodan1}, \eqref{proof:geodan2} and \eqref{kn},
\begin{equation}
\bVert\omega_{\phi_t(W),\phi_t(Z)}-\omega_{\phi_t(Z),\phi_t(Z)}\bVert_{\phi_t(Z)}\leq{\frac{1}{2^n}}K\alpha\leq\alpha.
\end{equation}
Thus, using Inequality \eqref{inegmetric11b},
\begin{equation}
d_0(\phi_t(W),\phi_t(Z))\leq K \bVert\omega_{\phi_t(W),\phi_t(Z)}-\omega_{\phi_t(Z),\phi_t(Z)}\bVert_{\phi_t(Z)}.\label{proof:geodan3}
\end{equation}
Combining finally Equations \eqref{proof:geodan1}, \eqref{proof:geodan2}, \eqref{proof:geodan3} and \eqref{kn}, we get that 

\begin{equation}
d_0(\phi_t(W),\phi_t(Z))\leq \frac{K^2}{2^n} d_0(W,Z)\leq \frac{1}{2} d_0(W,Z)
\label{proof:geodan4}
\end{equation}
for all $t>t_0$.

Therefore $\mathcal L^+$ is contracted by the  flow on $F_\rho$.
\vskip 0.2truecm
Let us now consider what happens in the quotient $\Ug=F_\rho/\Gamma$. For any $Z\in F_\rho$ and $\epsilon>0$, let
$$
\mathcal L_\epsilon^\pm(Z)=\mathcal L^\pm_Z\cap B(Z,\epsilon).
$$
and let
$$K_\epsilon(Z)=\Pi_Z\left( \mathcal L_\epsilon^+(Z)\times \mathcal L_\epsilon^-(Z)\times (-\epsilon,\epsilon)\right),$$
where $\Pi_Z$ is the product structure of Proposition \ref{pro:ps}.
By Proposition  \ref{geoflowreparam}, there exists  $\epsilon_0>0$ such that for all
$\gamma\in\Gamma\setminus\{1\}$ and $Z\in F_\rho$,
$$
\gamma(K_{\epsilon_0}(X))\cap K_{\epsilon_0}=\emptyset.
$$
Let $\epsilon\in (0,\min\{\epsilon_0/2,\alpha\})$ and $\hat Z\in \Ug$. Choose $Z\in F_\rho$ in the pre image of $\hat Z$,  then
inequality (\ref{proof:geodan4}) holds  for the  flow on $\Ug$ for
points in the chart which is the projection of $K_\epsilon(Z)$.
Therefore, $\mathcal L^+$ is contracted by the flow on $\Ug$.

A symmetric proof holds for the central unstable leaf.

\section{Analytic variation of the dynamics}\label{anvar}

In order to apply the thermodynamic formalism we need to check that if $\{\rho_u\}_{u\in M}$ is an analytic
family of projective Anosov representations, then the associated limit maps and reparameterizations
of the Gromov geodesic flow may be chosen to vary analytically, at least locally. Our proofs generalize earlier proofs
of the stability of Anosov representations, see Labourie \cite[Proposition 2.1]{labourie-anosov} and
Guichard-Wienhard \cite[Theorem 5.13]{guichard-wienhard}, and that the limit maps vary continuously,
see Guichard-Wienhard \cite[Theorem 5.13]{guichard-wienhard}. In the process, we also see that our
limit maps are H\"older.

We will make use of the following concrete description of the analytic structure of $\hom(\Gamma,\ms G)$.
Suppose that $\Gamma$ is a word hyperbolic group, hence finitely presented, and
let $\{g_1,\ldots,g_m\}$ be a generating set for $\Gamma$. If $\ms G$ is a real semi-simple Lie group, then
$\hom(\Gamma,\ms G)$ has the structure of a real algebraic variety.
An {\em analytic family}  \hbox{$\beta:M\to \hom(\Gamma,\ms G)$} of homomorphisms of $\Gamma$ into  $\ms G$
is a map with domain an analytic manifold $M$ so that, for each $i$, the map
$\beta_i:M\to \ms G$ given by $\beta_i(u)=\beta(u)(g_i)$ is real analytic.
If $\ms G$ is a complex Lie group, we may similarly define complex analytic families of  homomorphisms of
a complex analytic manifold into $\hom(\Gamma,\ms G)$.

We first show that the limit maps of an analytic family of Anosov homomorphisms vary analytically.
We begin by setting our notation.
If $\alpha>0$, $X$ is a compact metric space and $D$ and $M$ are real-analytic manifolds, 
then we let $C^\alpha(X,M)$ denote the space of $\alpha$-H\"older maps of $X$ into $M$ and let
$C^\omega(D,M)$ denote the space of real analytic maps of $D$ into $M$. If $D$ and $M$ are complex analytic
manifolds, we will abuse notation by letting $C^\omega(D,M)$ denote the space of complex analytic maps.

\begin{theorem}
\label{Comega}
Let $\ms G$ be a real (or complex) semi-simple Lie group and let $\ms P$ be a parabolic subgroup of $\ms G$. 
Let  $\{\rho_u\}_{u\in D}$ be a real  (or complex) analytic family of  homomorphisms of $\Gamma$ into 
$\ms G$ parameterized by a real (or complex) disk $D$ about $0$.
If  $\rho_0$ is a $(\ms G,\ms P)$-Anosov homomorphism with limit map $\xi_0:\bg\to\ms G/\ms P$,
then there exists a sub-disk $D_0$ of $D$ (containing $0$), $\alpha>0$  and a continuous map 
$$\xi:D_0\times \bg\to \ms G/\ms P$$
with the following properties:
\begin{enumerate}
\item 
If $u\in D_0$, then $\rho_u$ is a $(\ms G,\ms P)$-Anosov homomorphism with 
$\alpha$-H\"older limit map $\xi_u:\bg\to\ms G/\ms P$ given by $\xi_u(\cdot)=\xi(u,\cdot)$.
\item 
If $x\in\bg$, then $\xi_x:D_0\to \ms G/\ms P$ given by $\xi_x=\xi(\cdot,x)$ is real (or complex) analytic
\item
The map from $\bg$ to $C^\omega(D_0,\ms G/\ms P)$ given by
$
x\mapsto \xi_x
$
is $\alpha$-H\"older.
\item 
The map from $D_0$ to $C^\alpha(\bg,\ms G/\ms P)$  given by $u\to \xi_u$ is real (or complex) analytic.
\end{enumerate}
\end{theorem}

Given a projective Anosov representation \hbox{$\rho:\Gamma\to\sln$}, we constructed a 
geodesic flow $U_\rho\Gamma$ which is a reparameterization of the Gromov geodesic flow
$\Gg$. In Section \ref{analytic variation of reparam},
we show that given a real analytic family of projective Anosov representions, one may choose
the parameterizing functions to vary analytically.

\begin{proposition}
\label{Creparam}
Let $\{\rho_u\}_{u\in D}$ be a real analytic  family of  projective  Anosov homomorphisms of $\Gamma$ into $\sln$ 
parameterized by a disk about $0$. Then,  there exists a sub-disk $D_0$ about $0$ and 
a real analytic family $\{f_u:\Gg\to\mathbb R\}_{u\in D_0}$ of positive H\"older functions
such that the reparametrization  of $\Gg$ by $f_{u}$ is  H\"older conjugate to $\ms U_{\rho_u}\Gamma$ for all $u\in D_0$.
\end{proposition}

We first observe that the real analytic case of Theorem \ref{Comega} follows from the
complex analytic case, which we will establish in Section \ref{analytic variation of limit maps}.
If $\ms G$ is a real semi-simple Lie group and $P$ is a parabolic subgroup of $\ms G$, we let $G^{\mathbb C}$
and $P^{\mathbb C}$ be the complexification of $\ms G$ and $\ms P$.
Observe that a $(\ms G,\ms P)$-Anosov representation is automatically a 
$(\ms G^{\mathbb C},\ms P^{\mathbb C})$-Anosov representation. 
On a sub-disk $D_1$ of $D$, containing $0$, one may extend
$\{\rho_u\}_{u\in D_1}$ to a complex analytic family $\{\rho_u\}_{u\in D_1^{\mathbb C}}$ of homomorphisms of
$\Gamma$ into $\ms G^{\mathbb C}$ defined on the complexification $D_1^{\mathbb C}$ of $D_1$.
The map $\xi:D_0^{\mathbb C}\times\bg\to \ms G^{\mathbb C}/\ms P^{\mathbb C}$ provided by the complex analytic
case of Theorem \ref{Comega} restricts to a map $\xi|_{D_0}:D_0\times\bg\to \ms G/\ms P$ with the 
desired properties.
Notice that the real analyticity in properties (2) and (4) follows
from the fact that  restrictions of complex analytic functions to real analytic submanifolds are real analytic.

\subsection{Transverse regularity}

In this section, we set up our notation and establish a version of the
$C^r$-section Theorem of Hirsch-Pugh-Shub \cite[Theorem 3.8]{hps} which 
keeps track of the transverse regularity of the resulting section. Our version of Hirsch, Pugh and Shub's
result will be the main tool in the proof of Theorem \ref{Comega}.

\begin{definition}{\sc [transversely regular functions]} 
Let $D$ be a real  (or complex) disk, let $X$ be a compact metric space and let $M$ be a real (or complex) analytic manifold.
A continuous function  $f:D\times X\to M$ is {\em transversely  real (or complex) analytic} if
\begin{enumerate}
\item 
For every $x \in X$, the function $f_x:D\to M$ given by \hbox{$f_x(\cdot)=f(\cdot,x)$} is real (or complex) analytic, and
\item
The function from $X$ to $C^\omega(D,M)$ given by $x\to f_x$  is continuous.
\end{enumerate}
Furthermore, we say that $f$ is {\em $\alpha$-H\"older (or Lipschitz) transversely  real (or complex) analytic} 
if the map in (2) is $\alpha$-H\"older (or Lipschitz).
\end{definition}

If we replace $M$ with a  $C^k$ manifold,
we can similarly define $\alpha$-H\"older (or Lipschitz) transversely  $C^k$ functions
by requiring  that the maps in (1) are  $C^k$ and that  the map in (2) from $X$
to $C^p(D,M)$ is  $\alpha$-H\"older (or Lipschitz) for all $p\le k$,

Similarly, we define transverse regularity of bundles in terms of the transverse regularity of their trivializations.

\begin{definition}{\sc[tranversally regular bundles]}
Suppose that  the fiber of a bundle \hbox{$\pi:E\to D\times X$} is a real (or complex) analytic manifold $M$ and
that $D$ is a real (or complex) disk.
We say that $E$ is  {\em transversely  real (or complex) analytic} 
if it admits a family of  trivializations of the form \hbox{$\{D\times U_\alpha\times M\}$} (where $\{U_\alpha\}$ is an
open cover of $X$) so that the the corresponding change of coordinate functions are 
transversely real (or complex) analytic. We similarly say  \hbox{$\pi:E\to D\times X$} is 
{\em $\alpha$-H\"older (or Lipschitz) transversely  real (or complex) analytic}  if it admits a family of
trivializations which are $\alpha$-H\"older (or Lipschitz) transversely  real (or complex) analytic.

In this case, a section of $E$ is $\alpha$-H\"older (or Lipschitz) transversely  real  (or complex) analytic,
if in any of the trivializations the corresponding map to $M$ is
$\alpha$-H\"older (or Lipschitz) transversely  real (or complex) analytic.
\end{definition}

Clearly, if $M$ is a $C^k$-manifold,
we can similarly define $\alpha$-H\"older (or Lipschitz) transversely  $C^k$ bundles and 
sections.

We are now ready to state our version of the $C^r$-Section Theorem.

\begin{theorem}
\label{local contraction}
Let $X$ be a compact metric space and let $M$ be a complex analytic (or $C^k$) manifold.
Suppose that 
$\pi:E\to D\times X$ 
is a Lipschitz transversely complex analytic (or $C^k$) bundle  with fibre $M$ and $D$ is a complex (or real) disk.
Let \hbox{$f:X\to X$} be a Lipschitz homeomorphism and let 
$F$ be a Lipschitz transversely   complex analytic (or $C^k$) bundle automorphism of $E$ lifting 
${\rm id}\times f$.

Suppose that  $\sigma_0$ is a section of the restriction of $E$ over $\{0\}\times X$ which is fixed by $F$ and that
$F$ contracts along $\sigma_0$. Then there exists a neighborhood $U$ of $0$ in  $D$, a positive number  $\alpha>0$,
an $\alpha$-H\"older transversely complex analytic (or $C^k$)
section $\eta$ over $D_0\times X$  and a neighborhood $B$ of
\hbox{$\eta(U\times X)$ in $\pi^{-1}(U\times X)$} such that
\begin{enumerate}
\item
$F$ fixes $\eta$,
\item
$F$ contracts $E$ along $\eta$,
\item
$\eta|_{\{0\}\times X}=\sigma_0$, and
\item 
if $\nu:U\times X\to E$ is a section so that $\nu(U\times X)\subset B$ and $\nu$ is fixed by $F$,
then $\nu=\eta$.
\end{enumerate}
\end{theorem}

We recall that if $U$ is a subset of $D$, then a section $\sigma$ over $U\times X$ is
{\em fixed } by $F$ if $F(\sigma(u,x))=\sigma(u,f(x))$. 
In such a case, we further say that $F$ {\em contracts} along $\sigma$ if 
there exists a continuously varying fibrewise Riemannian metric $\Vert\cdot\Vert$ on the bundle $E$
such that if 
$${\rm D}^f F_{\sigma(u,x)}:T_{\sigma(u,x)}\pi^{-1}(u,x)\to T_{\sigma(u,f(x))}\pi^{-1}(u,f(x))$$
is the fibrewise tangent map, then
$$\Vert {\rm D}^f F_{\sigma(u,x)}\Vert <1.$$

We will derive Theorem \ref{local contraction} from the following version of the $C^r$-section theorem
which is a natural generalization of the ball bundle version of the $C^r$-section theorem in
Shub \cite[Theorem 5.18]{shub}.

\begin{theorem}{\sc [Fixed sections]}\label{lem:coc}
Let $X$ be a compact metric space equipped with a Lipschitz homeomorphism $f:X\to X$.
Suppose that $\pi:W\to D\times X$ is a Lipschitz transversely complex
analytic (or $C^k$) Banach space bundle, $D$ is a complex (or real) disk,
$B\subset W$ is the  closed ball sub-bundle of radius $r$,
and $F$  is a Lipschitz transversely   complex analytic (or $C^k$) bundle morphism of $B$ 
lifting the  homeomorphism ${\rm id}\times f:D\times X\to X$.

If $F$ contracts $B$, then there exists a unique $\alpha$-H\"older transversely  complex analytic (or $C^k$) section 
$\eta$ of $B$ which is fixed by $F$ (for some $\alpha>0$). 
\end{theorem}

Notice that we have not assumed that $F$ is either linear or bijective.

\begin{proof} 
Let $\sigma$ be the zero section of $B$. Observe that $\sigma$ has the same regularity as $W$ 
and is thus transversally complex analytic (or $C^k$).

We first assume that $\pi:W\to D\times M$ is a Lipschitz transversely $C^k$-bundle.
The existence of a unique continuous fixed section $\eta$ is a standard application of the
contraction mapping theorem.
Explicitly, for all $(u,x)\in D\times X$, we let
\begin{equation}
\eta(u,x)=\lim_{n\to\infty} F^n(\sigma(u,f^{-n}(x)).
\label{eq:coc1}
\end{equation}

We must work harder to show that $\eta$ is $\alpha$-H\"older transversely complex analytic (or $C^k$).
We first  assume that $W$ is transversely $C^k$ --and so is $\sigma$--  and obtain the $C^k$-regularity of $\eta$.
For any $p\in\mathbb N$, let $\Gamma^p$ be the  Lipschitz Banach bundle over $X$ whose fibers over a point $x\in X$
is the Banach space 
$\Gamma^p_{x}$
of $C^p$-sections of the restriction of $W$ to $D\times\{x\}$. Let $B^p$ be the sub-bundle whose fiber
$B^p_x$ over $x$ is the set of those sections with values in $B$.

Notice that each fiber $B^p_x$  can be identified with $C^p(D\times\{x\},B_0)$ where $B_0$ is a closed ball 
of radius $r$ in the fiber Banach space.
Let $F_*^p$ be the bundle automorphism of $\Gamma_p$ given by
$$
[F_*^p(\nu)](u,x)=F(\nu(u,f^{-1}(x))).
$$
We can renormalise the metric on $D$, so that all the derivatives of $F$ of order $n$ (with $p\geq n\geq 1$) along $D$
are arbitrarily small. 
Thus  after this renormalisation  the metric on $D$,  $F^p_*$ is contracting, since $F$ is contracting. 
We now apply
Theorem 3.8 of  Hirsch-Pugh-Shub \cite{hps}
(see also Shub \cite[Theorem 5.18]{shub}) to obtain an invariant $\alpha$-H\"older section $\omega$.
By the uniqueness of fixed sections, we see that 
$$
\eta(u,x)=\omega(x)(u)
$$
for all $1\le p\le k$.
It follows that $\eta$ is $\alpha$-H\"older transversely  $C^k$. 

Now suppose that $D$ is a complex disk and  \hbox{$\pi:E\to D\times X$} is Lipschitz transversely complex analytic bundle.
We see, from the above paragraph, that there exists a unique $\alpha$-H\"older transversely $C^k$ section
$\eta_k$ for all $k$.
By the uniqueness $\eta_k$ is independent of $k$ and we simply denote it by $\eta$.
Then, by Formula \eqref{eq:coc1},  for all $x\in X$,
\hbox{$\eta|_{D\times\{x\}}$} is a $C^k$-limit of a sequence of complex analytic sections for all $k$, 
hence is  complex analytic itself. It follows that $\eta$ is $\alpha$-H\"older transversely complex analytic.
\end{proof}

We now notice that one may identify a neighborhood of the section $\sigma_0$ in the statement of
Theorem \ref{local contraction} with a ball sub-bundle of a vector bundle. 

\begin{lemma}
\label{localization}
Let $\pi:E\to D\times X$
be a transversely complex analytic (or $C^k$)  bundle over $D\times X$, 
$D$ is a complex (or real) disk about 0, and 
$\sigma$ is a section  of $E$ defined over $\{0\} \times X$.
Then there exists
\begin{enumerate}
\item a neighborhood $U$ of zero in $D$,
\item a transversely complex analytic (or $C^k$)  closed ball bundle  $B$ of radius $R$ in a complex (or real) 
vector bundle $F$,
\item a transversely complex analytic (or $C^k$) bijective map from $B$ to a neighborhood of the graph of $\sigma_0$ so that 
\begin{itemize}
\item the graph of $\sigma_0=\sigma|_{\{0\}\times X}$ is in the image of the graph of the zero section,
\item the fibrewise metric on $B$ coincides along $\sigma_0$ with the fiberwise metric on $E$. 
\end{itemize}
\end{enumerate}
\end{lemma}

\begin{proof}
We first give the proof in the case that $\sigma$ is defined over $D\times X$.
Let $Z$ be the transversely complex analytic (or $C^k$) vector bundle over $D\times X$
so that the fibre over the point $(u,x)$ is given by $\ms T_{\sigma(u,x)}(\pi^{-1}(u,x))$.
We equip $Z$ with a Riemannian metric coming from $E$ and let $B(r)$ be the closed  ball sub-bundle of radius $r>0$.

Using the trivializations, we can find, after restricting to an open neighborhood $U$ of  $0$ in  $D$,
\begin{itemize}
\item a finite cover  $\{O_i\}_{1\leq i\leq n}$ of  $X$, 
\item an open neighborhood $W$ of the graph of $\sigma$, 
\item transversely holomorphic  (or $C^k$-diffeomorphic) bundle maps $\phi_i$ 
defined on \hbox{$\left.W\right\vert_{U\times O_i}$} with values in \hbox{$\left.Z\right\vert_{U\times O_i}$}
so that for all \hbox{$(u,x)\in U\times O_i$}
\begin{eqnarray}
\phi_i(\sigma(u,x))&=&0\in \ms T_{\sigma(u,x)}(\pi^{-1}(u,x))\cr
\rm D^f_{\sigma(u,x)}\phi_i&=&{\rm Id}.
\end{eqnarray}
\end{itemize}
Let $\{\psi_i\}_{1\leq i\leq n}$ be  a partition of unity  on $X$ subordinate to $\{O_i\}_{1\leq i\leq n}$
and, for each $i$, let $\hat\psi_i:W\to [0,1]$ be obtained by composing the projection of $W$ to $X$ with $\psi_i$.
One may then define $\Phi:W \to Z$ by letting
$$
\Phi=\sum_{i=1}^n\hat\psi_i\phi_i.
$$ 
Since $\hat\psi_i$ is constant in the direction of $D$, $\Phi$ is transversely holomorphic (or $C^k$-diffeomorphic), 
$$\Phi(\sigma(u,x))= 0  \ \ \textrm{and} \ \ {\rm D}^f_{\sigma(y)}\Phi = {\rm Id}.$$
It then follows from the implicit function theorem, that one may further restrict $U$ and $W$  so that
$\Phi$ is  a transversely holomorphic (or $C^k$-diffeomorphic) isomorphism of $W$ with $B(r)$ for some $r$.

If $\sigma$ is only defined on $\{0\}\times X$, 
it now suffices to extend the section $\sigma_0$ to a section $\sigma$ defined over $U\times X$ where 
$U$ is a neighborhood of $0$ in  $D$.
Composing $\pi$ with the projection  $\pi_2:D\times X\to X$, we may consider the bundle $\pi_2\circ\pi:E\to X$.
Then $\sigma_0$ is a section of $\pi_2\circ\pi$.
We now apply the result of the previous paragraph,  in the case where the disk is 0-dimensional, 
to identify, in a complex analytic (or $C^k$) way, 
a neighborhood of the graph of $\sigma_0$ with a ball bundle $B$ in a vector bundle $F$ over $X$.

Now $\pi$ restricts to a bundle morphism from \hbox{$\pi\circ\pi_2:B\to X$} to  \hbox{$\pi_2:D\times X\to X$}
which is a fiberwise complex analytic (or $C^k$) submersion  and whose fiberwise derivatives vary continuously.
Let $W$ be a linear sub-bundle of $F$, so that  if $W_x$ and $F_x$ are the fibers over $x\in X$, then
$$
\ms T(\pi^{-1}(0,x))\oplus W_x=F_x.
$$
Thus, after further restricting $B$, $\pi$ becomes a fiberwise complex analytic (or $C^k$)  injective local diffeomorphism 
from $W\cap B$ to  $D\times X$ whose
fiberwise derivatives vary continuously.

Applying the Implicit Function Theorem (with parameter), 
we obtain a neighborhood $U$ of $0$ and a map \hbox{$\sigma:U\times X\to B$} 
which is fiberwise complex analytic (or $C^k$) and whose fiberwise derivatives vary continuously, 
so that \hbox{$\pi\circ \sigma={\rm{Id}}$}. 
Thus $\sigma$ is the desired section of $E$.
\end{proof}

Theorem \ref{local contraction} now follows from Theorem \ref{lem:coc} and Lemma \ref{localization}.

\medskip\noindent
{\em Proof of Theorem \ref{local contraction}:}
Let $V$ be the complex (or real) vector bundle provided by Lemma \ref{localization}.
We know  that $\Vert \ms D^f_{\sigma_0(x)}F\Vert <1$ for all $x$ in $X$.
After further restraining $U$ and choosing $r$ small enough,
we may assume by continuity that for all $y$ in $B(r)$, $\Vert \ms D^f_y F\Vert < K<1$. 

After further restricting $U$, we may assume that for all $u\in U$ and $x\in X$, we have
$$\Vert F(\sigma(u,x))-\sigma(u,f(x))\Vert \leq (1-K)r,$$
In particular, if $y\in B(r)$ is in the fiber over $(u,x)$, 
\begin{eqnarray*}
\Vert F(y)-\sigma(u,f(x))\Vert &\leq&\Vert F(y)-F(\sigma(u,x))\Vert\cr & & +\Vert F(\sigma(u,x))-\sigma(u,f(x))\Vert\cr
 &\leq& Kr +  (1-K)r=r.
\end{eqnarray*}

Thus $F$ maps $B(r)$ to itself and is contracting. We can therefore apply Theorem \ref{lem:coc} to complete the proof
of Theorem \ref{local contraction}.
\qed

\medskip

In the proof of Theorem \ref{Comega}, we will also need to use the fact
that transverse regularity of a continuous function \hbox{$f:D\times X\to M$} 
implies regularity of the associated map of $D$ into \hbox{$C^\alpha(X,M)$}.  

Let $X$ be a compact metric space and let $M$ be a complex analytic (or $C^k)$ manifold. 
If $U$ is an open subset of $M$ and $V$ is a relatively compact open subset of $X$, then let 
$$\mathcal W(U,V)=\{g\in C^\alpha(X,M)\mid \overline{g(V)}\subset U\}.$$ 
We will say that  a map $f$ from $D$ to $C^\alpha(X,M)$ is {\em complex analytic (or $C^k$)} if for any
$U$ and $V$ as above and any complex analytic  function  $\phi:U\to\mathbb C$
(or $C^k$ function $\phi:U\to \mathbb R$), the function $f^\phi$ defined on  $f^{-1}\left(\mathcal W(U,V)\right)$, by
$$
f^\phi(x)=\phi\circ f(x)|_V,
$$
with values in $C^\alpha(V,\mathbb C)$ (or $C^\alpha(V,\mathbb R)$) is complex analytic (or $C^k$). 
Recall that the function $f^\phi$ is complex analytic if and only if it has a  a $\CC$-linear differential at each point,
see, for example, Hubbard \cite[Thm. A5.3]{hubbard}. 

The following lemma shows that an $\alpha$-H\"older transversely complex analytic map from $D\times X$ to $M$
gives rise to a complex analytic map from $D$ to $C^\alpha(X,M)$.
The proof is quite standard so we will omit it,
see Hubbard \cite[Prop. A5.9]{hubbard} for a very similar  statement.

\begin{lemma}
\label{inversion} 
Suppose that $D$ is a complex (or real) disk, $M$ is a complex analytic (or $C^k$) manifold,
$X$ is a compact metric space and
\hbox{$f:D\times X\to M$} is $\alpha$-H\"older transversely  complex analytic (or $C^k$),
then the map $\hat f$ from $D$ to $C^\alpha(X,M)$ 
given by $u\to f_u$ where $f_u(\cdot)=f(u,\cdot)$ is complex analytic (or $C^{k-1}$).
\end{lemma}

\subsection{Analytic variation of the limit maps} 
\label{analytic variation of limit maps}

We are now ready for the proof of Theorem \ref{Comega} in the complex analytic case.
Given a complex analytic family of
representations  which contains an Anosov representation, we construct an associated bundle where
we can apply the results of the previous section to produce a family of limit maps.

Let $G$ be  a complex Lie group and let $P$ be a parabolic subgroup.
Let $\{\rho_u\}_{u\in D}$ be a complex analytic family of homomorphisms of $\Gamma$ into
$\ms G$ parameterized by a complex disk $D$ about $0$ so that $\rho_0$ is $(\ms G,\ms P)$-Anosov.

We  construct a $\ms G/\ms P$-bundle  over $D\times \Gg$. Let 
$$\tilde A=D\times \widetilde \Gg\times \ms G/\ms P$$
which is a $\ms G/\ms P$-bundle  over $D\times \widetilde{\Gg}$.
Then $\gamma\in\Gamma$ acts on $\tilde A$, by 
$$\gamma(u,x,[g])=(u,\gamma(x),[\rho_u(\gamma)g])$$ 
and we let
$$A=\tilde A/\Gamma.$$

The geodesic flows on $\widetilde \Gg$ and $\Gg$ lift to geodesic flows $\{\tilde\Psi_t\}_{t\in\mathbb R}$
and   $\{\tilde\Psi_t\}_{t\in\mathbb R}$ on $\tilde A$ and $A$.
(The flow  $\{\tilde\Psi_t\}_{t\in\mathbb R}$ acts trivially on the
$D$ and $\ms G/\ms P$ factors.)

Since $\rho_0$ is $(\ms G,\ms P)$-Anosov 
there exists a section $\sigma_0$ of $A$ over $\{0\}\times \Gg$.
Concretely, if $\xi_0:\bg \to \ms G/\ms P$  is the limit map, we construct an
equivariant section
$\tilde\sigma_0$ of $\tilde A$ over $\{0\}\times\widetilde{\Gg}$ of
the form
$$(0,(x,y,t))\to (0,(x,y,t),\xi_0(x)).$$ The section $\tilde\sigma_0$ descends to the desired section $\sigma_0$ of
$A$ over $\{0\}\times\Gg$. One may identify the bundle over $\{0\}\times \Gg$ with fiber
$\ms T_{\sigma_0(x)}\pi^{-1}(0,x)$ with $\mathcal{N}_\rho^-$.
Since the geodesic flow lifts to a flow  on $\mathcal{N}_\rho^-$ whose inverse flow is contracting,
the  inverse flow  $\{\Phi_{-t}\}_{t\in\mathbb R}$ is contracting along $\sigma_0(\Gg)$.

Theorem \ref{local contraction}  then implies that there exists a sub-disk 
$D_1\subset D$ containing $0$, $\alpha>0$, and
an $\alpha$-H\"older transversely  complex analytic  section 
$\eta:D\times\Gg\to A$ that extends $\sigma_0$, is fixed by $\{\Phi_t\}_{t\in\mathbb R}$ and
so that  the inverse flow $\{\Phi_{-t}\}_{t\in\mathbb R}$ contracts along $\eta$.
(More concretely, Theorem \ref{local contraction} produces, for large enough $t_0$, a section fixed by $\Phi_{-t_0}$
so that $\Phi_{-t_0}$ contracts along $\eta$. One may then use the uniqueness portion of the statement to
show that $\eta$ is fixed by $\Phi_t$ for all $t$.)
We may lift $\eta$ to a section 
\hbox{$\tilde\eta:D_1\times\widetilde{\Gg}\to\tilde A$}
which we may view as a map \hbox{$\bar\eta:D_1\times\widetilde{\Gg}\to\ms G/\ms P$}.

We next observe that $\bar\eta(u,(x,y,t))$ does not depend on either $y$ or $t$.
Since $\bar\eta$ is flow-invariant, $\bar\eta(u,(x,y,t))$ does not depend on $t$. Fix $u\in D_1$ and let
\hbox{$\bar\eta_u:\widetilde{\Gg}\to \ms G/\ms P$} be given by $\bar\eta_u(\cdot)=\bar\eta(u,\cdot)$.
Let $\gamma$ be an infinite order element of $\Gamma$ whose associated orbit in $\Gg$ has period $t_\gamma$
and let $d$ be an arbitrary metric on $\ms{G}/\ms{P}$. Since $\{\Phi_{-t}\}_{t\in\mathbb R}$ is contracting along $\eta$, 
there exists a constant $k_0>0$ such that if  $\{p_n\}$ is a sequence in $\mathsf{G}/\mathsf{P}$ with 
$d(\bar\eta_u(\gamma^+,\gamma^-,0),p_n)\le k_0$ for all $n$, then
\begin{equation}
\label{contraction property}
\lim_{n\to \infty}d(\bar\eta_u(\gamma^+,\gamma^-,0),\gamma^{n}(p_n))=
\lim_{n\to \infty}d(\bar\eta_u(\gamma^n(\gamma^+,\gamma^-,-nt_\gamma)),\gamma^{n}(p_n))=0.
\end{equation}
Given $z\in\bg$, there exists $t_z\in R$, so that, if $\bar d$ denotes  a $\Gamma$-invariant metric on $\widetilde{\Gg}$,
then 
$$\lim_{n\to\infty}\bar d(\gamma^n(\gamma^+,\gamma^-,0),(\gamma^+,z,t_z+n t_\gamma))=0.$$
Therefore, 
$$\lim_{n\to\infty}\bar d((\gamma^+,\gamma^-,0),\gamma^{-n}(\gamma^+,z,t_z+n t_\gamma))=0.$$
Applying (\ref{contraction property}) with $p_n=\bar\eta(\gamma^{-n}(\gamma^+,z,t_z+nt_\gamma))$, we see that
$$
\lim_{n\to\infty} d(\bar\eta_u(\gamma^+,\gamma^-,0),\gamma^{n}\bar\eta_u(\gamma^{-n}(\gamma^+,z,t_z+nt_\gamma)))=0.
$$
Since $\bar\eta_u$ is $\Gamma$-equivariant, this implies that
$$
\lim_{n\to\infty} d(\bar\eta_u(\gamma^+,\gamma^-,0),\bar\eta_u(\gamma^+,z,t_z+n t_\gamma))=0.
$$
Since $\bar\eta_u(\gamma^+,z,t)$ does not depend on $t$, we finally obtain that
$$\bar\eta(u,(\gamma^+,\gamma^-,0))=\bar\eta(u,(\gamma^+z,t))$$
for any $z\in\bg$, $u\in D_1$ and $t\in\mathbb R$. Since, fixed points of infinite order elements are dense in $\bg$
and $\bar\eta$ is continuous, we see that  $\bar\eta(u,(x,y,t))$ does not depend on $y$ or $t$.

Therefore, we obtain a transversely complex analytic map
$$\xi:D_1\times \bg\to \ms G/\ms P$$
which extends $\xi_0$.  The map $\xi$ satisfies properties (2) and (3), since
$\xi$ is $\alpha$-H\"older transversely complex analytic, while property
(4) follows from Lemma \ref{inversion}. 

It remains to prove that we may restrict to a sub disk $D_0$ of $D_1$ so that
if $u\in D_0$, then $\rho_u$ is $(\ms G,\ms P)$-Anosov
with limit map $\xi_u$. Let $\ms Q$ be a parabolic subgroup of $\ms G$ which is opposite
to $\ms P$. 
Then there exists a Lipschitz transversely complex analytic
$\ms G/\ms Q$-bundle $A'$ over $D\times\Gg$ and we may lift
the geodesic flow to a flow $\{\Phi_t'\}$ on $A'$. Since $\rho_0$ is $(\ms G,\ms P)$-Anosov,
there exists a map $\theta_0:\bg\to \ms G/\ms Q$ which gives
rise to a section $\sigma_0'$ of $A'$ over $\{0\}\times\Gg$ such that  the flow is contracting on a neighborhood
of $\sigma_0'(\{0\}\times\Gg)$. We again apply Corollary \ref{local contraction} 
to find
an $\alpha'$-H\"older (for some $\alpha'>0$)  transversely  complex analytic   flow invariant section 
$\eta':D_2\times\Gg\to A'$ that extends $\sigma_0'$, for some sub-disk $D_2$ of
$D$ which contains $0$, such that the flow $\{\Psi_t'\}_{t\in\mathbb R}$ contracts along
$\eta'(D_2\times\Gg)$.
The section $\eta'$ lifts to a section of  $\tilde\eta'$ of $\tilde A'$ which
we may reinterpret as a map
$\bar\eta':D_2\times \widetilde\Gg\to\ms G/\ms Q$ 
so that $\bar\eta'(u,(x,y,t))$ depends only on $u$ and $y$. So we obtain an $\alpha'$-H\"older transversely
complex analytic map
$$\theta:D_2\times\bg\to \ms G/\ms Q$$
which restricts to $\theta_0$. Since $\xi_0$ and $\theta_0$ are transverse, we may find a sub-disk
$D_0$ of $D_1\cap D_2$ so that $\xi_u$ and $\theta_u$
are transverse if $u\in D_0$. It follows that if $u\in D_0$, then $\rho_u$ is
$(\ms G,\ms P)$-Anosov with limit maps $\xi_u$ and $\theta_u$.
This completes the proof of Theorem \ref{Comega} in the complex analytic case. 

{\bf Remark:} Notice that the same proof applies to a $C^k$-family $\{\rho_u\}_{u\in D}$ of
representations of a hyperbolic group $\Gamma$ into a real semi-simple Lie group $\ms G$ 
such that $\rho_0$ is $(\ms G,\ms P)$-Anosov. It produces a sub-disk $D_0$ and a 
$\alpha$-H\"older transversely $C^k$ map
\hbox{$\xi:D_0\times\bg\to\ms G/\ms P$} so that if $u\in D_0$, then $\rho_u$ is $(\ms G,\ms P)$-Anosov
with limit map $\xi_u$.

\subsection{Analytic variation of the reparameterization}
\label{analytic variation of reparam}

We now turn to the proof of Proposition \ref{Creparam}.

Let $\{\rho_u:\Gamma\to\sln\}_{u\in D}$ be a real analytic family of projective Anosov representations and
let $D^{\mathbb C}$ be the complexification of $D$. 
We may extend $\{\rho_u\}_{u\in D}$ to
a complex analytic family $\{\rho_u:\Gamma\to\ms{SL}_m(\mathbb C)\}_{u\in D^{\mathbb C}}$ of
homomorphisms. Theorem \ref{Comega} implies that, after possibly restricting $D^{\mathbb C}$,
there exists a $\alpha$-H\"older transversely complex analytic map
$$\xi:D^{\mathbb C}\times\bg\to \ms G^{\mathbb C}/\ms P^{\mathbb C}=\mathbb{CP}(m)$$
such that if $u\in D^\mathbb C$, then $\rho_u$
is Anosov with respect to the parabolic subgroup $\ms P^{\mathbb C}$, which
is the stabilizer of a complex line, with limit map $\xi_u$.
(We call such representations {\em complex projective Anosov}.)

We  construct a Lipschitz transversely complex analytic $\mathbb C^m$-bundle $W^\mathbb C$ over
$D^{\mathbb C}\times\Gg$ which is the quotient of 
$\tilde W^\mathbb C=D^\mathbb C\times\widetilde{\Gg}\times \mathbb C^m$
associated to the family $\{\rho_u\}_{u\in D^\mathbb C}$.
We  can then lift the Gromov geodesic flow on $\Gg$ to a Lipschitz transversely complex analytic flow 
$\{\Psi_t\}_{t\in\mathbb R}$ on $W^\mathbb C$. 
Since the functions in the partition of unity for our trivializations of $W^\mathbb C$ 
are constant in the the $D^\mathbb C$ direction,  we have:

\begin{proposition} 
After possibly further restricting $D^\mathbb C$, 
the bundle $W^\mathbb C$ is equipped with a Lipschitz transversely   complex analytic  2-form $\omega$ of type $(1,1)$
such that 
$$
\tau(u,v)=\omega(u,v)+\overline{\omega(v,u)},
$$
is Hermitian.
\end{proposition}

Let $L^\mathbb C$ be the (complex) line sub-bundle of $W^\mathbb C$ determined by $\xi$,
i.e.  $L^\mathbb C$ is the quotient of the line sub-bundle of $\tilde W^\mathbb C$ whose
fiber over $(u,(x,y,t))\in D^\mathbb C\times\widetilde{\Gg}$ is the  complex line $\xi_u(x)$.
Then, $L^\mathbb C$ is a  $\alpha$-H\"older transversely  complex analytic line bundle over 
$D^\mathbb C\times \Gg$. Since each $\rho_u$ is complex projective Anosov with limit map $\xi_u$,
$L^\mathbb C$ is preserved by the flow $\{\Psi_t\}\tR$.
We restrict $\omega$ and $\tau$  to $ L^\mathbb C$ (and still denote them by $\omega$ and $\tau$).

Since $L^{\mathbb C}$ is a line bundle, we can consider the function 
$$a:D^{\mathbb C}\times\Gg\to \mathbb C$$
such that
$$
\omega(u,x)(v,v)=a(u,x)\tau(u,x)(v,v).
$$
whenever $v$ is in the fiber of $L^\mathbb C$ over $(u,x)$.
Concretely.
$$
a(u,x)=\frac{\omega(v,v)}{2\Re(\omega(v,v))} 
$$
for any non-trivial $v$ in the fiber over $(u,x)$.

We observe that $a$ is $\alpha$-Holder transversely real analytic.
If $U$ is an open subset of $\Gg$ in one of our trivializing sets, we can construct
a non-zero section 
$$V:D^{\mathbb C}\times U\to L^\mathbb C$$  
which is $\alpha$-H\"older transversely complex analytic. 
Then
$$\omega(V,V):D^{\mathbb C}\times U\to \mathbb C$$
is $\alpha$-H\"older transversely complex analytic.
Lemma \ref{inversion} implies that
the map from $D^{\mathbb C}$ to $C^\alpha(U,\mathbb C)$ given by $u\to \omega(V(u,\cdot),V(u,\cdot))$ is complex analytic.
Therefore,  the map from $D^\mathbb C$ to $C^\alpha(U,\mathbb R)$ given by 
$u\to \Re(\omega(V(u,\cdot),V(u,\cdot)))$ is real analytic.
It follows that the map from $D^{\mathbb C}$ to $C^\alpha(U,\mathbb C)$ given by $u\to a(u,\cdot)$ is
real analytic since
$$a|_{D^{\mathbb C}\times U}=\frac{\omega(V,V)}{2\Re(\omega(V,V))}.$$
Since $x$ was arbitrary the map from $D^{\mathbb C}$ to $C^\alpha(\Gg,\mathbb C)$ given by
$u\to a(u,\cdot)$ is real analytic. Similarly, $a$ itself is $\alpha$-Holder transversely real analytic.

If we define, for all $t$, the map
$$h_t:D^{\mathbb C}\times\Gg\to \mathbb C$$
so that 
$$
\Psi_t^*\omega=h_t\omega,
$$
then, we may argue, just as above, that $h_t$
is $\alpha$-H\"older transversely complex analytic.
Lemma \ref{inversion} guarantees that the map
from $D^{\mathbb C}$ to $C^\beta(\Gg,\mathbb C)$ given by $u\to h_t(u,\cdot)$ is complex analytic.

If $t\in\mathbb R$,
$$
\Psi_t^*\tau(\cdot)=2\Re(\Psi_t^*\omega(\cdot))=2\Re(h_t(\cdot)\omega(\cdot))=2\Re\left(a(\cdot)h_t(\cdot)\right) G(\cdot).
$$
We define $k_t(\cdot)=\Re(a h_t)(\cdot)$ and note that $\Psi_t^*\tau=k_t\tau$.
Then,  $k_t$ is $\alpha$-H\"older transversely real analytic and the map from $D^{\mathbb C}$ to 
$C^\alpha(\Gg,\mathbb R)$ given
by $u\to k_t(u,\cdot)$ is real analytic (since it is the real part of a product of  a real analytic and a complex analytic function).

We apply the construction of Lemma \ref{XiHolder} to
produce an $\alpha$-H\"older transversely real analytic metric $\tau^0$ on $\hat L$
such that 
$$
\Psi_t^*(\tau^0)<  e^{-\beta t}\tau^0.
$$
for some $\beta>0$ and all $t>0$.
Concretely,
$$\tau^0=\int_{0}^{t_0}e^{\beta s}\Psi_s^*(\tau)\d s = \left(\int_{0}^{t_0}e^{\beta s}k_{s}\ \d s\right)G$$
for some appropriately chosen $t_0>0$.

We define, for all $t$, $K_t:D^{\mathbb C}\times\Gg\to \mathbb R$ by
$$
K_t=e^{-\beta t}\frac{\int_{t}^{t_0+t}e^{\beta s}k_{s}\ \d s}{\int_{0}^{t_0}e^{\beta s}k_{s}\ \d s}.
$$ 
One then checks that
$$
\Psi_t^*(\tau^0)=K_t \tau^0
$$
for all $t$.
Then, for each $u\in D^\mathbb C$ we define $f_u:\Gg\to\mathbb R$, by setting
$$
f_u(\cdot)=\frac{\partial K_t}{\partial t}(u,\cdot,0)=-\beta+\frac{e^{\beta t_0}k_{t_0}(\cdot)-1}{\int_{0}^{t_0}e^{\beta s}k_{s}(\cdot)\ \d s}.
$$
Then, since $u\to k_t(u,\cdot)$ is real analytic for  all $t$, our formula for $f_u$ guarantees  that 
the map from $D^\mathbb C$ to $C^\beta(\Gg,\mathbb R)$ given by $u\to f_u$  is real analytic.
Therefore, the restriction of this map to the real submanifold $D$ is also real analytic.

To complete the proof of Proposition \ref{Creparam} we will show that, for each $u\in D$,
the periods of the reparameterization of $\Gg$ by $f_u$ and the periods of $U_{\rho_u}\Gamma$ agree.
Liv\v sic's Theorem \ref{theorem:livsic} then implies that the reparameterization
of $\Gg$ by $f_u$ is H\"older conjugate to $U_{\rho_u}\Gamma$ as desired.

For $u\in D$, let $j_u:\Gg\times\mathbb R$ be given by $j_u(\cdot, t)=\log K_t(u,\cdot)$.
We can differentiate the equality
$$
j_u(\cdot, t+s)=j_u(\Psi_s(\cdot),t)+j_u(\cdot,s)
$$
with respect to $t$ and evaluate at $t=0$ to conclude that
$$
f_u(\cdot, s)=f_u(\Psi_s(\cdot),0).
$$
In particular, for any $t$,
$$
\int_0^t (f_u(\Psi_s(\cdot),0)\, \d s=j_u(\cdot, t).
$$

Let $\gamma\in\Gamma$ and let $x\in\Gg$ be a point on the periodic orbit associated to $\gamma$
(which is simply the quotient of $(\gamma^+,\gamma^-)\times\mathbb R\subset\widetilde \Gg$).
If $t_\gamma$ is the period of the orbit  of $\Gg$ containing $x$, then
$$
e^{\int_0^{t_\gamma} f_u(\Psi_s(u,x))\d s} \tau^0(x,u)=\Psi^*_{t_\gamma}\tau^0(u,x)=e^{\Lambda(\rho_u,\gamma)}\tau^0(u,x),
$$
so 
$$\int_0^{t_\gamma} f_u(\Psi_s(u,x))\d s=\Lambda(\rho_u,\gamma)$$
is the period of the reparameterization of the flow $\Gg$ by $f_u$, which agrees with the period
of the orbit in $U_{\rho_u}\Gamma$  associated to $\gamma$
(see Proposition \ref{geoflowreparam}). This completes the proof of Proposition \ref{Creparam}.

\medskip\noindent
{\bf Remark:} Notice that a simpler version of the above proof establishes that given a $C^k$ family
of projective Anosov representations, one may, at least locally, choose the reparameterization functions to vary $C^{k-1}$.

\section{Deformation  spaces of projective Anosov representations}
\label{MCR}

In this section, we collect a few facts about the structure of  deformation  spaces of projective 
Anosov representations of $\Gamma$ into $\sln$ and their relatives.

\subsection{Irreducible  projective Anosov representations}

We first observe that our deformation spaces $\Cm$ and $\Cg$
are  real analytic manifolds. Let 
\hbox{$\widetilde{\mathcal{C}}(\Gamma,m)\subset \hom(\Gamma,\sln)$}
denote the set of regular, irreducible, projective Anosov representations
and let
$$\Cm=\widetilde{\mathcal{C}}(\Gamma,m)/\sln.$$
If $\ms G$ is a reductive subgroup of $\sln$, then we similarly let 
\hbox{$\widetilde{\mathcal{C}}_g(\Gamma,\ms G)\subset \hom(\Gamma,\ms G)$}
denote the space of $\ms G$-generic, regular representations which are  irreducible and projective Anosov 
when viewed as representations into $\sln$. Let
$$\Cg=\widetilde{\mathcal{C}}_g(\Gamma,\ms G)/\ms G.$$

\begin{proposition}
\label{convexsmooth}
Suppose that $\Gamma$ is a word hyperbolic group. Then
\begin{enumerate}
\item The deformation spaces $\Cm$ and 
$\mathcal{C}_g(\Gamma,\ms{SL}_m(\mathbb R))$  have the structure of a real analytic manifold compatible 
with the algebraic structure  on $\hom(\Gamma,\ms{SL}_m(\mathbb R))$
\item
If $\ms G$ is a reductive subgroup of $\sln$, then $\Cg$ has the structure of a real analytic manifold compatible
with  the algebraic structure  on $\hom(\Gamma,\ms G)$.
\end{enumerate}
\end{proposition}

\begin{proof}  
We may regard $\hom(\Gamma,\sln)$ as a subset of $\hom(\Gamma,\ms{SL}_m(\mathbb C))$.
We first notice that an irreducible homomorphism in $\hom(\Gamma,\sln)$ is also irreducible when regarded as
a homomorphism  in  $\hom(\Gamma,\ms{SL}_m(\mathbb C))$.
Lubotzky and Magid (\cite[Proposition 1.21 and Theorem 1.28]{lubotzky-magid}) proved that the set of irreducible
homomorphisms form an open subset of  $\hom(\Gamma,\ms{SL}_m(\mathbb C))$, so they also form an open
subset of $\hom(\Gamma,\sln)$. Results of Labourie \cite[Prop. 2.1]{labourie-anosov} and 
Guichard-Wienhard \cite[Theorem 5.13]{guichard-wienhard}
imply that the set  of projective Anosov homomorphisms is an open subset  of  $\hom(\Gamma,\sln)$ 
(see also Proposition \ref{Comega}).
Therefore, $\widetilde{\mathcal{C}}(\Gamma,m)$ is an open subset of $\hom(\Gamma, \sln)$. Since the former consists of
regular homomorphisms, it is an analytic manifold.

Lubotzky--Magid  (\cite[Theorem 1.27]{lubotzky-magid}) also proved that $\mathsf{SL}_m(\mathbb C)$ acts properly
(by conjugation) on the set of irreducible representations in $\hom(\Gamma,\ms{SL}_m(\mathbb C))$. It  follows that $\sln$
acts properly on  $\widetilde{\mathcal{C}}(\Gamma,m)$. Schur's Lemma guarantees that the centralizer of an irreducible representation is
contained in the center of $\sln$. Therefore, $\ms{PSL}_m(\mathbb R)$  acts freely, analytically and properly on 
the analytic manifold $\widetilde{\mathcal{C}}(\Gamma,m)$, so its quotient $\Cm$ is also an analytic manifold.

Since the set of $\ms G$-generic elements of $\ms G$ is an open $\ms G$-invariant subset of $\ms G$, we may argue
exactly as above to show that $\widetilde{\mathcal{C}}_g(\Gamma,\ms G)$ is an open subset of $ \hom(\Gamma,\ms G)$
which is an analytic manifold. The action of $\ms G/Z(\ms G)$ on $\widetilde{\mathcal{C}}_g(\Gamma,\ms G)$ is again free, analytic and
proper, so its quotient $\Cg$ is again an analytic manifold.
\end{proof}

If $\rho \in \widetilde{\mathcal{C}}(\Gamma,m)$, then one may identify $\TT_\rho \widetilde{\mathcal{C}}(\Gamma,m)$
with the space $Z^1_\rho(\Gamma,\mk{sl}_m(\mathbb R))$ of cocycles and one may  then
identify $\TT_{[\rho]}\Cm$ with the cohomology group $H^1_\rho(\Gamma,\mk{sl}_m(\mathbb R))$
(see \cite{lubotzky-magid,johnson-millson}). In particular, the space $B^1_\rho(\Gamma,\mk{sl}_m(\mathbb R) )$ is identified with the tangent
space of the $\sln$-orbit of $\rho$.
Similarly, if $\rho\in \widetilde{\mathcal{C}}_g(\Gamma,\ms G)$, we identify
$\TT_\rho \widetilde{\mathcal{C}}_g(\Gamma,\ms G)$ with $Z^1(\Gamma,\mk g)$ and
$\TT_{[\rho]}\Cg$ with $H^1_\rho(\Gamma,\mk g)$. More generally, if $\rho$ is an irreducible 
representation in $\hom(\Gamma,\ms G)$, the tangent vector to any analytic path through
$\rho$ may be identified with an element of $Z^1_\rho(\Gamma,\mk g)$ (see \cite[Section 2]{johnson-millson}).

A simple calculation in cohomology gives that  irreducible projective Anosov representations of fundamental groups of
3-manifolds with non-empty boundary are regular. These include free groups and fundamental groups of closed
surfaces.

\begin{proposition}
\label{3-manifold case}
If $\Gamma$  is isomorphic to the fundamental group of a compact  orientable 3-manifold $M$ with non empty
boundary, then $\mathcal C(\Gamma,m)$ is the set of conjugacy classes of irreducible projective Anosov representations.
\end{proposition}

\begin{proof}
Let $\Gamma=\pi_1(M)$ where $M$ is a compact orientable 3-manifold with non-empty boundary.
It suffices to show that the open subset of $\hom(\Gamma,\sln)$ consisting of  irreducible projective Anosov homomorphisms
consists entirely of regular points. We recall that $\rho_0\in \hom(\Gamma,\sln)$ is regular if there exists a neighborhood
$U$ of $\rho_0$ so that  $\dim(Z^1_\rho(M,\mk g))$   is constant on $U$  and the centralizer of any representation $\rho\in U$ is 
trivial  \cite{lubotzky-magid}.  

If $\rho_0$ is projective Anosov and irreducible, we
can take $U$ to be any open neighborhood of $\rho_0$ consisting of 
irreducible projective Anosov representations.
Since $\rho\in U$ is irreducible, Schur's Lemma guarantees that  the centralizer of $\rho(\Gamma)$ is  the center of $\sln$.
Moreover, if $\rho\in U$, then
$$
\dim(H^0_\rho(M,\mk g))-\dim(H^1_\rho(M,\mk g))+\dim(H^2_\rho(M,\mk g))=\chi(M)\dim(\mathsf G).
$$ 
Since the centralizer is trivial, $\dim(H^0_\rho(M,\mk g))=0$. By Poincar\'e duality,
$\dim(H^2_\rho(M,\mk g))=\dim(H^0_\rho(M,\partial M,\mk g))$. 
Since $\dim(H^0_\rho(M,\mk g))=0$, the long exact sequence for relative homology implies that 
\hbox{$\dim(H^0_\rho(M,\partial M,\mk g))=0$}. Thus,
$$
\dim(H^1_\rho(M,\mk g))=-\chi(M)\dim(\mathsf G).
$$ 
Therefore,
$\dim(Z^1_\rho(M,\mk g))=(1-\chi(M))\dim(\mathsf G)$
for all $\rho\in U$, so $\rho$ is a regular point.
\end{proof}

\subsection{Virtually Zariski dense representations} 

We recall that if $\Gamma$ is a word hyperbolic group, $\ms G$ is a semi-simple Lie group
with finite center and $\ms P$ is a non-degenerate parabolic subgroup, then $\mathcal{Z}(\Gamma;\ms G,\ms P)$
is the space of (conjugacy classes of) regular virtually Zariski dense $(\ms G,\ms P)$-Anosov representations of 
$\Gamma$ into $\ms G$.  We will prove that $\mathcal{Z}(\Gamma;\ms G,\ms P)$ is a real analytic orbifold.

\begin{proposition}
\label{zd variety}
Suppose that $\Gamma$ is a word hyperbolic group, $\ms G$ is a  semi-simple Lie group with finite center and $\ms P$ is
a non-degenerate parabolic subgroup of $\ms G$. Then
$\mathcal{Z}(\Gamma;\ms G,\ms P)$ is a real analytic orbifold. 

Moreover, if $\ms G$ is connected,
then $\mathcal{Z}(\Gamma;\ms G,\ms P)$ is a real analytic manifold.
\end{proposition}

\begin{proof} 
Let $\hom^*(\Gamma,\ms G)$ be the set of regular homomorphisms. By definition,
$\hom^*(\Gamma,\ms G)$  is an open subset of $\hom(\Gamma,\ms G)$ and hence it is an analytic manifold, 
since it is the set of smooth points of a real algebraic variety.
Results of Labourie \cite[Prop. 2.1]{labourie-anosov} and 
Guichard-Wienhard \cite[Theorem 5.13]{guichard-wienhard} again imply that the set of
$(\ms G,\ms P)$-Anosov homomorphisms is open in $\hom^*(\Gamma,\ms G)$.  The main difficulty in
the proof is to show that the set $\widetilde{\mathcal{Z}}(\Gamma;\ms G,\ms P)$
of virtually Zariski dense Anosov homomorphisms is open in $\hom^*(\Gamma,\ms G)$ and hence an analytic manifold.

Once we have shown that $\widetilde{\mathcal{Z}}(\Gamma;\ms G,\ms P)$ is an analytic manifold, we may
complete the proof in the same spirit as the proof of Proposition \ref{convexsmooth}. We observe that
if $\rho\in\widetilde{\mathcal{Z}}(\Gamma;\ms G,\ms P)$ then its centralizer is finite, 
since the Zariski closure of $\rho(\Gamma)$ has
finite index in $\ms G$. Then,  $\ms G/Z(\ms G)$ acts properly and analytically on
$\widetilde{\mathcal{Z}}(\Gamma;\ms G,\ms P)$ with finite point stabilizers, so the quotient 
$\mathcal{Z}(\Gamma;\ms G,\ms P)$ is
an analytic orbifold. If $\ms G^0$ is the connected component of $\ms G$, 
then the Zariski closure of any representation $\rho\in\widetilde{\mathcal{Z}}(\Gamma;\ms G,\ms P)$
contains $\ms G^0$, so the intersection of the centralizer of $\rho$ with $\ms G^0$ is simply
$Z(\ms G)\cap\ms G^0.$ Therefore, 
$\widetilde{\mathcal{Z}}(\Gamma;\ms G,\ms P)/\ms G^0$ is an analytic manifold.
In particular, if $\ms G$ is connected, 
$\mathcal{Z}(\Gamma;\ms G,\ms P)$ is an analytic manifold.

We complete the proof by showing that the 
set of virtually Zariski dense $(\ms G,\ms P)$-Anosov homomorphisms is open in $\hom^*(\Gamma,\ms G)$. 
If not, then there exists a sequence $\seq{\rho}$ of  $(\ms G,\ms P)$-Anosov representations  which are not
virtually Zariski dense converging to a virtually Zariski dense $(\ms G,\ms P)$-Anosov representation $\rho_0$. 

Since $\ms G$ has finitely many components, $\rho_n^{-1}(\ms G^0)$ has bounded finite index for all $n$.
Since $\Gamma$ is finitely generated, it contains only finitely many subgroups of a given index, so
we may pass to a finite index subgroup $\Gamma_0$ of $\Gamma$ so that $\rho_n(\Gamma_0)$ is contained
in the identity component $\ms G^0$ of $\ms G$ for all $n$. Since each $\rho_n|_{\Gamma_0}$ is 
$(\ms G,\ms P)$-Anosov 
and $\rho_0(\Gamma_0)$ 
is also virtually Zariski dense,
we may assume for the remainder of the proof that $\ms G$ is the Zariski closure of $\ms G^0$.
 
Let $\ms Z_n$ be the Zariski closure of $\im(\rho_n)$ and let $\mathfrak z_n$ be the  Lie algebra of $\ms Z_n$.
Consider the decomposition of the Lie algebra $\mk g$ of $\ms G$
$$\mk g=\bigoplus_{i=1}^p \mk g_i,$$
where $\mk g_i$ are simple  Lie algebras. 
Let   $\ms G_i=\operatorname{Aut}(\mk g_i).$ We consider the adjoint representation $\ad:\ms G\to \operatorname{Aut}(\mk g)$.
Let $\ms H$ be the subgroup of $\ms G$ consisting of all $g\in\ms G$ so that $\ad(g)$ preserves the factors of $\mk g$.
Then $\ms H$ is a finite index, Zariski closed subgroup of $\ms G$. Hence, with our assumptions, $\ms H=\ms G$.
Therefore, we get a well-defined projection map $\pi_i:\ms G \to \ms G_i$. If $\mk p$ is the Lie algebra of $\ms P$,
then $\mk p=\bigoplus_{i=1}^p \mk p_i,$ where $\mk p_i$ is a Lie subalgebra of $\mk g_i$.
Let  $\ms P_i$ be the stabilizer of $\mk p_i$ in $\ms G_i$. 
Then we also obtain a $\ms G$-equivariant projection, also denoted $\pi_i$,
$$\pi_i:\ms G/\ms P \to \ms G_i/\ms P_i=\ms G\mk p_i\subset\operatorname{Gr}_{\dim(\mk p_i)}(\mk g_i)$$
where  $\operatorname{Gr}_{\dim(\mk p_i)}(\mk g_i)$ is the Grassmanian space of  $\dim(\mk p_i)$-dimensional vector spaces in $\mk g_i$.

If $\xi_n:\bg\to \ms G/\ms P$ is the limit map of $\rho_n$, $\pi_i\circ\xi_n$ is a $\rho_n$-equivariant map from $\bg$ to $\ms G_i/\ms P_i$. 
If $\pi_i\circ\xi_n$ is constant, then $\rho_n(\Gamma)$ would normalize a conjugate of $\mk p_i$. So, if 
$\pi_i\circ\xi_n$ is constant for infinitely many $n$, then $\rho_0(\Gamma)$ would normalize  a conjugate of $\mk p_i$,which is
impossible since $\rho_0(\Gamma)$ is Zariski dense and $\ms P_i$ is a proper parabolic subgroup of $\ms G_i$.
Therefore, we may assume that $\pi_i\circ \xi_n$ is non-constant for all $i$ and all $n$.
Since $\Gamma$ acts topologically transitively on $\bg$, we then know that the image must then be infinite.
Therefore, for all $i$ and $n$,
\begin{equation}
\dim(\pi_i(\mk z_n))>0.\label{eq:vzd2} 
\end{equation}

We may thus assume that $\{{\mathfrak z}_n\}$ converges  to a proper  Lie subalgebra $\mk z_0$ which is normalized by $\rho_0(\Gamma)$ with \begin{equation}
\dim(\mk z_0)>0.
\end{equation}
Since $\rho_0$ is virtually Zariski dense, $\mk z_0$ must be a strict factor in the Lie algebra $\mk g$ of $\ms G$.
Thus, after reordering, we may assume that 
\begin{equation}
\mk z_0=\bigoplus_{i=1}^q \mk g_i.\label{eq:vzd1} 
\end{equation}
For $n$ large enough, $\mk z_n$ is thus a graph  of an homomorphim 
$$
f_n:\mk z_0\to \mk h=\bigoplus_{i=q+1}^p \mk g_i.
$$

Since there are only finitely many  conjugacy classes (under the adjoint representation) of 
homomorphisms of $\mk z_0$ into $\mk h$,
we may  pass to a subsequence such that
$$
f_n=\ad (g_n)\circ f_0\circ \pi_{\mk h_1},
$$
where $f_0$ is a fixed isomorphism from an ideal $\mk h_1$ in $\mk z_0$ to an ideal  $\mk h_2$ in $\mk h$, 
$\pi_{\mk h_1}$ is the projection from $\mk z_0$ to $\mk h_1$ and $g_n\in \ms H_2$ where $\ms H_i$ is the subgroup  of $\ms G$
whose Lie algebra is $\mk h_i$.

Let $\ms Z_0$ be the subgroup of $\ms G$ whose Lie algebra is $\mk z_0$ and consider 
$\ms A_1=\exp \mk a_{{\ms Z}_0},$ where $\mk a_{{\ms Z}_0}$ is a Cartan subspace of 
$\mk z_0,$ and let $\ms A_2=\exp \mk a_{{\ms H_2}},$ where the Cartan subspace $\mk a_{{\ms H_2}}$ is chosen 
so that $f_0(\pi_{\mk h_1}(\ms A_1))= \ms A_2$.
Considering the Cartan  decomposition $\ms H_2=\ms K\ms A_2\ms K$ of $\ms H_2$ where $\ms K$
is a maximal compact subgroup, 
we may write $g_n=k_na_nc_n$ with $a_n\in \ms A_2$ and $k_n,c_n\in\ms K$.
Moreover we may write $\ad(c_n)=f_0(\ad(d_n))$, where $d_n$ lies in a fixed   compact subgroup  of $\ms H_1$.
Thus, if $u\in \mk a_{\ms Z_0}$, since  $\ms A_2$ is commutative, we have
$$
f_n( \ad (d_n^{-1})u)=\ad (g_n) f_0(\ad (d_n^{-1})u)=\ad(k_n) f_0(u).
$$
We may extract a subsequence so that that $\{k_n\}_{n\in\mathbb N}$ and  $\{ d_n\}_{n\in\mathbb N}$  converge respectively to $k_0$ and $d_0$. Therefore,
$$
\{(\ad(d_0^{-1})u, \ad(k_0)f_0(u))\mid u\in\mk a_{\ms Z_0}\}\subset \mk z_0,
$$
which contradicts the fact that $\mk z_0=\bigoplus_{i=1}^q \mk g_i$. 
This contradiction establishes the fact that the set of Anosov, virtually Zariski dense regular homomorphisms is open, which
completes the proof.
\end{proof}

We record the following observation, established in the proof of Proposition \ref{zd variety} which will
be useful in the proof of Corollary \ref{metric on anosov}.

\begin{proposition}
\label{zd manifold}
Suppose that $\Gamma$ is a word hyperbolic group, $\ms G$ is a  semi-simple Lie group with finite center and $\ms P$ is
a non-degenerate parabolic subgroup of $\ms G$. Then $\widetilde{\mathcal{Z}}(\Gamma;\ms G,\ms P)/\ms G^0$ is an analytic manifold.
\end{proposition}

\subsection{Kleinian groups}

Let  ${\Cc}(\Gamma,\ms{PSL}_2(\mathbb C))$ be the set of (conjugacy classes of)
convex cocompact representations of $\Gamma$ into $\ms{PSL}_2(\mathbb C))$. We say that a convex cocompact representation $\rho$ in   $\ms{PSL}_2(\mathbb C)$ is {\em Fuchsian} if
its image is conjugate into $\ms{PSL}_2(\mathbb R)$.  Since every non-elementary Zariski closed, connected
subgroup of $\ms{PSL}_2(\mathbb C)$ is conjugate to $\ms{PSL}_2(\mathbb R)$, we note that 
$\rho\in{\Cc}(\Gamma,\ms{PSL}_2(\mathbb C))$ is Zariski dense unless $\rho$ is {\em virtually Fuchsian}, i.e.
there exists a finite index subgroup of $\rho(\Gamma)$ which is conjugate into $\ms{PSL}_2(\mathbb R)$
(see also Johnson-Millson \cite[Lemma 3.2]{johnson-millson}).
Notice that if $\rho$ is virtually Fuchsian, then $\rho(\Gamma)$ contains a finite index subgroup which
is isomorphic to a free group or a closed surface group.

Bers \cite{bers-spaces} proved that ${\Cc}(\Gamma,\ms{PSL}_2(\mathbb C))$ is a complex analytic manifold.
which has real dimension $-6 \chi(\Gamma)$ if $\Gamma$ is torsion-free.
(See also Kapovich \cite[Section 8.8]{kapovich} where a proof of this is given
in the spirit of Proposition \ref{convexsmooth}.) We summarize these results in the following proposition.

\begin{proposition}
\label{CC manifold}
Let $\Gamma$ be a word hyperbolic group. Then
\begin{enumerate}
\item
${\Cc}(\Gamma,\ms{PSL}_2(\mathbb C))$ is a smooth analytic manifold.
\item
$\rho\in{\Cc}(\Gamma,\ms{PSL}_2(\mathbb C))$ is Zariski dense if and only if $\rho$ is not virtuallyFuchsian.
\item
If $\Gamma$ is torsion-free, then ${\Cc}(\Gamma,\ms{PSL}_2(\mathbb C))$ has dimension
$-6\chi(\Gamma)$.
\end{enumerate}
\end{proposition}

\subsection{Hitchin components}

Let  $S$ be  a closed  orientable surface of genus at least 2 and let
$\tau_m:\ms{PSL}_2(\mathbb R)\to \ms{PSL}_m(\mathbb R)$ be an irreducible
homomorphism.
If $\rho:\pi_1(S)\to\ms{PSL}_2(\mathbb R)$  is discrete and faithful, hence uniformizes $S$, 
then $\tau_m\circ\rho$ is  called a {\em Fuchsian representation}. A representation
$\rho:\pi_1(S)\to \ms{PSL}_m(\mathbb R)$  that can be deformed into a Fuchsian representation
is called a {\em Hitchin representation}. 
Lemma 10.1 of \cite{labourie-anosov} implies that all Hitchin representations are irreducible.

Let $H_m(S)$ be the space of Hitchin representations into $\ms{PSL}_m(\mathbb R)$ and let
$$
\mathcal H_m(S)=H_m(S)/\ms{PGL}_m(\mathbb R).
$$
Each  $\mathcal H_m(S)$ is called a {\em Hitchin component} and Hitchin \cite{hitchin} proved that 
$H_m(S)$ is  an analytic manifold diffeomorphic to $\mathbb R^{(m^2-1)|\chi(S)|}.$ 

One may identify the Teichm\"uller space $\mathcal{T}(S)$ with $\mathcal{H}_2(S)$.
The irreducible representation gives rise to an analytic embedding  that we also denote
$\tau_m$, of $\mathcal{T}(S)$ into the Hitchin component $\mathcal H_m(S)$ and we
call its image the {\em Fuchsian locus} of the Hitchin component. 

Each Hitchin representation lifts to a representation into $\sln$. Labourie \cite{labourie-anosov} showed that
all lifts of Hitchin representations are irreducible and  $(\sln, \ms B)$-Anosov where
$\ms B$ is a minimal parabolic subgroup of $\sln$.
In particular, lifts of Hitchin representations are projective Anosov.
Moreover, Labourie \cite{labourie-anosov} showed that the
image of every non-trivial element of $\pi_1(S)$ under the lift of a Hitchin representation
is diagonalizable with distinct eigenvalues.
In particular,  every lift of a Hitchin representation is $\sln$-generic, so is contained in
$\mathcal{C}_g(\pi_1(S), \sln)$.
Moreover, notice that distinct lifts of a given Hitchin representation must be contained in
distinct components of $\mathcal{C}_g(\pi_1(S), \sln)$.

We summarize what we need from Hitchin and Labourie's work in the following result.

\begin{theorem}
\label{hitchin component}
Every  Hitchin component  lifts to a component of the analytic manifold  $\mathcal{C}_g(\pi_1(S), \sln)$.
 \end{theorem}

\section{Thermodynamic formalism on the deformation space of projective Anosov representations}

In Section \ref{analyticity of entropy}, we  show that entropy, intersection and renormalized intersection vary analytically 
over $\Cm$, then in section \ref{thermodynamic mapping} we construct the thermodynamic mapping of $\Cm$ into 
the space of Liv\v sic cohomology classes of pressure zero functions on $\Gg$ and use it to define the pressure
form on $\Cm$ and $\Cg$.

\subsection{Analyticity of entropy and intersection}
\label{analyticity of entropy}

Let $\Gamma$ be a word hyperbolic group admitting a  projective Anosov representation. 
By Proposition \ref{pro:geodano}, the Gromov  geodesic flow  on $\Gg$  admits a H\"older reparametrization
which turns it into a topologically transitive metric Anosov flow. 
Since the Gromov geodesic flow is only well defined up to reparametrization, we choose a fixed  H\"older reparametrization 
which gives rise to a  topologically transitive metric Anosov flow, and use the corresponding flow, 
denoted by $\psi=\{\psi_t\}_{t\in\mathbb R}$, as a background flow on $\Gg$.

Let $\rho:\Gamma\to\sln$ be a projective Anosov representation.
By Proposition  \ref{geoflowreparam}, the geodesic flow  $(\Ug, \{\phi_t\}_{t\in\mathbb R})$ of $\rho$ is 
H\"older conjugate to a H\"older reparametrization of the flow $\{\psi_t\}_{t\in\mathbb R}$.
Periodic orbits of $\{\phi_t\}_{t\in\mathbb R}$ are in one-to-one correspondence with
conjugacy classes of infinite order elements of
$\Gamma.$ The periodic orbit associated to the conjugacy class $[\gamma]$ has period
$\Lambda(\rho)(\gamma)$.

If $\rho:\Gamma\to \sln$ is projective Anosov, 
let $f_\rho:\Gg\to\mathbb R$ be a H\"older function  such that the reparameterization of $\Gg$
by $f_\rho$ is H\"older conjugate to $\Ug$.
Liv\v sic's theorem \ref{theorem:livsic} implies that the correspondence $\rho\mapsto f_\rho$ 
is well defined modulo Liv\v sic cohomology and invariant under conjugation of the homomorphism $\rho$. 
Therefore, we may define
\begin{eqnarray}
h(\rho_1)&=&h(f_{\rho_1}),\\
\II(\rho_1,\rho_2)&=&\II(f_{\rho_1},f_{\rho_2}), \ {\rm and}\\
\JJ(\rho_1,\rho_2)&= &\JJ(f_{\rho_1},f_{\rho_2})=\frac{h(\rho_2)}{h(\rho_1)}\II(\rho_1, \rho_2),
\end{eqnarray}
for  projective Anosov representations $\rho_1:\Gamma\to\sln$ and \hbox{$\rho_2:\Gamma\to\sln$}.
These quantities  are well defined and agree with the definition given in the Introduction.
Proposition 7.3.1 implies that 
$$h(f_{\rho_1})=\lim_{T\to\infty}\frac{1}{T}\log\sharp(R_T(\rho_1))$$
while equation (\ref{inter1}) implies that
$$\II(f_{\rho_1},f_{\rho_2})=\lim_{T\to\infty}\left(\frac{1}{\sharp (R_T(\rho_1))}\sum_{[\gamma]\in R_T(\rho_1)}\frac{\log(\Lambda(\gamma)(\rho_2))}{\log(\Lambda(\gamma)(\rho_1))}\right).$$

Proposition \ref{Creparam} implies that  if  $\{\rho_u\}_{u\in D}$ is an analytic family of
of projective Anosov homomorphisms defined on a disc $D$, then we can choose,  at least locally, the  
map $u\mapsto f_{\rho_u}$ to be analytic.  Proposition \ref{pro:ana0} then implies that entropy, intersection and
renormalized intersection all vary analytically.

\begin{proposition}\label{pro:ana}
Given  two analytic families $\{\rho_u\}_{u\in D}$ and $\{\eta_v\}_{v\in D'}$  of projective Anosov homomorphisms,  
the functions $u\mapsto h(\rho_u)$,  $(u,v)\mapsto\II(\rho_u,\eta_v)$ and  $(u,v)\mapsto\JJ(\rho_u,\eta_v)$ 
are analytic on their domains of definition.
\end{proposition}

Combining Propositions \ref{proposition:lowerbound}, \ref{proposition:deriv-presion} and \ref{prop:deriv2} one obtains the following.
\begin{corollary}\label{corollary:c}
\label{J non-negative}
For every pair $\rho_1:\Gamma\to\sln$ and \hbox{$\rho_2:\Gamma\to\sln$} of
projective Anosov representations, one has 
$$\JJ(\rho_1,\rho_2)\geq1.$$
If $\JJ(\rho_1,\rho_2)=1$,
then there exists a constant $c\geq1$ such that $$\Lambda_{\rho_1}(\gamma)^c=\Lambda_{\rho_2}(\gamma)$$
for every $\gamma\in\Gamma.$ 

Moreover, if $\{\rho_t\}$ is a smooth one parameter family of projective Anosov 
representations and $\{f_t\}$ is an 
associated smooth family of reparametrizations,  then 
$$\left.\frac{\partial^2}{\partial t^2}\right|_{t=0}\JJ(\rho_0,\rho_t)=0$$
if and only if 
$$
\left.\frac{\partial}{\partial t}\right\vert_{t=0}\left(h_{\rho_t}f_t\right)
$$
is  Liv\v sic cohomologous to 0.
\end{corollary}

\subsection{The thermodynamic mapping and the pressure form}
\label{thermodynamic mapping}
If $\rho\in\Cm$ and $f_\rho$ is a reparametrization of the Gromov geodesic flow giving rise to the geodesic
flow of $\rho$, we define 
\hbox{$\Phi_\rho:\Gg\to\Real$} by $$\Phi_\rho(x) = -h(\rho)f_\rho(x).$$
Lemma \ref{lemma:entropia2} implies that $\Phi_\rho\in{\mathcal P}(\Gg)$. Let $\mathcal H(\Gg)$ be the set of Liv\v sic cohomology classes of pressure zero function, we saw that the class of $\Phi_\rho$ in  $\mathcal H(\Gg)$ only depends on $\rho$.  We define
the {\em thermodynamic mapping} to be
$$
\mathfrak T:\left\{\begin{array}{rcl}\Cm &\rightarrow &{\mathcal H}(\Gg)\cr
 \rho& \mapsto& [\Phi_\rho]\end{array}\right.
 $$
 
By Proposition \ref{Creparam},  the thermodynamic mapping is ``analytic" in the following sense: 
for every representation $\rho$ in the analytic manifold  $\Cm$, there exists a neighborhood $U$ of $\rho$ in $\Cm$
and an analytic mapping from $U$ to $\mathcal P(\Gg)$ which lifts the thermodynamic mapping.

We use the thermodynamic mapping to define a 2-tensor on our deformation spaces.

\begin{definition}{\sc[Pressure Form]}\label{def:pressureform}
Let $\{\rho_u\}_{u\in M}$ be an analytic family of projective Anosov homomorphims
parametrized by an analytic manifold  $M$. 
If $z\in M$, we define $\JJ_z:M\to \mathbb R$ by letting
$$\JJ_z(u)=\JJ(\rho_z,\rho_u).$$
The associated {\em pressure form} ${\bf p}$  on $M$ is the 2-tensor such that if $v,w\in \TT_zM$,
then
$$
{\bf p}(v,w)={\rm D}_z^2\JJ_z(v,w).
$$
Notice that,
by Corollary \ref{J non-negative}, the pressure form is non-negative.
\end{definition} 

In particular, we get pressure forms on $\widetilde{\mathcal C}(\Gamma,m)$ and on
$\widetilde{\mathcal C}(\Gamma,\ms G)$ when $\ms G$ is a reductive subgroup of $\sln$.
Since $\JJ$ is invariant under the action of conjugation on each variable,
these pressure forms descend to 2-tensors, again called 
pressure forms, on the analytic manifolds $\Cm$ 
and $\Cg$.

\section{Degenerate vectors for the pressure metric}

In this section, we analyze the norm zero vectors for the pressure metric.  
If $\Gamma$ is a word hyperbolic group, $\alpha$ is an infinite order element  of $\Gamma$
and $\{\rho_u\}_{u\in M}$ is an analytic family of  projective Anosov homomorphisms parameterized by an analytic manifold $M$,
one may view $\LL(\alpha)$ as an analytic function on
$M$ where we abuse notation by letting $\LL(\alpha)(u)=\LL(\alpha)(\rho_u)$  denote the eigenvalue of
$\rho_u(\alpha)$ of maximal modulus. The following is the  main result of the section.

\begin{proposition}
\label{generic implies log type 0}
Let $\Gamma$ be a word hyperbolic group and let $\ms G$ be a reductive subgroup of $\sln$.
Suppose that $\{\rho_u:\Gamma\to \ms G\}_{u\in D}$ is an analytic  family of projective Anosov 
$\ms G$-generic homomorphisms  defined on a disc $D$ with associated pressure form ${\bf p}$.
Suppose that $z\in D$, $v\in \TT_z D$ and
$${\bf p}(v,v)  = 0.$$
Then, for every element $\alpha$ of infinite order in $\Gamma$,
$${\rm D}_{z}\LL(\alpha)(v) =0.$$
\end{proposition}

\subsection{Log-type functions}
We begin by showing that if $v$ is a norm zero vector, then each $\LL(\alpha)$ 
is of log-type $K$at $v$
for some fixed $K$.

\begin{definition}
We say that an analytic function $f$ has {\em log-type} $K$ at $v\in\ms T_uM$, if $f(u)\ne 0$ and
$$
{\rm D}_u{\log(|f|)}(v)=K\log(|f(u)|),
$$
and is of {\em log-type}  if it is of log-type $K$ for some $K$.
\end{definition}

\begin{lemma}
\label{degenerate implies log type}
Let $\{\rho_u\}_{u\in M}$ be an analytic family of projective Anosov homomorphims
parametrized by an analytic manifold  $M$ and let ${\bf p}$ be the associated pressure form.
If $v \in \ms T_zM$ and
$${\bf p}(v,v)  = 0,$$
 then there exists  $K\in\mathbb R$ such that if $\alpha$ is any element of  infinite order in $\Gamma$, then 
$\LL(\alpha)$ is of log-type $K$ at $v$.
\end{lemma}

\begin{proof} 
Consider a smooth one parameter family $\{u_s\}_{s\in (-1,1)}$ in $M$ such that
\hbox{$u_0=z$} and $\dot u_0=v.$ Let $\rho_s=\rho_{u_s}$ and let $f_s=f_{u_s}$
where $\{f_{u_s}\}$ is a smooth family of reparametrizations obtained from Proposition \ref{Creparam}.
We define, for all $s\in(-1,1)$,
$$\Phi_s=\Phi_{\rho_s}=-h(\rho_s)f_{s},$$
 
By Corollary \ref{corollary:c}, $\dot\Phi_0$ is Liv\v sic cohomologous to zero. 
In particular, the integral of $\dot\Phi_0$ is zero on any $\phi_s$-invariant measure. 
Thus for any infinite order element  $\alpha\in\Gamma$ one has
$$\braket{\delta_\alpha|\dot{\Phi}_0}=0.$$ 
By definition, $\Phi_s=-h(\rho_s)f_{\rho_s}$ and thus 
$$\braket{\delta_\alpha|\Phi_s}=-h(\rho_s)\log\Lambda(\alpha)(u_s).$$

It then follows that
$$
0=\braket{\delta_\alpha|\left.\frac{\d {\Phi}_s(x)}{\d s}\right\vert_{s=0}}
=\left.\frac{\d (\braket{\delta_\alpha|\Phi_s})(x) }{\d s}\right\vert_{s=0}
=\left.\frac{\d\left(h(\rho_s)\log(\Lambda(\alpha)(u_s)\right)}{\d s}\right\vert_{s=0}.
$$
Applying the chain rule we get
$$
0=\left(\left.\frac{\d h(\rho_s)}{\d s}\right\vert_{s=0}\right)\log(\Lambda(\alpha)(u_s))+h(\rho_s)\left(\left.\frac{\d \log(\Lambda(\alpha)(u_s)}{\d s}\right\vert_{s=0}\right).
$$
It follows that setting $$
K=-\frac{1}{h(\rho_0)}
\left.\frac{\d \left(h(\rho_s)\right)}{\d s}\right\vert_{s=0}
,
$$ 
we get that for all $\alpha\in\Gamma$, 
$${\rm D}_{z}\log(\Lambda(\alpha))(v) =
\left.\frac{\d}{\d s}\right\vert_{s=0}\left(\log(\Lambda(\alpha)(\rho_s)\right)=K\log\left(\Lambda(\alpha)(z)\right).$$
Since $\Lambda(\alpha)=|\LL(\alpha)|$, $\LL(\alpha)$ has log-type $K$ at $v$.
\end{proof}

\subsection{Trace functions}

Recall, from Proposition \ref{proxi2}, that if  $\alpha$
is an infinite order element of $\Gamma$ and $\rho$ is a projective Anosov representation in $\Cm$, then we may write
$$\rho(\alpha)=\LL(\alpha)(\rho)\p(\rho(\alpha)) +\m(\rho(\alpha))+\frac{1}{\LL(\alpha^{-1})(\rho)}\q(\rho(\alpha)),
$$
where \begin{enumerate}
\item
$
\LL(\alpha)(\rho)
$
is the eigenvalue of $\rho(\alpha)$ of maximum modulus and $\p(\rho(\alpha))$ is the projection on 
$\xi(\alpha^+)$ parallel to $\theta(\alpha^-)$
\item $
\LL(\alpha^{-1})(\rho)
$
is the eigenvalue of $\rho(\alpha^{-1})$ of maximal modulus and  $\q(\rho(\alpha))$ is the projection onto the line
$\xi(\alpha^-)$ parallel to $\theta(\alpha^+)$,  
and
\item 
the spectral radius of $\m(\rho(\alpha))$ is less than $\delta^{l(\alpha)}\Lambda(\alpha)(\rho)$ for some 
$\delta=\delta(\rho)\in(0,1)$ 
which depends only on $\rho$.
\end{enumerate}

It will be useful to define
$${\bf r}(\rho(\alpha))=\m(\rho(\alpha))+\frac{1}{\LL(\alpha^{-1})(\rho)}\q(\rho(\alpha))$$
which also has spectral radius less than 
$\delta^{l(\alpha)}\Lambda(\alpha)(\rho)$.

If $\{\rho_u\}_{u\in D}$ is an analytic  family of projective Anosov $\ms G$-generic homomorphisms  defined on a disc $D$
and $\alpha$  and $\beta$ are infinite order elements of  $\Gamma$, we consider the following analytic functions on $D$:
\begin{eqnarray*}
\TT(\alpha,\beta):&u\mapsto&\tr(\rho_u(\alpha)\rho_u(\beta))\cr
\TT(\p(\alpha),\beta):&u\mapsto&\tr(\p(\rho_u(\alpha))\rho_u(\beta)),\cr
\TT(\p(\alpha),\p(\beta)):&u\mapsto&\tr(\p(\rho_u(\alpha))\p(\rho_u(\beta))),\cr
\TT(\p(\alpha),{\bf r}(\beta)):&u\mapsto&\tr(\p(\rho_u(\alpha)){\bf r}(\rho_u(\beta))),\cr
\TT({\bf r}(\alpha),\p(\beta)):&u\mapsto&\tr({\bf r}(\rho_u(\alpha))\p(\rho_u(\beta))),\cr
\TT({\bf r}(\alpha),{\bf r}(\beta)):&u\mapsto&\tr({\bf r}(\rho_u(\alpha)){\bf r}(\rho_u(\beta))).\cr
\end{eqnarray*}
We say that two infinite order elements of $\Gamma$ are {\em coprime} if they have distinct fixed points in
$\bg$ (i.e. they do not share a common power).

We then  have

\begin{proposition}
\label{typk}
Let  $\{\rho_u\}_{u\in D}$ be an analytic  family of projective Anosov homomorphisms  defined on a disc $D$.
If $\alpha $ and $\beta$ are infinite order, coprime elements of $\Gamma$, then
$$
\TT(\p(\alpha),\p(\beta))=\lim_{n\to\infty} \frac{\LL(\alpha^n\beta^n)}{\LL(\alpha)^n\LL(\beta)^n}
$$
and 
$$
\TT(\p(\alpha),\beta)=\lim_{n\to\infty}\frac{\LL(\alpha^n\beta)}{\LL(\alpha)^n}.
$$

Moreover, if  $\LL(\gamma)$ has log-type $K$ at $v\in {\ms T }_uD$ for all infinite order $\gamma\in\Gamma$,
then both 
$\TT(\p(\alpha),\p(\beta))$ and $\TT(\p(\alpha),\beta)$ have log-type $K$ at $v$.
\end{proposition}

We say that a family $\{f_n\}_{n\in\mathbb N}$ of analytic functions defined on a disk $D$
{\em decays} at $v\in \TT_zD$ if
$$\lim_{n\rightarrow\infty}f_n(z)=0\ \ \textrm{and}\ \
\lim_{n\rightarrow\infty}{\rm D}_z f_n(v)=0.$$
The following observation will be useful in the proof of Proposition \ref{typk}. 

\begin{lemma}\label{pro:logtypedecay}
Let $G$ be an analytic function that may be written, for all positive integers $n$, as
$$
G=G_n(1+h_n),
$$
where $G_n$ has log-type $K$ and $\{h_n\}_{n\in\mathbb N}$ decays at $v\in\ms T_uM$, then
$G$ has log-type $K$.
\end{lemma}

\begin{proof} 
Notice that
\begin{eqnarray*}
D_u \log(G)(v) & = & D_u \log(G_n)(v)+D_u\log( 1+h_n)(v)\\
 & = & K\log G_n(u) +\frac{D_u h_n(v)}{1+h_n(u)}.\\
\end{eqnarray*}
We now simply notice that the right hand side of the equation converges to $K\log G(u)$
\end{proof}

\noindent
{\em Proof of Proposition \ref{typk}:}
First  notice that
$$\TT(\alpha^n,\beta^n)=\LL(\alpha^n\beta^n)(1+g_n)$$
where
$$g_n=\frac{\tr({\bf r}(\alpha^n\beta^n))}{\LL(\alpha^n\beta^n)}.$$
Since ${\bf r}(\alpha^n\beta^n)(\rho_u)$ has spectral radius at most 
$\delta(\rho_u)^{\ell(\alpha^n\beta^n)}|\LL(\alpha^n \beta^n)|$,
\hbox{$\delta(\rho_u)\in(0,1)$},  and $\lim_{n\to\infty}\ell(\alpha^n\beta^n)=+\infty$, we see that
$\lim_{n\to\infty} g_n(\rho_u)=0$ for all $\rho_u\in\Cm$.
Since $\{g_n\}$ is a sequence of analytic functions, $g_n$ decays at $v$.

On the other hand,
$$
\rho_u(\alpha^n\beta^n)=\LL(\alpha)^n\LL(\beta)^n\p(\alpha)\p(\beta) +\LL(\alpha)^n\p(\alpha){\bf r}(\beta^n)
+\LL(\beta)^n{\bf r}(\alpha^n)\p(\beta) +{\bf r}(\alpha^n){\bf r}(\beta^n),$$
so
$$\TT(\alpha^n,\beta^n)=\LL(\alpha)^n\LL(\beta)^n \TT(\p(\alpha),\p(\beta))(1+\hat g_n)$$
where
$$\hat g_n=\frac{\LL(\alpha)^n\TT(\p(\alpha),{\bf r}(\beta^n))
+\LL(\beta)^n\TT({\bf r}(\alpha^n),\p(\beta))+\TT({\bf r}(\alpha^n),{\bf r}(\beta^n))}{
\TT(\p(\alpha),\p(\beta))\LL(\alpha)^n\LL(\beta)^n}.$$
and again $\hat g_n$ decays at $v$.
(Notice that, since $\alpha$ and $\beta$ are co-prime, $\xi_{\rho_u}(\beta^+)$ is not contained in
$\theta_{\rho_u}(\alpha^-)$ for any $u\in D$, so $\TT(\p(\alpha),\p(\beta))$ is non-zero on $D$.)

Combining, we see that 
$$\TT(\p(\alpha),\p(\beta))=\frac{\LL(\alpha^n\beta^n)(1+g_n)}{\LL(\alpha)^n\LL(\beta)^n(1+\hat g_n)},$$
which implies, since $\lim g_n=0$ and $\lim\hat g_n=0$, that
$$
\TT(\p(\alpha),\p(\beta))=\lim_{n\to\infty} \frac{\LL(\alpha^n\beta^n)}{\LL(\alpha)^n\LL(\beta)^n}.
$$
Moreover, if $\LL(\gamma)$ has log-type $K$ at $v$ for all infinite order $\gamma\in\Gamma$,
then $G_n=\frac{\LL(\alpha^n\beta^n)}{\LL(\alpha)^n\LL(\beta)^n}$ has log-type $K$, being the
ratio of log-type $K$ functions and we may apply
Lemma \ref{pro:logtypedecay} to see that $\TT(\p(\alpha),\p(\beta))$ has log-type $K$.

We similarly derive the claimed facts about $\TT(\p(\alpha),\beta)$ by noting that
$$\TT(\alpha^n,\beta)=\LL(\alpha^n\beta)(1+h_n)$$
where 
$$h_n=\frac{\tr({\bf r}(\alpha^n\beta))}{\LL(\alpha^n\beta)},$$
and that
$$\TT(\alpha^n,\beta)=\LL(\alpha)^n\TT(\p(\alpha),\beta)
(1+\hat h_n)$$
where
$$\hat h_n=\frac{\TT({\bf r}(\alpha^n),\beta)}{\LL(\alpha^n)\TT(\p(\alpha),\beta)}$$
and applying an argument similar to the one above.
\qed

\medskip
\noindent {\bf Remark:} Dreyer \cite{dreyer} previously established that 
$$\left\{\frac{\Lambda(\alpha^n\beta)(\rho)}{\Lambda(\alpha)(\rho)^n}\right\}$$
has a finite limit when $\rho$ is  a Hitchin representation.

\subsection{Technical lemmas}

We will need a rather technical lemma, Lemma \ref {techl}, in the proof of Lemma \ref{K is 0}, which is 
itself the main ingredient in
the proof of Proposition \ref{generic implies log type 0}.

We first prove a preliminary lemma, which may be viewed
as a complicated version of the fact that exponential functions grow
faster than polynomials. If $a_s$ is a polynomial in $q$ variables and their conjugates,  we will use the notation
$$
\Vert a_s\Vert=\sup\{\vert a_s(z_1,\ldots,z_q)\vert\mid \vert z_i\vert=1\}.
$$

\begin{lemma}
Let $(f_1,\ldots, f_q)$ and $(\theta_1,\ldots,\theta_q)$ be two $q$-tuples of real numbers and
let $(g_1,\ldots,g_q)$ be a $q$-tuple of complex numbers, such that
$$1>f_1>\cdots>f_q>0.$$ 
Suppose that there exists a  strictly decreasing sequence $\{\mu_s\}_{s\in\mathbb N}$ of positive real numbers
so that \hbox{$\mu_1<1$} and a sequence of
complex-valued polynomials
$\{a_s\}_{s\in\mathbb N}$ in $q$ variables and their conjugates,
such that, for all $n\in\mathbb N$,
\begin{equation}\label{eq:basic}
\sum_{p=1}^{q}nf_p^n\, \Re(e^{in\theta_p}g_p)=\sum_{s=1}^{\infty}\mu_s^n \Re (a_{s}(e^{in\theta_1},\ldots, e^{in\theta_{q}})),
\end{equation}
and there exists $N$ such that
$$\sum_{s=1}^\infty \vert\mu_s\vert^n\Vert a_s\Vert$$
is convergent for  all $n\ge N$.
Then, for all $p=1,\ldots,q$,
\begin{eqnarray*}
\Re(g_p)=0 &\hbox{ if }& \theta_p\in 2\pi\mathbb Q,\cr
g_p=0 &\hbox{ if }& \theta_p\not\in 2\pi\mathbb Q.
\end{eqnarray*}

\end{lemma}

\begin{proof} 
There exists $r\in\mathbb N$, so that, for all $i$, either $r\theta_i\in 2\pi\mathbb Z$ or
$r\theta_i\not\in 2\pi\mathbb Q$. Equation (\ref{eq:basic}) remains true if we replace
$(\theta_1,\ldots,\theta_q)$ with $(r\theta_1,\ldots,r\theta_q)$, so we may assume that
either $\theta_i\not\in2\pi\mathbb Q$ or $\theta_i\in\mathbb Z$.

Let $V$ be the set of accumulation points of $\{ (e^{in\theta_1},\ldots,e^{in\theta_q})\ |\ n\in\mathbb N\}$.
We first show that if   $(z_1,\ldots,z_q)\in V$, then
$\Re(g_1 z_1)=0$. This will suffice to prove our claim if $p=1$, since
if $\theta_i\in 2\pi\mathbb Z$, then $z_1=1$ and $\Re(g_1)=0$. If not,
any $z_1\in S^1$ can arise in such a limit, so $\Re(z_1g_1)=0$ for all $z_1\in S^1$, which implies that
$g_1=0$.

So, suppose that $\{n_m\}$ is an increasing sequence in $\mathbb N$ and
$\{ (e^{in_m\theta_1},\ldots,e^{in_m\theta_q})\}$ converges to $(z_1,\ldots,z_q)$.
Then either
\begin{itemize}\item[(1)]   $$\Re(a_s(z_1,z_2,\ldots,z_q))=0$$
for all $s$, or 
\item[(2)] there exists  $s_0\in\mathbb N$ so that 
$$\mathfrak A_0=\Re(a_{s_0}(z_1,z_2,\ldots,z_q))\neq 0,$$
and  for all $s<s_0$ 
$$\Re(a_{s}(z_1,z_2,\ldots,z_q))=0.$$ 
\end{itemize}

If (1) holds, then Equation \eqref{eq:basic} implies
\begin{eqnarray}
\lim_{m\to\infty} n_m\Re(e^{in_m\theta_1}g_1)+\epsilon_0(n_m)=0.\label{proof:techl2}
\end{eqnarray} 
where
$$\epsilon_0(n_m)=\sum_{p=2}^{q}n_m\left(\frac{f_p}{f_1}\right)^{n_m}\Re(e^{in_m\theta_p}g_p).$$
Since, $\lim_{m\to\infty} \Re(e^{in_m\theta_1}g_1)=\Re(z_1g_1)$ and
$\lim_{m\to\infty}\epsilon_0(n_m)=0$, we conclude that $\Re(z_1g_1)=0.$

If (2) holds, then Equation \eqref{eq:basic} implies that
$$
\lim_{m\to\infty} n_m \Re(z_1g_1)+\epsilon_0(n_m)-\left(\frac{\mu_{s_0}}{f_{1}}\right)^{n_m}\mathfrak A_m(1+\epsilon_1(n_m))=0
$$
where
\begin{eqnarray}
\mathfrak A_m&=&\Re (a_{s_0}(e^{i{n_m}\theta_1},e^{i{n_m}\theta_2},\ldots,e^{i{n_m}\theta_q})),\cr
A_{m,s}&=&\frac{1}{\mathfrak A_m}\left(\frac{\mu_s}{\mu_{s_0}}\right)^{n_m}\Re \left(a_{s}(e^{in_m\theta_1},\ldots, e^{in_m\theta_{q}})\right), \ {\rm and}\cr
\epsilon_1(n_m)&=&\sum_{s=s_0+1}^\infty A_{m,s}.
\end{eqnarray}
Observe that
$$
\lim_{m\to\infty}\mathfrak A_m= 
\mathfrak A_0\not=0$$
If $m$ is large enough that $\vert \mathfrak A_m\vert\geq \frac{1}{2}\vert \mathfrak A_0\vert$ and $n_m>N$, then
$$
\vert A_{m,s}\vert\leq\frac{\mu_{s_0+1}^{n_m-N}}{\mu_{s_0}^{n_m}}B_s\ \ {\rm where}\ \ B_s= \frac{2}{\mathfrak A_0}\vert \mu_s\vert^N\Vert a_s\Vert.
$$
Since $\lim_{m\to\infty} \frac{\mu_{s_0+1}^{n_m-N}}{\mu_{s_0}^{n_m}}=0$ and $\sum_{s=1}^\infty B_s$ is convergent, 
 $\lim_{n\to\infty}\epsilon_1(n_m)=0$. 
It then follows that the sequence 
$$
\left\{\frac{1}{n_m}\left(\frac{\mu_{s_0}}{f_1}\right)^{n_m}\right\}_{m\in\mathbb N}
$$
is bounded. Thus $\mu_{s_0}\leq f_1$   and it follows that $\Re(z_1g_1)=0$.
 
Once  we have proved that $\Re(z_1g_1)=0$ for all $(z_1,\ldots,z_q)\in V$, we may use the same argument to prove
that $\Re(z_2g_2)=0$ for all $(z_1,z_2,\ldots,z_q)$ and proceed iteratively to complete the proof for all $p$.
\end{proof}

We are now read to prove the technical lemma used in the proof of  Lemma \ref{K is 0}

\begin{lemma}\label{techl}
Let $\{f_p\}_{p=1}^q$ and $\{\theta_p\}_{p=1}^q$ be 2 families of real analytic functions defined on $(-1,1)$
such that, for all $t\in(-1,1)$,
$$
1>|f_1(t)|>\cdots >|f_q(t)|>0\ \  \hbox{ and } \ \ \dot\theta_q(0)=0
$$ 
Let  $\{g_p\}_{p=1}^q$  be a family of complex valued analytic functions defined on $(-1,1)$ 
so that $g_q(0)\in \mathbb R\setminus\{0\}$. 
For all $n\in\mathbb N$,  let
$$
F_n=1+\sum_{p=1}^{q}f_p^n\Re(e^{in\theta_p}g_p).
$$
If there exists a constant $K$ such that for all large enough $n$, 
$$
\dot{F_n}(0)=KF_n(0)\log(F_n(0)).
$$
Then, $\dot f_q(0)=0$.
\end{lemma}

\begin{proof}
We first notice that it suffices to prove the lemma in the restricted setting where  $f_p(t)>0$ for all $p$ and all $t$.
In general, we can then replace each $f_p$ with $f_p^2$ and
each $\theta_p$ with $2\theta_p$ and apply the restricted form of the lemma to
conclude that $\left.\frac{\d}{\d t}\right\vert_{t=0} f_q^2=0$, which implies that $\dot f_q(0)=0$.
For the remainder of the proof, we will assume that $f_p(t)>0$ for all $p$ and all $t$.

Let $g(x)=K(1+x)log(1+x)$. Then $g$ is analytic at $0$. Consider the expansion
$$g(x)=\sum_{n>0} a_mx^m $$
with radius of convergence $\delta>0$.
Notice that there exists $N$ such that if $n\ge N$, then 
$$\sum_{p=1}^q f_p(0)^n |g_p(0)|<\frac{\delta}{2}.$$

If $n\ge N$, then
\begin{eqnarray*}
KF_n(0)\log(F_n(0)) & = & g\left(\sum_{p=1}^{q}f_p(0)^n\Re(e^{in\theta_p(0)}g_p(0))\right)\cr
 & =  & \sum_{m>0} a_m\left(\sum_{p=1}^{q}f_p(0)^n\Re(e^{in\theta_p(0)}g_p(0))\right)^m.
\end{eqnarray*}
If we expand this out, for each $q$-tuple of non-negative integers $\vec m=(m_1,\ldots,m_q)$, we get
a  term of the form 
\begin{equation}
\label{m term}
a_{m_1+\cdots+m_q}\left(\Pi_{p=1}^q f_p(0)^{m_p}\right)^n\binom{m_1+\cdots+m_q}{m_1\  m_2\ \cdots\  m_q} \left(\Pi_{p=1}^q (\Re(g_p(0)e^{in\theta_p(0)})^{m_p}\right).
\end{equation}
Let 
$$h_{\vec m}= \Pi_{p=1}^q f_p(0)^{m_p} <1.$$
Using the equality $\Re (z(w+\bar w))=2\Re(z)\Re(w)$ repeatedly, we may rewrite the term in (\ref{m term})
in the form
$$h_{\vec m}^n \Re(H_{\vec m}(e^{in\theta_1(0)},\ldots,e^{in\theta_q(0)})$$
where $H_{\vec m}$ is a complex polynomial in $q$ variables and their conjugates.
Since the series $\sum_{\vec m} h_{\vec m}^n \Vert H_{\vec m}\Vert$ is  convergent for all $n\ge N$, 
we are free to re-arrange the terms.
We group all terms where the coefficient $h_{\vec m}$ agrees (of which there are only finitely many for each value
of $h_{\vec m}$) and
order the resulting terms in decreasing order of co-efficient to express
$$
KF_n(0)\log(F_n(0))=\sum_{s=0}^\infty h_s^n\Re(H_{s}( e^{in\theta_1},\ldots, e^{in\theta_{q}})),
$$
where each $H_s$ is a complex polynomial in $q$ variables and their conjugates
and
$\{ h_s\}_{s\in\mathbb N}$ 
is a strictly decreasing sequence  of positive numbers less than 1.
Moreover, for all $n\ge N$ the series
$$\sum_{s=0}^\infty h_s^n \Vert H_s\Vert$$
is convergent.

On the other hand,
\begin{eqnarray*}
\dot {F_n}(0)
=\sum_{p=1}^{q} nf_p^n\Re\left(e^{in\theta_p}g_p\left(\frac{\dot f_p}{f_p}+i\dot\theta_p  \right)\right)+\sum_{p=1}^{q}f_p^n
\Re(e^{in\theta_p}\dot g_p)
\end{eqnarray*}
where all functions on the right hand side are evaluated at 0.
Since 
$\dot{F_n}(0)=KF_n(0)\log(F_n(0))$
we see that 
$$
\sum_{p=1}^{q}n f_p^n
\Re\left(e^{in\theta_p}g_p\left(\frac{\dot f_p}{f_p}+i\dot\theta_p  \right)\right)=\sum_{s=1}^{\infty}h_s^n \Re(H_s(e^{in\theta_1},\ldots,e^{in\theta_{q}})).
$$

The previous lemma then implies that for all $p$
$$
\Re\left(g_p\left(\frac{\dot f_p}{f_p}+i\dot\theta_p \right)\right)=0
$$
Since $g_q(0)$ is a non zero real number, $f_q(0)\not=0$ and $\dot\theta_q(0)=0$,  we get that $\dot f_q(0)=0$.
\end{proof}

\subsection{Degenerate vectors have log-type zero}
Proposition \ref{generic implies log type 0} then follows from the following lemma and Lemma
\ref{degenerate implies log type}. 

\begin{lemma}
\label{K is 0} 
Let $\Gamma$ be a word hyperbolic group and let $\ms G$ be a reductive subgroup of $\sln$.
If $\{\rho_u\}_{u\in D}$ is an analytic  family of projective Anosov $\ms G$-generic homomorphisms  defined on a disc $D$
and $\LL(\alpha)$ has log-type $K$ at $v\in \ms T_zD$ for all infinite order $\alpha\in\Gamma$, then 
${\rm D}_z\LL(\alpha)(v)=0$ for all infinite order $\alpha\in\Gamma$.
\end{lemma}

\begin{proof}
Notice that if we replace  the family $\{\rho_u\}_{u\in D}$ by a  conjugate family 
\hbox{$\{\rho'_u=g_u\rho_ug_u^{-1}\}_{u\in D}$}
where $\{g_u\}_{u\in D}$ is an analytic family of elements of $\sln$,
then $\LL(\alpha)(\rho_u)=\LL(\alpha)(\rho_u')$ for all $u\in D$. Therefore, we are free to conjugate our original
family when proving the result.

By Proposition \ref{pro:gen}, we may choose $\beta\in\Gamma$, so that $\rho_u(\beta)$ is generic.
After possibly restricting to a smaller disk about $z$, we may assume that  $\rho_u(\beta)$ is generic
for all $u\in D$.  We may then conjugate the family so that
$\rho_u(\beta)$ lies in the same maximal torus for all $u$,
we can write
$$
\rho_u(\beta^n)=\LL(\beta)^n\mathsf p+\sum_{p=1}^{q-1}\lambda^n_p\left( \cos(n\theta_p)\mathsf p_p+
\sin(n\theta_p) \widehat{\mathsf p_p}\right)+ \frac{1}{\LL(\beta^{-1})^n}\mathsf q,
$$
where $\LL(\beta)$, $\LL(\beta^{-1})$, $\lambda_p$, and $\theta_p$ are analytic functions of $u$
and 
$$|\LL(\beta)(u)|>|\lambda_1(u)|>|\lambda_2(u)|>\cdots >|\lambda_{q-1}(u)|>\frac{1}{|\LL(\beta^{-1}(u)|}>0$$
for all $u\in D$.

Choose an infinite order element $\alpha\in\Gamma$  which is coprime to $\beta$.
Proposition \ref{typk}, implies that, for all $n$,
\begin{eqnarray*}
& &\frac{\TT(\p(\alpha),\beta^n)}
{\LL(\beta^n)\TT(\p(\alpha),\p(\beta))}=1+\left(\frac{1}{\LL(\beta)\LL(\beta^{-1})}\right)^n\left(\frac{\tr(
\p(\rho(\alpha))\ms q))}
{\TT(\p(\alpha),\p(\beta))}\right)\cr
& &+\sum_{p=1}^{q-1}\left(\frac{\lambda_p}{\LL(\beta)}\right)^n \Re\left(e^{in\theta_p}\left(\frac{\tr(
\p(\rho(\alpha))\mathsf p_p)}
{\TT(\p(\alpha),\p(\beta))}+i\frac{\tr(
\p(\rho(\alpha))\widehat{\mathsf p_p})}
{\TT(\p(\alpha),\p(\beta))}\right)\right).
\end{eqnarray*}
has log-type $K$ at $v$, since the numerator has log-type $K$ at $v$  and the denominator is
a product of two functions which have log-type $K$ at $v$.

Since $\alpha$ and $\beta$ are coprime and $\rho$ is projective Anosov,
$\xi(\beta^-)\oplus\theta(\alpha^-)=\mathbb R^m$, so
$
\tr(\p(\rho(\alpha)),\ms q)\not=0
$
(since $\p(\rho(\alpha))$ is a projection onto the line $\xi(\alpha^+)$ parallel to $\theta(\alpha^-)$ and 
$\ms q=\q(\rho(\beta))$ is a projection onto the line $\xi(\beta^-)$). Similarly, $\TT(\p(\alpha),\p(\beta))\ne 0$, since
$\xi(\beta^+)\oplus \theta(\alpha^-)=\mathbb R^m$. 

Let $\{u_s\}_{s\in(-1,1)}$ be a smooth family in $D$ so that $u_0=z$ and $\dot u_0=v$.
We now apply Lemma \ref{techl}, taking
\begin{eqnarray*}f_p(s)&=&\frac{\lambda_p(u_s)}{\LL(\beta)(u_s)},\cr
g_p(s)&=&\left(\frac{\tr(
\p(\rho_{u_s}(\alpha))\mathsf p_p)}
{\TT(\p(\alpha),\p(\beta))(u_s)}+i\frac{\tr(
\p(\rho_{u_s}(\alpha))\widehat{\mathsf p_p})}
{\TT(\p(\alpha),\p(\beta))(u_s)}\right),
\end{eqnarray*}
if $p=1,\ldots,q-1$,  and taking
\begin{eqnarray*}
f_q(s)&=&\frac{1}{\LL(\beta)(u_s)\LL(\beta^{-1})(u_s)},\cr 
g_q(s)&=&\frac{\tr(\p(\rho_{u_s}(\alpha))\ms q)}
{\TT(\p(\alpha),\p(\beta))(u_s)}, \ {\rm and}\cr
\theta_q(s)& =& 0.
\end{eqnarray*}
We conclude from  Lemma \ref{techl} that $\dot f_q=0$. Thus
\begin{eqnarray}
{\rm D}_z\LL(\beta)(v)\cdot\LL(\beta^{-1})(z)=-\LL(\beta)(z)\cdot {\rm D}_z\LL(\beta^{-1})(v).\label{eq:tech1}
\end{eqnarray}
Since $\LL(\beta)$ and $\LL(\beta^{-1})$ both have log-type $K$ at $v$, we get that 
\begin{eqnarray}
\frac{{\rm D}_z\LL (\beta)(v)}{\LL(\beta)(z)}=K\log(|\LL(\beta)(z)|)\ \ \textrm{and}\ \ 
\frac{{\rm D}_z\LL (\beta^{-1})(v)}{\LL(\beta^{-1})(z)}=K\log(|\LL(\beta^{-1})(z)|)
.\label{eq:tech2}
\end{eqnarray}
Combining (\ref{eq:tech1}) and (\ref{eq:tech2}) we see that
$$K\log(|\LL(\beta)(z)|)=\frac{{\rm D}_z\LL(\beta)(v)}{\LL(\beta)(z)}=
-\frac{{\rm D}_z\LL(\beta^{-1})(v)}{\LL(\beta^{-1})(z)}=-K\log(|\LL(\beta^{-1})(z)|).$$
Since $\log|\LL(\beta)(z)|>0$ and $\log|\LL(\beta^{-1})(z)|>0$, this implies that $K=0$.
Therefore, $\LL(\alpha)$ has log-type 0 at $v$ for all infinite order $\alpha\in\Gamma$, so
${\rm D}_z\LL(\alpha)(v)=0$ for all infinite order $\alpha\in\Gamma$.
\end{proof}

\section{Variation of length and cohomology classes}

The aim of this section is to prove the following proposition.

\begin{proposition}\label{typ0ba}
Let $\Gamma$ be a word hyperbolic group and let $\ms G$ be a reductive subgroup of $\sln$.
Suppose that $\eta:D\to \hom(\Gamma,\ms G)$ is an analytic map such that for each $u\in D$,
$\eta(u)=\rho_u$ is  irreducible, projective Anosov, and $\ms G$-generic.
If $v\in \TT_zD$ and
$$
{\rm D}_z \LL(\alpha)(v)=0
$$ 
for all infinite order elements $\alpha\in\Gamma$,
then ${\rm D}_z\eta(v)$ defines a zero cohomology class in $H^1_{\eta(z)}(\Gamma,\mathfrak g)$.
\end{proposition}

We recall that ${\rm D}_z\eta(v)$ defines a zero cohomology class in $H^1_{\eta(z)}(\Gamma,\mathfrak g)$
if and only if it is tangent to the orbit $\ms G \eta(z)$ in $\hom(\Gamma,\ms G)\subset \ms G^r$.

Propositions   \ref{generic implies log type 0}  and \ref{typ0ba} together imply 
that the pressure form is non-degenerate on $\mathcal{C}_g(\Gamma,\ms G)$.
More generally, we obtain the following corollary.

\begin{corollary}\label{typ0b}
Let $\Gamma$ be a word hyperbolic group and let $\ms G$ be a reductive subgroup of $\sln$.
Suppose that $\eta:D\to \widetilde{\mathcal{C}}_g(\Gamma,\ms G)$ is an analytic map and $\bf p$ is
the associated pressure form on $D$.
If $v\in\ms T_zD$ and 
$$
{\bf p}(v,v)=0,
$$ 
then ${\rm D}_z\eta(v)$ defines a zero cohomology class in $H^1_{\eta(z)}(\Gamma,\mathfrak g)$.
\end{corollary}

In the course of the proof of Proposition \ref{typ0ba} we also obtain the following fact which is
of independent interest.

\begin{proposition}\label{typ0}
Suppose that $\ms G$ is a reductive subgroup of $\sln$ and \hbox{$\rho\in\Cg$}.
Then
the set
$$\left\{D_\rho \LL(\alpha)\mid\alpha \hbox{ infinite order in }\Gamma\right\},$$
generates the cotangent space $\ms T^*_\rho\Cg$.
\end{proposition}

Both propositions will be established in section \ref{vectors with log type zero}.

\subsection{Invariance of the cross-ratio}
\label{cross ratio section}

We recall the definition of the \emph{cross ratio} of a pair of hyperplanes and a pair of lines. First
define
$$\rpn^{(4)}=\{(\varphi,\psi,u,v)\in{\rpn^*}^2\times{\rpn^2}: (\varphi,v)\textrm{ and }(\psi,u)\ \textrm{span}\ \mathbb R^m\}.$$
We then define ${\bf b}:\rpn^{(4)}\to\Real$ by
$${\bf b}(\varphi,\psi,u,v)=\frac{\braket{\varphi| u}\braket{\psi|v}}{\braket{\varphi|v}\braket{\psi|u}}.$$
Notice that for this formula to make sense we must make choices of elements in $\varphi$, $\psi$, $u$, and $v$, 
but that the result is independent of our choices.

If $\rho$ is a projective Anosov representation with limit curves $\xi:\bg\to\rpn$ and 
$\theta:\bg\to \rpn^*$, we define the associated {\em cross ratio}  
on $\bg^{(4)}$,  as in \cite{cross}, to be 
\begin{equation}\label{crossatio}
{\rm b}_\rho(x,y,z,w)={\bf b}(\theta(x),\theta(y),\xi(z),\xi(w)).
\end{equation}

We first derive a formula for the cross-ratio at points associated to co-prime elements.
This formula generalizes the formula in Corollary 1.6 from Benoist \cite{benoist2}.

\begin{proposition}
\label{cr formula}
If $\rho:\Gamma\to \sln$ is a projective Anosov representation and
$\alpha$ and $\beta$ are infinite order co-prime elements of $\Gamma$, then
$${\bf b}_{\rho}(\alpha^-,\beta^-,\beta^+,\alpha^+)=\TT(\p(\alpha),\p(\beta))=
\lim_{n\to\infty}\frac{\LL(\alpha^n\beta)}{\LL(\alpha)^n}.$$
\end{proposition}

\begin{proof}
Choose $a^+\in \xi(\alpha^+)$, $a^-\in\theta(\alpha^-)$, $b^+\in\xi(\beta^+)$ and
$b^-\in\theta(\beta^-)$.
Observe  that
$$
\p(\alpha)(u)= \frac{\braket{a^-|u}}{\braket{a^-|a^+}}\ a^+.
$$
for all $u\in\mathbb R^m$.
In particular,
$$
\p(\beta)\p(\alpha)(u)= \frac{\braket{b^-|a^+}}{\braket{a^-|a^+}\braket{b^-|b^+}}\braket{a^-|u}b^+.
$$
Therefore,
$$
\TT(\p(\alpha)\p(\beta))=\frac{\braket{a^-|b^+}\braket{b^-|a^+}}{\braket{a^-|a^+}\braket{b^-|b^+}}={\bf b}_{\rho}(\alpha^-,\beta^-,\beta^+,\alpha^+).
$$
The last equality in the formula follows immediately from Proposition \ref{typk}.
\end{proof}

As a corollary, we see that if $\LL(\alpha)$ has log-type zero for all infinite order $\alpha\in\Gamma$, then
the cross-ratio also has log-type zero.

\begin{corollary}\label{bir0}
Let $\Gamma$ be a word hyperbolic group and let $\ms G$ be a reductive subgroup of $\sln$.
Suppose that  $\{\rho_u:\Gamma\to\ms G\}_{u\in D}$ is an analytic family of projective Anosov $\ms G$-generic homomorphisms
parametrized by a disc $D$.
If  $\LL(\alpha)$ has log-type 0  at $v\in\ms T_zD$ for all infinite order $\alpha\in\Gamma$,
then for all distinct collections of points $x, y, z,w\in\bg$, the function
$$
u\mapsto {\rm b}_{\rho_u}(x,y,z,w),
$$
is of log-type 0 at $v$.
\end{corollary}

\begin{proof}
Suppose that $\alpha,\beta\in\Gamma$ have infinite order.
Propositions \ref{typk} and \ref{cr formula} imply that
${\bf b}_{\rho}(\alpha^-,\beta^-,\beta^+,\alpha^+)$ has log-type 0.

Since pairs of fixed points of infinite order elements are dense in $\bg^{(2)}$ and
$\xi_u$ and $\theta_u$ vary analytically by Proposition \ref{Comega},
we see that
$$
\rho\mapsto {\rm b}_\rho(x,y,z,w),
$$
has log-type 0 for all pairwise distinct $x,y,z,w\in\bg$.
\end{proof}

\subsection{An useful immersion}

We define a mapping from $\psln$ into a quotient $\ms W(m)$ of the
vector space $\mathsf M^{m+1}$ of  all $(m+1)\times (m+1)$-matrices
and use it to encode a collection of cross ratios.

Consider the action of the multiplicative group $(\mathbb R\setminus\{0\})^{2(m+1)}$ on 
$\mathsf M^{m+1}$ given  by
$$
(a_0,\ldots,a_m,b_0,\ldots,b_m)(M_{i,j})=(a_ib_j M_{i,j}). 
$$
We denote the quotient by 
$$
\mathsf W(m)=\mathsf M^{m+1}/(\mathbb R\setminus\{0\})^{2(m+1)}.
$$

Given a projective frame  $F=(x_0,\ldots,x_{m})$ for ${\rpn}$ 
and a projective frame $F^*=(X_0,\ldots,X_{m})$ for the dual $\rpn^*$, let
\begin{itemize}
\item $\hat x_i$ be non zero vectors in $x_i$, such that
\begin{eqnarray}
0=\sum_{i=0}^m\hat x_i,\label{projframe1}
\end{eqnarray}
\item $\hat X_i$ be non zero covectors in $X_i$ such that 
\begin{eqnarray}
0=\sum_{i=0}^m\hat X_i.\label{projframe2}
\end{eqnarray}
\end{itemize}
Observe that $\hat x_i$, respectively $\hat X_i$, are uniquely defined up to a common multiple.
Then, the mapping  
$$\mu_{F,F^*}:\psln\to \mathsf W(m)$$
given by
$$
\mu_{F,F^*}:A\mapsto \hat X_i(A(\hat x_j)) 
$$
is well defined, independent of the choice of $\hat x_i$ and $\hat X_i$.

\begin{lemma}\label{prel}
The mapping $\mu_{F,F^*}$ is a smooth injective  immersion.
\end{lemma}

\begin{proof} 
Since  $\mu_{F,F^*}(A)$ determines the projective coordinates of the image of the projective frame $(x_0,\ldots,x_n)$ by $A$, $\mu_{F,F^*}$ is injective.

Let $\mu=\mu_{F,F^*}$. Let $\{A_t\}_{t\in (-1,1)}$ be a smooth one-parameter family in $\ms{PSL}_m(\mathbb R)$ such
that 
$$\dot A\in \TT_{A_0}(\ms{PSL}_m(\mathbb R)) \ \ \textrm{and}\ \  {\ms D}\mu(\dot A)=0.$$
Let $\{\hat X^t_i\}_{t\in(-1,1)}$ and $\{\hat x^t_j\}_{t\in(-1,1)}$ be time dependent families of
covectors in $X_i$ 
and vectors  $x_j$ respectively, and  let
$$
a^t_{i,j}=\hat X_i^t(A_t(\hat x_j^t)).
$$
If $\mathsf D\mu (\dot A)=0$, then there exists $\lambda_i$ and $\mu_j$ such that
$$
\dot a_{i,j}=\lambda_i a_{i,j}+\mu_j a_{i,j}.
$$ 
Multiplying each  $\hat X_i^t$ by $e^{-\lambda_0 t}$ and each $\hat x_i^t$ by $e^{-\mu_0 t}$ has the effect of replacing
$\lambda_i$ and $\mu_j$  by $\lambda_i-\lambda_0$ and $\mu_j-\mu_0$ respectively. 
Thus,  we may assume  that $\lambda_0=\mu_0=0$.

We now use the normalization \eqref{projframe1} and \eqref{projframe2}, to see that
$$
\sum_{i=1}^m \lambda_i a_{i,j}=0=\sum_{j=1}^m \mu_j a_{i,j}.
$$
On the other hand, since the collections of vectors $\{v_i=(a_{i,j})_{1\leq j\leq m}\}$ and 
\hbox{$\{w_j= (a_{i,j})_{1\leq i\leq m}\}$ }
are linearly independent, this implies that $\lambda_i=\mu_j=0$ for all $i$ and $j$.
\end{proof}

The following lemma  relates the immersion $\mu$ and the cross ratio.

\begin{lemma}
\label{mub}
Let $\{x_0,\ldots,x_m\}$  and $\{y_0,\ldots,y_m\}$ be collections of $m+1$  pairwise distinct points in $\bg$.
Suppose that $\rho:\Gamma\to\sln$ is projective Anosov with limit maps
$\xi$ and $\theta$ and that
 \begin{eqnarray*}
F&=&(\xi(x_0),\ldots,\xi(x_m)),\cr
F^*&=&(\theta(y_0),\ldots,\theta(y_m)).
\end{eqnarray*} 
are projective frames for $\rpn$ and $\rpn^*$.
If
$\alpha\in\Gamma$, then
$$
\mu_{F,F^*}(\pi_m(\rho(\alpha)))=[{\rm b}_\rho(y_i,z,\alpha(x_j),w)]
$$
where $z$ and $w$ are arbitrary points in $\bg$.
\end{lemma}

\begin{proof} Choose, for each $i=0,\ldots,m$, $\phi_i\in\theta(y_i)$ and $v_i\in \xi(x_i)$, and choose
$\phi\in\theta(z)$ and $v\in\xi(w)$. Then
$$\mu_{F,F^*}(\pi_m(\rho(\alpha)))=\left[\braket{\phi_i| \alpha(v_j)}\right]$$
while
$$
[{\rm b}_\rho(y_i,z,\alpha(x_j),w)]=\left[\frac{\braket{\phi_i|\alpha(v_j)}\braket{\phi|  v}}{\braket{\phi_i| v}\braket{\phi| \alpha(v_j)}}
\right].$$
The equivalence is given by taking $a_i=\frac{\braket{\phi|  v}}{\braket{\phi_i| v}}$ and
$b_j=\frac{1}{\braket{\phi| \alpha(v_j)}}$.
\end{proof}

\subsection{Vectors with log type zero}
\label{vectors with log type zero}

Propositions \ref{typ0ba} and \ref{typ0} follow from Proposition \ref{generic implies log type 0} and the following lemma.

\begin{lemma}
\label{log type zero implies zero} 
Let $\Gamma$ be a word hyperbolic group and let $\ms G$ be a reductive subgroup of $\sln$.
Suppose that $\eta:D\to \hom(\Gamma,\ms G)$ is an analytic map such that for each $u\in D$,
$\eta(u)=\rho_u$ is  irreducible, projective Anosov and $\ms G$-generic.
Suppose that $v\in \ms T_zD$ and that ${\rm D}_z \LL(\alpha)(v)=0$ for all
infinite order $\alpha\in\Gamma$. 
Then the cohomology class of ${\rm D}\eta(v)$ vanishes in $H^1_{\eta(z)}(\Gamma,\mathfrak g)$.
\end{lemma}

\begin{proof}
Let $\{ u_t\}_{t\in (-1,1)}$ be a path in $D$ so that $u_0=z$ and $\dot u_0=v$.
Let $\rho_t=\rho_{u_t}$.
By Corollary \ref{bir0}, 
$$
\left.\frac{\d}{\d t}\right\vert_{t=0}\left({\rm b}_{\rho_t}(x,y,z,w)\right)=0
$$
for any pairwise distinct $(x,y,z,w)$ in $\bg$.

Lemma \ref{projective frame} allows us to choose collections
$\{x_0,\ldots,x_m\}$ and $\{ y_0,\ldots,y_m\}$  of pairwise distinct points in $\bg$ such that
if
\begin{eqnarray*}
F_t&=&(\xi_t(x_0), \ldots \xi_t(x_m)),\cr
F_t^*&=&(\theta_t(y_0), \ldots \theta_t(y_m)).
\end{eqnarray*}
then $F_0$ and $F^*_0$ are both projective frames.
For some $\epsilon>0$, $F_t$ and $F_t^*$ are projective frames for all $t\in (-\epsilon,\epsilon)$.
(We will restrict to this domain for the remainder of the argument.)
We may then normalize, by conjugating $\rho_t$ by an appropriate element of $\sln$,
so that $F_t=F_0$ for all $t\in (-\epsilon,\epsilon)$.

Let 
$$
\mu_t=\mu_{F_t,F_t^*}\circ\pi_m.
$$
Then, by Lemma \ref{mub}, 
$$\mu_t(\rho_t(\alpha))=[{\rm b}_{\rho_t}(x_i,z,\alpha(y_j),w)]$$
for all $\alpha\in\Gamma$.
Therefore,
$$
\left.\frac{\rm d}{{\rm d}t}\right\vert_{t=0}{\mu_t(\rho_t(\alpha))}=0.
$$
for all $\alpha\in\Gamma$.
Notice that if $\chi$ and $\chi^*$ are projective frames, then
$$
\mu_{\chi,B^*\chi^*}(A)=\mu_{\chi,\chi^*}(B^{-1}\circ A),
$$
for all $A,B\in\sln$.
If we choose $C_t\in\sln$ so that $(C_t^{-1})^*(F_t^*)=F_0^*$, then
\begin{eqnarray*}
0&=&\left.\frac{\rm d}{{\rm d}t}\right\vert_{t=0}(\mu_t(\rho_t(\alpha)))=\left.\frac{\rm d}{{\rm d}t}\right\vert_{t=0}(\mu_0(C_t\rho_t(\alpha)))\cr
&=&{\rm D}\mu_0\left(\left.\frac{\rm d}{{\rm d}t}\right\vert_{t=0}(C_t\circ \rho_t(\alpha))\right).
\end{eqnarray*}
Lemma \ref{prel} implies that $\mu_0$ is an immersion, so
$$
\left.\frac{\rm d}{{\rm d}t}\right\vert_{t=0}(C_t\circ \rho_t(\alpha))=0
$$
Thus,
\begin{eqnarray}
C_0\circ \left.\frac{\rm d}{{\rm d}t}\right\vert_{t=0}\rho_t(\alpha) +\dot C_0\circ \rho(\alpha)=0.\label{cocycle}
\end{eqnarray}
Taking $\alpha={\rm id}$ in Equation \eqref{cocycle}, we see that $\dot C_0=0$. 
Since $C_0=I$,
$$
\left.\frac{\rm d}{{\rm d}t}\right\vert_{t=0}\rho_t(\alpha)=0
$$
for all  $\alpha\in\Gamma$.
Therefore the cohomology class of ${\rm D}\eta(v)$ vanishes  in $H^1_{\eta(z)}(\Gamma,\mathfrak{sl}_m(\mathbb R))$.
Since $\ms G$ is a reductive subgroup of $\sln$, $\mk{sl}_m{\mathbb R}=\mk g\oplus \mk g^\perp$, so
$H^1_{\eta(z)}(\Gamma,\mathfrak g)$ injects into $H^1_{\eta(z)}(\Gamma,\mathfrak{sl}_n(\mathbb R))$.
Therefore, ${\rm D}\eta(v)$ vanishes in $H^1_{\eta(z)}(\Gamma,\mathfrak g)$
as claimed.
\end{proof}

\section{Rigidity results}
\label{rigidity}

In this section, we establish two rigidity results for projective Anosov representations.
We first establish Theorem \ref{theorem:length} which states that the signed spectral radii  
determine the limit map of a projective Anosov representation,
up to the action of $\sln$, and that they determine the conjugacy class, in $\ms{GL}_m(\mathbb R)$,
of an irreducible projective Anosov representation. 

\begin{theorem}{\sc [Spectral rigidity]}
\label{theorem:length0} 
Let $\Gamma$ be a word hyperbolic group and let  \hbox{$\rho_1:\Gamma\to\sln$} and
\hbox{$\rho_2:\Gamma\to\sln$} be  projective Anosov representations such that 
$$\LL(\gamma)(\rho_1)=\LL(\gamma)(\rho_2)$$
for all infinite order $\gamma\in\Gamma.$ 
Then there exists $g\in\ms{GL}_m(\mathbb R)$ such that $g\circ \xi_1=\xi_2.$ 

Moreover, if $\rho_1$ is irreducible, then $\rho_2=g\rho_1 g^{-1}$.
\end{theorem}

We next establish our rigidity result for 
renormalised intersection. If $\ms H$ is a Lie group, denote by $Z(\ms H)$ its center and by $\ms H^0$ the connected component of the identity. We denote by $\pi_m$ the projection from $\sln$ to $\ms{PSL}_m(\mathbb R).$ If $\ms H\subset\sln$ denote by $\ms P\ms H=\pi_m(\ms H)$ the projectivised group. Finally, if $\rho:\Gamma\to\sln$ is a representation, denote by $\ms G_{\rho}$ the Zariski closure of $\rho(\Gamma).$
 
\begin{theorem}{\sc [Intersection rigidity]}\label{theorem:intersectionrig}
Let $\Gamma$ be a word hyperbolic group and let $\rho_1:\Gamma\to \ms{SL}_{m_1}(\mathbb R)$ and
$\rho_2:\Gamma\to \ms{SL}_{m_2}(\mathbb R)$ be projective Anosov representations such that 
$${\bf J}(\rho_1,\rho_2)= 1.$$ 
If $\ms G_{\rho_1}$ and $\ms G_{\rho_2}$ are connected, 
then there exists an isomorphism \hbox{$\sigma:\ms G_{\rho_1}/Z(\ms G_{\rho_1}) \to\ms G_{\rho_2}/Z(\ms G_{\rho_2})$}
such that 
$$\sigma \bar\rho_1=\bar\rho_2,$$
where $\bar\rho_i:\Gamma\to \ms G_{\rho_i}/Z(\ms G_{\rho_i})$ is the composition of $\rho_i$ and the projection of
$\ms G_{\rho_i}$ onto $\ms G_{\rho_i}/Z(\ms G_{\rho_i}).$
\end{theorem}

\rmks \begin{enumerate} 
\item
If either $\ms G_{\rho_1}$ or $\ms G_{\rho_2}$ is not connected, 
then Theorem \ref{theorem:intersectionrig} holds for the finite index subgroup 
$$\Gamma_0=\Gamma\cap\rho_1^{-1}(\ms G_{\rho_1}^0)\cap\rho_2^{-1}(\ms G_{\rho_2}^0).$$  
Indeed, each $\rho_i|\Gamma_0$ is again projective Anosov (see \cite[Cor. 3.4]{guichard-wienhard}), 
and Corollary \ref{corollary:c} implies that $\JJ(\rho_1|_{\Gamma_0},\rho_2|_{\Gamma_0})=1.$
\item
Consequently, if $\ms G^0_{\rho_1}$ and $\ms G^0_{\rho_2}$ are not isomorphic, then Theorem
\ref{theorem:intersectionrig} implies that $\JJ(\rho_1,\rho_2)>1.$ 
\item 
The representations need not actually be conjugate if \hbox{$\JJ(\rho_1,\rho_2)=1.$ }
Let $\rho:\pi_1(S)\to\ms{PSL}_2(\Real)$ be  a Fuchsian representation and let
\hbox{$\tau_{k}:\ms{PSL}_2(\Real)\to\ms{PSL}_{k}(\Real)$} be the irreducible representation,
then $$\JJ(\tau_n\circ\rho,\tau_m\circ\rho)=1$$ but $\tau_n\circ\rho$ and $\tau_m\circ\rho$ are not conjugate if $n\neq m.$
\end{enumerate}

\subsection{Spectral rigidity}
Our spectral rigidity theorem will follow from Proposition \ref{cr formula} and  work of Labourie \cite{cross}.

Recall, from Section \ref{cross ratio section}, that we defined the cross ratio $\bf b$ of a
pair of hyperplanes and a pair of lines. Then, given a projective Anosov representation $\rho$ with limit
maps $\xi$ and $\theta$, we defined a cross ratio $\rm b_\rho$ on $\bg^{(4)}$ by letting
\begin{equation}
{\rm b}_\rho(x,y,z,w)={\bf b}(\theta(x),\theta(y),\xi(z),\xi(w)).
\end{equation}

Labourie \cite[Theorem 5.1]{cross} showed that if $\rho$ is a projective Anosov representation with limit map $\xi$,
then the dimension $\dim\braket{\xi(\bg)}$  can be read directly from the cross ratio $\rm b_\rho$.
(In \cite{cross}, Labourie explicitly handles the case where $\Gamma=\pi_1(S)$, but his proof
generalizes immediately.)
Consider $S^p_*$ the set of pairs $(e,u)=(e_0,\ldots,e_p,u_0,\ldots,u_p)$ of $(p+1)$-tuples in $\bg$ such that
$e_j\neq e_i\neq u_0$ and $u_j\neq u_i\neq e_0$ when $j>i>0$. If $(e,u)\in S^p_*$, he defines 
$$\chi^p_{\rm b_\rho}(e,u)=\det_{i,j>0}({\rm b}_\rho(e_i,e_0,u_j,u_0)).$$

\begin{lemma}
\label{cross ratio and dim}
If $\rho:\Gamma\to\sln$ is projective Anosov, then
$$\dim\braket{\xi(\bg)}=\inf\{p\in\mathbb N:\chi^p_{\rm b_\rho}\equiv0\}-1.$$
\end{lemma}

Lemma 4.3 of Labourie  \cite{cross} extends in our setting to give:

\begin{lemma}
\label{lemma:4.3}
If $\rho_1:\Gamma\to\sln$ and $\rho_2:\Gamma\to\sln$ are projective Anosov and
$\rm b_{\rho_1}=\rm b_{\rho_2},$
then there exists \hbox{$g\in \ms{GL}_m(\mathbb R)$} such that $g\circ\xi_1=\xi_2.$

Moreover, if  $\rho_1$ is irreducible, then $g\left(\pi_m\circ\rho_1\right) g^{-1}=\pi_m\circ\rho_2.$
\end{lemma}

\begin{proof} 
Lemma \ref{cross ratio and dim} implies that $$\dim\braket{\xi_1(\bg)}= \dim\braket{\xi_2(\bg)}=p.$$ 
Choose $\{x_0,\ldots,x_p\}\subset\bg$  so that 
$$\{\xi_1(x_0),\ldots,\xi_1(x_p)\}\textrm{ and }\{\xi_2(x_0),\ldots,\xi_2(x_p)\}$$ 
are projective frames for $\braket{\xi_1(\bg)}$ and $\braket{\xi_2(\bg)}$ (see Lemma \ref{mutual frame}).

Choose $u_0\in\xi_1(x_0)$ and $\{\varphi_1,\ldots,\varphi_p\}\subset(\Real^m)^*$ such that $\varphi_i\in\theta_1(x_i)$ 
and $\varphi_i(u_0)=1.$
One may check that $\{\varphi_1,\ldots,\varphi_p\}$ is a basis for $\braket{\theta_1(\bg)}.$
Complete $\{\varphi_1,\ldots,\varphi_p\}$  to  a basis
$$\mathcal{B}_1=\{\varphi_1,\ldots,\varphi_p,\varphi_{p+1},\ldots,\varphi_m\}$$ 
for $(\Real^m)^*$ such that $\varphi_i(\braket{\xi_1(\bg)})=0$ for all $i>p$.
For $y\in\bg,$ the  projective coordinates
of $\xi_1(y)$ with respect to  the dual basis of $\mathcal{B}_1$ are given by
$$[\ldots:\braket{\varphi_i|\xi_1(y)}:\ldots]=[\ldots:\frac{\braket{\varphi_i|\xi_1(y)}}{\braket{\varphi_1|\xi_1(y)}}\frac{\braket{\varphi_1|u_0}}{\braket{\varphi_i|u_0}}:\ldots]$$ 
which reduces to
$$[{\rm b}_{\rho_1}(x_1,x_1,y,x_0),\ldots,{\rm b}_{\rho_1}(x_p,x_1,y,x_0),0,\ldots,0].$$

Now choose
$v_0\in\xi_2(x_0)$ and $\{\psi_1,\ldots,\psi_p\}$ such that $\psi_i\in\theta_2(x_i)$ and \hbox{$\psi_i(v_0)=1.$} 
One sees that $\{\psi_1,\ldots,\psi_p\}$ is a basis of $\braket{\theta_2(\bg)}.$ One can then
complete $\{\psi_1,\ldots,\psi_p\}$  to  a basis
$$\mathcal{B}_2=\{\psi_1,\ldots,\psi_p,\psi_{p+1},\ldots,\psi_m\}$$ 
for $(\Real^m)^*$ such that $\psi_i(\braket{\xi_2(\bg)})=0$ for all $i>p$.
One checks, as above, that if $y\in\bg$, then the projective coordinates
$\xi_2(y)$ with respect to  the dual basis of $\mathcal{B}_2$ are given by
$$[{\rm b}_{\rho_2}(x_1,x_1,y,x_0),\ldots,{\rm b}_{\rho_2}(x_p,x_1,y,x_0),0,\ldots,0].$$

We now choose $g\in\ms{GL}_m(\mathbb R)$ so that $g\varphi_i=\psi_i$ for all $i$. It follows from the
fact that ${\rm b}_{\rho_1}(x_i,x_1,y,x_0)={\rm b}_{\rho_2}(x_i,x_1,y,x_0)$ for all $i\leq p$, that
$g\circ\xi_1=\xi_2.$

Assume now that $\rho_1$ is irreducible, so that $p=m$. Lemma \ref{mutual frame} implies that
there exists a $(m+1)$-tuple  $(x_0,\ldots, x_m)$ of points in  $\bg$ such that 
$F=(\xi_1(x_0),\ldots,\xi_1(x_m))$ is a projective frame for $\rpn$ and $F^*=(\theta_1(x_0),\ldots,\theta_1(x_m))$ is
a projective frame for $\rpn^*$. Thus, using the notation of Lemma \ref{mub}, we have
that, given arbitrary distinct points  $z,w\in\bg$,
$$
\mu_{F,F^*}(\pi_m(\rho_1(\gamma)))=\left[{\rm b}_{\rho_1}(x_i,z,\gamma(x_j),w)\right]$$
Similarly
$$
\mu_{F,F^*}(g^{-1}\pi_m(\rho_2(\gamma))g)=\mu_{gF,gF^*}(\pi_m(\rho_2(\gamma)))=\left[{\rm b}_{\rho_2}(x_i,z,\gamma(x_j),w)\right]
$$
Thus, since ${\rm b}_{\rho_1}={\rm b}_{\rho_2}$,
$$
\mu_{F,F^*}(\rho_1(\gamma))=\mu_{F,F^*}(g^{-1}\rho_2(\gamma)g).
$$
Since  $\mu_{F,F*}$  is injective, see Lemma \ref{prel}, it follows that
$$g\left(\pi_m\circ\rho_1\right)g^{-1}=\pi_m\circ\rho_2.$$
\end{proof}

We can now prove our spectral rigidity theorem:

\begin{proof}[Proof of Theorem \ref{theorem:length0}:] 
Consider two projective Anosov representations \hbox{$\rho_1:\Gamma\to\sln$} and \hbox{$\rho_2:\Gamma\to\sln$} such that 
$\LL(\gamma)(\rho_1)=\LL(\gamma)(\rho_2)$ for all $\gamma\in\Gamma.$ 
Suppose that $\alpha$ and $\beta$ are infinite order, co-prime elements of $\Gamma$.
Proposition \ref{cr formula} implies that 
\begin{eqnarray*}
\rm b_{\rho_1}(\beta^-,\alpha^-,\alpha^+,\beta^+) &= &\lim_{n\to\infty}\frac{\LL(\alpha^n\beta^n)(\rho_1)}{\LL(\alpha)(\rho_1)^n
\LL(\beta)(\rho_1)^n}\cr
&=&\lim_{n\to\infty}\frac{\LL(\alpha^n\beta^n)(\rho_2)}{\LL(\alpha)(\rho_2)^n
\LL(\beta)(\rho_2)^n}\cr
&=& \rm b_{\rho_2}(\beta^-,\alpha^-,\alpha^+,\beta^+).
\end{eqnarray*}

Since pairs of fixed points of infinite order elements of $\Gamma$ are dense in $\bg^{(2)}$ \cite{gromov}
and $\rm b_{\rho_1}$ and $\rm b_{\rho_2}$ are continuous, we see
that $\rm b_{\rho_1}=\rm b_{\rho_2}.$ 

Lemma \ref{lemma:4.3} implies that there exists $g\in\ms{GL}_m(\mathbb R)$ such that $g\circ \xi_1=\xi_2$.
If $\rho_1$ is irreducible,
then Lemma \ref{lemma:4.3} guarantees that $g\left(\pi_m\circ\rho_1\right) g^{-1}=\pi_m\circ\rho_2,$
so
$$\pi_m \circ (g\rho_1 g^{-1})=\pi_m\circ\rho_2.$$
Notice that if $A$ and $B$ are proximal matrices such 
that $\pi(A)=\pi(B)$  and that the eigenvalues of $A$ and $B$ of maximal absolute value have the same sign, then $A=B$.
Therefore, if $\alpha$ is any infinite order element of $\Gamma$, $g\rho_2(\alpha)g^{-1}=\rho_1(\alpha)$.
It follows that $g\rho_2g^{-1}=\rho_1$ as claimed.
\end{proof}

\subsection{Renormalized intersection rigidity}

Theorem \ref{theorem:intersectionrig} follows from Corollary  \ref{Zclosureconvex}, Corollary \ref{corollary:c} and 
Corollary \ref{corquint} below, which is a consequence of a deep result of Benoist \cite{limite}.

If $\ms G$ is a real-algebraic semi-simple Lie group, let $\mathfrak a_{\ms G}$ be a Cartan subspace of the 
Lie algebra $\mk g$ of $\ms G$ and let $\mathfrak a_{\ms G}^+$ be a Weyl Chamber. 
Let \hbox{$\mu_{\ms G}:\ms G\to\mathfrak a_{\ms G}^+$} be the Jordan projection.

Let 
$$(\frak a_{\ms G}^+)^*=\{\varphi\in\frak a_{\ms G}^*:\varphi|\frak a_{\ms G}^+\geq0\}.$$
If $\varphi$ lies in the interior $\interior(\frak a_{\ms G}^+)^*$ of $(\frak a_{\ms G}^+)^*$,
then if $v\in\frak a_{\ms G}^+$ and $\varphi(v)=0$, then $v=0.$

For a subgroup $\Delta$ of $\ms G$ the \emph{limit cone} $\mathscr L_\Delta$ of $\Delta$ is the smallest closed cone in
$\mathfrak a_{\ms G}^+$ that contains $$\{\mu(g):g\in \Delta\}.$$

Benoist \cite{limite} proved that Zariski dense subgroups have limit cones with non-empty interior.

\begin{theorem}{\sc[Benoist]} \label{limit cone}
If $\Delta$ is a Zariski dense subgroup of a connected real-algebraic semi-simple Lie group $\ms G$, then $\mathscr L_\Delta$ has non empty interior.
\end{theorem}

Benoist's theorem implies the following corollary, which was explained to us  by J.-F. Quint. 
This corollary  is a stronger version of a result of  Dal'Bo-Kim \cite{dalbo-kim} (see also Labourie \cite[Prop. 5.3.6]{Lbook}).

\begin{corollary}{\sc[Quint]}\label{corquint} 
Suppose that $\Delta$ is a group, $\ms G_\rho$ and $\ms G_\eta$ are center-free connected real-algebraic 
semi-simple Lie groups without compact factors, and
\hbox{$\rho:\Delta\to \ms G_\rho$} and \hbox{$\eta:\Delta\to \ms G_\eta$} are Zariski dense representations.
If there exist 
\hbox{$\varphi_1\in\interior(\frak a_{G_\rho}^+)^*$} and \hbox{$\varphi_2\in\interior(\frak a_{\ms G_\eta}^+)^*$} such that
for all $g\in\Delta$ one has 
$$\varphi_1(\mu_{\ms G_\rho}(\rho (g)))=\varphi_2(\mu_{\ms G_2}(\eta (g))),$$
then $\eta\circ\rho^{-1}:\Delta\to\Delta$ extends to an isomorphism $\ms G_\rho\to \ms G_\eta.$
\end{corollary}

\begin{proof} Let $\ms H$ be the Zariski closure of the image of the product representation 
\hbox{$\rho\times\eta:\Delta\to \ms G_\rho\times \ms G_\eta,$} defined by $g\mapsto (\rho g,\eta g).$ Since the equation \begin{equation}\label{radio}\varphi_1(\mu_{\ms G_\rho} (g_1))=\varphi_2(\mu_{\ms G_\eta} (g_2))\end{equation} holds for every pair $(g_1,g_2)\in\rho\times\eta\,(\Delta),$ Benoist's \cite{limite} Theorem \ref{limit cone} implies that the same relation holds for every pair $(g_1,g_2)\in \ms H.$

The group $\ms H\cap (\ms G_\rho\times\{e\})$ is a normal subgroup of $\ms G_\rho,$ it is hence (up to finite index) a product of simple factors. Equation (\ref{radio}) implies that for all $(g,e)\in \ms H\cap (\ms G_\rho\times\{e\})$ necessarily one has $\varphi_1(\mu_{\ms G_\rho} g)=0.$ Since $\varphi_1(v)>0$ for all $v\in\frak a_{\ms G_\rho}^+-\{0\},$ one has $\mu_{\ms G_\rho}(g)=0.$ This implies that $\ms H\cap (\ms G_\rho\times\{e\})$ is a normal compact subgroup of $\ms G_\rho.$ Since $\ms G_\rho$ does not have compact factors and is center free one concludes that $\ms H\cap (\ms G_\rho\times{e})=\{e\}.$

The same argument implies that $\ms H\cap(\{e\}\times \ms G_\eta)=\{e\}$ and hence $\ms H$ is the graph of an isomorphism extending $\eta\circ\rho^{-1}.$
\end{proof}

\subsection{Rigidity for Hitchin representations} 
\label{hitchin rigidity}

O. Guichard \cite{guichard-or} has announced a classification of the Zariski closures of 
lifts of Hitchin representations.
 
\begin{theorem}{\sc[Guichard]} \label{Zariski-Guichard}
If $\rho:\pi_1(S)\to\sln$ is the lift of a Hitchin representation and $\ms H$ is the Zariski
closure of $\rho(\pi_1(S))$, then
\begin{itemize}
\item If $m=2n$ is even, $\ms H$ is conjugate to either $\tau_{m}(\ms{SL}_2(\mathbb R))$,
$\ms{Sp}(2n,\mathbb R)$ or $\ms{SL}_{2n}(\mathbb R)$.
\item
If $m=2n+1$ is odd and $m\ne 7$, then $\ms H$ is conjugate to either
$\tau_m(\ms{SL}_2(\mathbb R))$, $\ms{SO}(n,n+1)$ or $\ms{SL}_{2n+1}(\mathbb R)$.
\item 
If $m=7$, then $\ms H$ is conjugate to either $\tau_7(\ms{SL}_2(\mathbb R))$,
$\ms{G}_2$, $\ms{SO}(3,4)$  or $\ms{SL}_7(\mathbb R)$.
\end{itemize}
where $\tau_m:\ms{SL}_2(\mathbb R)\to\sln$ is the irreducible representation.
\end{theorem}

Notice in particular, that the Zariski closure of the lift of a Hitchin representation is always
simple and connected. 
We can then apply our rigidity theorem for renormalized intersection to get a rigidity statement
which is independent of dimension in the Hitchin setting.

\begin{corollary}{\sc[Hitchin rigidity]}
\label{hitchrigid}
Let $S$ be a closed, orientable surface and let
$\rho_1\in \mathcal H_{m_1}(S)$ and $\rho_2\in \mathcal H_{m_2}(S)$ 
be two Hitchin representations such that  
$$
\JJ(\rho_1,\rho_2)=1.
$$
Then, \begin{itemize}
\item either $m_1=m_2$ and $\rho_1=\rho_2$ in $\mathcal H_{m_1}(S)$,
\item or there exists an element $\rho$ of the Teichmüller space $\mathcal T(S)$ so that
$\rho_1=\tau_{m_1}(\rho)$ and  $\rho_2=\tau_{m_2}(\rho)$.\end{itemize}
 \end{corollary}

Observe that the second case in the corollary only happens if both $\rho_1$ and $\rho_2$ are Fuchsian.

\begin{proof} 
In order to apply our renormalized intersection rigidity theorem, we will need the following
analysis of the outer automorphism groups of the Lie algebras of  Lie groups which arise as Zariski closures of lifts
of Hitchin representations. This analysis was carried about by G\"undo\u gan \cite{gundogan}
(see Corollary 2.15 and its proof).

\begin{theorem}
\label{gundogan}
{\sc [G\"undo\u gan \cite{gundogan}]}
Let $\operatorname{Out}(\mathfrak g)$ be the group of exterior automorphism of the Lie algebra $\mk g$. Then, if $n>0$,
\begin{enumerate}
\item
If $\mk g=\mk{sl}_{2n+2}(\mathbb R)$, then $\operatorname{Out}(\mathfrak g)$ is isomorphic to $\left(\mathbb Z/2\mathbb Z\right)^2$
and is generated by $X\mapsto -X^t$, and conjugation by an element of $\ms{GL}_{2n+2}(\mathbb R)$.
\item 
If $\mk g=\mk{sl}_{2n+1}(\mathbb R)$, then $\operatorname{Out}(\mathfrak g)$ is isomorphic to $\mathbb Z/2\mathbb Z$
and is generated by $X\mapsto -X^t$.
\item 
If $\mk g=\mk{so}(n,n+1,\mathbb R)$, then $\operatorname{Out}(\mathfrak g)$ is isomorphic to $\mathbb Z/2\mathbb Z$ 
and is generated by conjugation by an element of $\ms{SL}_{2n+1}(\mathbb R)$.
\item
If $\mk g=\mk{sp}(2n+2,\mathbb R)$, then $\operatorname{Out}(\mathfrak g)$ is isomorphic to $\mathbb Z/2\mathbb Z$ and 
is generated by conjugation by an element of $\ms{GL}_{2n+2}(\mathbb R)$.
\item 
If $\mk g=\mk{g}_2$ then $\operatorname{Out}(\mathfrak g)$ is  trivial.
\item
If $\mk g=\mk{sl}_2(\mathbb R)$, then $\operatorname{Out}(\mathfrak g)$ is isomorphic to $\mathbb Z/2\mathbb Z$
and is generated by  conjugation by an element of $\ms{GL}_{2}(\mathbb R)$.
\item
If $\mk g=\mk{so}(n,1,\mathbb R)$, then $\operatorname{Out}(\mathfrak g)$ is isomorphic to $\mathbb Z/2\mathbb Z$
and is generated by  conjugation by an element of $\ms{GL}_{n+1}(\mathbb R)$.
\end{enumerate}
\end{theorem}

Let 
$\rho_1:\pi_1(S)\to \ms{PSL}_{m_1}(\mathbb R)$ and $\rho_2:\pi_1(S)\to \ms{PSL}_{m_2}(\mathbb R)$
be two Hitchin representations such that  
$$
\JJ(\rho_1,\rho_2)=1.
$$
Theorem \ref{Zariski-Guichard}
implies that $\ms G_{\rho_1}$ and $\ms G_{\rho_2}$ are simple and connected and have center contained in $\{\pm I\}$.

Theorem \ref{theorem:intersectionrig} implies that  there exists an isomorphism $\sigma:\sf G_{\rho_1}\to \ms G_{\rho_2}$ 
such that $\rho_2=\sigma\circ\rho_1$.  
If $\ms G_1$ is not conjugate to $\tau_{m_1}(\ms{SL}_2(\mathbb R))$, then it follows from 
Theorem \ref{Zariski-Guichard}, that $m_1=m_2=m$, and  that, after conjugation of $\rho_1$, $\ms G_{\rho_1}=\ms G_{\rho_2}=\ms H$
so that $\sigma$ is an automorphism of $\ms H$.

We  first observe  that, since $\ms H$ is connected, there is an injective map from $\operatorname{Out}(\ms H)$ to $\operatorname{Out}(\mk h)$. 
We now analyze the situation in a case-by-case manner using G\"undo\u gan's Theorem \ref{gundogan}.

(1) If $\ms H=\ms {PG}_2$, then $\sigma$ is an inner automorphism, so $\rho_1=\rho_2$ in $\mathcal H_{7}(S)$.

(2) If $\ms H=\ms{PSO}(n,n+1)$  or $\ms H=\ms{PSp}(2n,\mathbb R)$,
$\sigma$ is either the identity or the conjugation by an element of $\ms{PGL}_{2n+1}(\mathbb R)$ or $\ms{PGL}_{2n}(\mathbb R)$, so
$\rho_1=\rho_2$ in $\mathcal H_{2n+1}(S)$ or $\mathcal H_{2n}(S)$.

(3) If $\ms H=\sln$, then,  after  conjugation  of $\rho_1$ by an element of $\ms{PGL}_m(\mathbb R)$, $\sigma$ is either trivial or
$\rho_2=\eta\circ\rho_1$ where $\eta(g)=\operatorname{transpose}(g^{-1})$. 
If $\sigma$ is non-trivial, then since $\JJ(\rho_1,\rho_2)=1$ Corollary \ref{corollary:c} implies that there exists $c>0$ so that
$$c\mu_1(\rho_1(\gamma))=\mu_1((\rho_2(\gamma))=-\mu_m(\rho_1(\gamma))$$
for all $\gamma\in\Gamma,$ where 
$$(\mu_1,\ldots,\mu_m):\sln\to\{(a_1,\ldots,a_m)\in\Real^m:\sum a_i=0\textrm{ and }a_1\geq\cdots\geq a_m\}$$
is the Jordan projection of $\sln.$ Thus, the limit cone of $\rho_1(\Gamma)$ has empty interior. 
Since $\rho_1(\Gamma)$ is Zariski dense, this contradicts
Benoist's Theorem \ref{limit cone}. Therefore, $\rho_1=\rho_2$ in $\mathcal H_m(S)$ in this case as well.

(4) If $\ms G_{\rho_1}$ is conjugate to $\tau_{m_1}(\ms{SL}_2(\mathbb R))$, then $\ms G_{\rho_2}$ is
conjugate to $\tau_{m_2}(\ms{SL}_2(\mathbb R))$. So, after conjugation, there exist Fuchsian representations,
$\eta_1:\pi_1(S)\to \ms{SL}_2(\mathbb R)$ and $\eta_2:\pi_1(S)\to \ms{SL}_2(\mathbb R)$, such
that $\rho_1=\tau_{m_1}\circ\eta_1$, $\rho_2=\tau_{m_1}\circ\eta_1$ and there exists an automorphism
of $\sigma$ of $\ms{SL}_2(\mathbb R)$ such that $\sigma\circ\eta_1=\eta_2$.
Case (6) of G\"undo\u gan's Theorem then implies
that $\eta_1$ is conjugate to $\eta_2$ by an element of $\ms{GL_2}(\mathbb R)$. Therefore, we are in the
second case of Theorem \ref{hitchrigid}. This completes the proof.

\end{proof}

\subsection{Benoist representations}
\label{convex projective}

We say that an open subset $\Omega$ of $\rpn$ is \emph{properly convex} if its intersection with any 
projective line is connected and its closure $\bar\Omega$ is contained in the complement of a projective hyperplane.
Moreover, a properly convex open set $\Omega$ is said to be {\em strictly convex} if its boundary $\partial\Omega$
does not contain a projective line segment. A subgroup 
\hbox{$\Delta\subset\operatorname{Aut}(\Omega)=\{g\in\mathsf{PGL}_m(\mathbb R): g\Omega=\Omega\}$} is said to \emph{divide} the open properly convex set  $\Omega$ if the quotient $\Delta\backslash\Omega$ is compact.
Benoist \cite[Thm. 1.1]{benoist-divisible1} proved that if $\Delta$ divides the properly convex open set $\Omega$, 
then $\Omega$ is  strictly convex if and only if $\Delta$ is hyperbolic.

\begin{definition}\label{benoistrep} 
If $\Gamma$ is a torsion-free hyperbolic group, a faithful representation 
\hbox{$\rho:\Gamma\to \mathsf{PGL}_m(\mathbb R)$} is a {\em Benoist representation} if $\rho(\Gamma)$ divides
an open strictly convex set $\Omega\subset\rpn.$
\end{definition}

It is a consequence of Benoist's work \cite{benoist-divisible1} that a Benoist representation is  irreducible
and projective Anosov
(see Guichard-Wienhard \cite[Proposition 6.1]{guichard-wienhard} for a detailed explanation).

Benoist \cite[Corollary 1.2]{benoist-divisible3}  (see also Koszul \cite{koszul})
proved that the space $B_m(\Gamma)$ of Benoist representations of
$\Gamma$ into $\psln$ is a collection of components of
$\operatorname{Hom}(\Gamma,\psln)$. Let 
$$\mathcal B_m(\Gamma)=B_m(\Gamma)/\mathsf{PGL}_m(\mathbb R).$$
We call the components of $\mathcal B_m(\Gamma)$ {\em Benoist components}.

Benoist \cite[Theorem 1.3]{benoist-divisible0} proved that the Zariski closure of any  Benoist representation is either
$\psln$ or is conjugate to \hbox{$\ms{PSO}(m-1,1)$}.
We may thus apply the technique of proof of Theorem \ref{hitchrigid} to prove:

\begin{corollary}{\sc [Benoist rigidity]}
\label{convex rigid} Let $\rho_1,\rho_2\in \mathcal B_m(\Gamma).$ If ${\bf J}(\rho_1,\rho_2)=1,$
then $\rho_1=\rho_2$ in $\mathcal B_m(\Gamma)$.
\end{corollary}

The same techniques also provide the following related rigidity result for Benoist representations.
Observe that if $\rho$ is a projective Anosov representation, 
then so is $\ad\rho:\Gamma\to\ms{PGL}(\frak{sl}(m,\Real))$ 
(see the discussion in Guichard-Wienhard \cite[Section 10.2]{guichard-wienhard})
If $\eta(g)=(g^{-1})^t$ for all $g\in \mathsf{PGL}_m(\mathbb R)$, and 
$\rho\in \mathcal B_m(\Gamma)$, then $\eta\circ\rho$ is the dual (or contragredient) representation of $\rho$.

\begin{corollary} 
\label{adjoint Benoist}
If $\rho_1,\rho_2\in\mathcal B_m(\Gamma)$, then $\JJ(\ad\rho_1,\ad\rho_2)=1$ if and only 
if either $\rho_1=\rho_2$ or $\rho_2=\eta\circ\rho_1.$
\end{corollary}

As a consequence, we recover a result of Cooper-Delp \cite{cooper-delp} and Kim \cite{kim-JDG}
which asserts that if $\rho_1,\rho_2\in\mathcal B_m(\Gamma)$ are the holonomies of strictly convex projective structures
with the same Hilbert marked length spectrum, then $\rho_1$ and $\rho_2$ either agree or are dual.
Recall that if $\rho\in\mathcal B_m(\Gamma)$ and $\gamma\in\Gamma,$ then the length, in the Hilbert metric, 
of the closed geodesic on $\rho(\Gamma)\backslash\Omega_\rho$ associated to $[\gamma]$ is 
$$\frac{\mu_1(\rho(\gamma))-\mu_m(\rho(\gamma))}2$$
(see, for example, Benoist \cite[Proposition 5.1]{benoist-divisible1}).
Furthermore, if \hbox{$g\in \mathsf{PGL}_m(\mathbb R)$} then 
$$\log(\Lambda(\ad g))=\mu_1(g)-\mu_m(g).$$ 
Hence if $\rho_1$ and $\rho_2$ are the holonomies of strictly convex projective structures
with the same Hilbert marked length 
spectrum, then $\Lambda(\ad \rho_1(\gamma))=\Lambda(\ad \rho_2(\gamma))$ for all $\gamma\in\Gamma$.
Hence,  $\JJ(\ad\rho,\ad\rho_2)=1$, so the result follows from Corollary \ref{adjoint Benoist}.

\section{Proofs of main results}\label{proofmainresult}

In this section, we assemble the proofs of the results claimed in the introduction.
Several of the results have already been established.

The inequality in Theorem \ref{intersection} follows from Corollary \ref{J non-negative} and rigidity follows from 
Theorem \ref{theorem:intersectionrig}. 
Theorem \ref{theorem:length}  is proven in Section \ref{rigidity} as Theorem \ref{theorem:length0}, while
Corollary  \ref{hitch-rigid} is proven as Corollary \ref{hitchrigid}.

Theorem \ref{analytic} follows from Proposition \ref{pro:ana} and Corollary \ref{J non-negative}.
Theorem \ref{geodesic flow} combines the results of Propositions \ref{geoflowreparam} and
\ref{pro:geodano}.

The proof of Theorem \ref{path metric} is easily assembled.

\medskip\noindent
{\em Proof of Theorem \ref{path metric}:} 
Consider the pressure form defined on 
$\mathcal{C}_g(\Gamma,\ms G)$ as in Definition \ref{def:pressureform}. Recall that by 
Corollary \ref{J non-negative} the pressure form is  non-negative.
Moreover, by Corollary \ref{typ0b} the  pressure form is positive definite, so gives a Riemannian metric.
The invariance with respect to $\operatorname{Out}(\Gamma)$ follows directly from the definition.

\medskip\noindent
{\em Proof of Corollary \ref{hitchin metric}:}
Corollary \ref{hitchin component} implies that every Hitchin component lifts to a component of 
$\mathcal{C}_g(\pi_1(S),\sln)$ which is an analytic manifold.
Theorem \ref{path metric} then assures that the
pressure form is an analytic Riemannian metric which is invariant under the action of the mapping class group.
Entropy is constant on the Fuchsian locus, so if $\rho_1,\rho_2\in\mathcal{T}(S)$,
the renormalized intersection has the form
\begin{eqnarray*}
\JJ(\tau_m\circ\rho_1,\tau_m\circ\rho_2) &=&
\lim_{T\to\infty}\frac{1} {\#(R_{\tau_m\circ\rho_1}(T))}\sum_{[\gamma]\in R_{\tau_m\circ\rho_1} }\frac{\log\Lambda(\tau_m\circ\rho_2)(\gamma)}{\log\Lambda(\tau_m\circ \rho_1)(\gamma)}\\
& = &\lim_{T\to\infty}\frac{1} {\#(R_{\rho_1}(T))}\sum_{[\gamma]\in R_{\rho_1} }\frac{\log\Lambda(\rho_2)(\gamma)}{\log\Lambda(\rho_1)(\gamma)}\\
\end{eqnarray*}

Wolpert \cite{wolpert} showed that the Hessian of the final expression, regarded as a function on $\mathcal{T}(S)$, is
a multiple of the  Weil-Petersson metric (see also Bonahon \cite{bonahon} and McMullen \cite[Theorem 1.12]{mcmullen}).

\medskip\noindent
{\em Proof of Corollary \ref{psl2c}:}
We may assume that $\Gamma$ is  the fundamental group of a compact 3-manifold with non-empty
boundary, since otherwise ${\Cc}(\Gamma,\ms{PSL}_2(\mathbb C))$ consists of 0 or 2 points.

We recall, from  Theorem  \ref{CC manifold}, that the deformation space  $\Cc(\Gamma,\ms{PSL}_2(\mathbb C))$ is an analytic manifold. Let $\alpha:\ms{PSL}_2(\mathbb C)\to \sln$ be the Pl\"ucker representation given by 
Proposition \ref{plucker}. 

If we choose co-prime infinite order elements $\alpha$ and $\beta$ of $\Gamma$,
we may define a global analytic lift
$$\omega:\Cc(\Gamma,\ms{PSL}_2(\mathbb C))\to\hom(\Gamma,\ms{PSL}_2(\mathbb C))$$
by choosing $\omega([\rho])$  to be a representative $\rho\in[\rho]$ so that $\rho(\alpha)$ has attracting fixed point 0 and
repelling fixed point $\infty$ and $\rho(\beta)$ has attracting fixed point $1$.
Then 
$$A=\alpha\circ\omega:\Cc(\Gamma,\ms{PSL}_2(\mathbb C))\to\hom(\Gamma,\sln)$$
is an analytic family of projective Anosov homomorphisms.

We define the associated entropy $\bar h$ and  renormalised intersection $\bar\JJ$ functions on 
${\Cc}(\Gamma,\ms{PSL}_2(\mathbb C))$
by setting
$$\bar h([\rho])=h(A([\rho]))\ \ \textrm{and}\ \ \bar\JJ([\rho_1],[\rho_2])=\JJ(A([\rho_1]),A([\rho_2)]).$$
Since $\omega$ is analytic, both $\bar h$ and $\bar\JJ$ vary analytically over ${\Cc}(\Gamma,\ms{PSL}_2(\mathbb C))$
and we may again define a non-negative 2-tensor on  the tangent space $\ms T{\Cc}(\Gamma,\ms{PSL}_2(\mathbb C))$
which we again call the pressure form, by considering the Hessian of $\bar\JJ$. 

Let $\ms G=\alpha(\ms{PSL}_2(\mathbb C))$. Then $\ms G$ is a reductive subgroup of $\sln$.
If $\rho(\Gamma)$ is Zariski dense, then $A(\rho)(\Gamma)$ is Zariski dense in $\ms G$, so Lemma \ref{pro:gen}
implies that $\rho(\Gamma)$ contains a $\ms G$-generic element. Since $\alpha$ is an immersion, 
$$\alpha_*:H^1_\rho(\Gamma,\mk{sl}_2(\mathbb C))\to H^1_{\alpha([\rho])}(\Gamma,\mk g)$$
is injective where
$\mathfrak g$ is the Lie algebra of $\ms G$. 
Corollary \ref{typ0b} then implies that the pressure form 
on $\ms T_\rho{\Cc}(\Gamma,\ms{PSL}_2(\mathbb C))$ is Riemannian if $\rho$ is Zariski dense.

If $\rho=\omega([\rho])$ is not Zariski dense, then its limit set is a subset of $\hat{\mathbb R}\subset\hat{\mathbb C}$,
and the Zariski closure of
$\rho(\Gamma)$ is either \hbox{$\ms H_1=\ms{PSL}(2,\mathbb R)$} or 
\hbox{$\ms H_2=\ms{PSL}(2,\mathbb R)\cup (z\to -z)\ms{PSL}(2,\mathbb R)$}. 
Since each $\ms H_i$ is a real semi-simple Lie group, Proposition \ref{3-manifold case} then implies that
the subset of non-Zariski dense
representations in ${\Cc}(\Gamma,\ms{PSL}_2(\mathbb C))$ is an analytic submanifold. 
We then again apply
Corollary \ref{typ0b} to see that the restriction of the pressure form to the submanifold
of non-Zariski dense representations is Riemannian.

The pressure form determines a path pseudo-metric on the deformation space ${\Cc}(\Gamma,\ms{PSL}_2(\mathbb C))$, 
which is a Riemannian metric off the analytic submanifold of non-Zariski dense representations and restricts
to a Riemannian metric on the submanifold.
Lemma \ref{pro:path-metric} then implies that the path metric is actually a metric. This establishes the main claim.

Theorem \ref{CC manifold} implies that if $\Gamma$ is not either virtually free or virtually a surface group,
then every $\rho\in{\Cc}(\Gamma,\ms{PSL}_2(\mathbb C))$ is Zariski dense. Auxiliary claim (1)
then follows from our main claim.

In the case that $\Gamma$  is the fundamental group of a closed orientable surface, then the 
restriction of the pressure metric to the Fuchsian locus is given by the Hessian of the  intersection form $\II$.
It again follows from work of Wolpert \cite{wolpert} that the restriction to the
Fuchsian locus is a multiple of the Weil--Petersson metric. This establishes auxiliary claim (2).

\medskip\noindent
{\em Proof of Corollary \ref{rank 1}:}
Let $\alpha:\ms{G}\to \sln$ be the Pl\"ucker representation given by Proposition \ref{plucker}.
An analytic family \hbox{$\{\rho_u:\Gamma\to\ms G\}_{u\in M}$} of convex cocompact homomorphisms
parameterized by an analytic manifold $M$, gives rise to an analytic family 
$\{\alpha\circ\rho_u\}_{u\in M}$ of projective Anosov homomorphisms of $\Gamma$ into $\sln$.
Theorem \ref{analytic}, and Corollary \ref{entropy:rk1}
then imply that topological entropy varies analytically for this family.  Results of Patterson \cite{patterson}, Sullivan \cite{sullivan}, Yue \cite{yue} and 
Corlette-Iozzi \cite{corlette-iozzi} imply that the topological entropy agrees with the Hausdorff dimension of the limit set, so
Corollary \ref{rank 1} follows.

\medskip\noindent
{\em Proof of Corollary \ref{metric on anosov}:}
Given a semi-simple real Lie group $\ms G$ with finite center and a non-degenerate parabolic subgroup $\ms P$,
let  $\alpha:\ms G\to \sln$ be the Pl\"ucker representation  given by 
Proposition \ref{plucker}.  Then $\ms H=\alpha(\ms G)$ is a reductive subgroup of $\sln$.  

We will adapt the notation of Proposition \ref{zd variety}.
Let 
$$\widehat{\mathcal Z}(\Gamma;\ms G,\ms P)=\tilde{\mathcal Z}(\Gamma;\ms G,\ms P)/\ms G_0$$
where $\ms G_0$
is the connected component of $\ms G$. Then, $\widehat{\mathcal Z}(\Gamma;\ms G,\ms P)$ is a finite analytic manifold
cover of  the analytic orbifold $\mathcal Z(\Gamma;\ms G,\ms P)$ with covering transformations given by $\ms G/\ms G_0$,
see Proposition \ref{zd manifold}.
Since $\mathsf G^0$ acts freely on  $\tilde{\mathcal Z}(\Gamma;\ms G,\ms P)$,  the slice theorem
implies that if $[\rho]\in\widehat{\mathcal Z}(\Gamma;\ms G,\ms P)$, 
then there exists a neighborhood $U$ of $[\rho]$ and a lift 
$$\beta:U\to \tilde{\mathcal Z}(\Gamma;\ms \ms G,\ms P)\subset\hom(\Gamma,\ms G).$$
Then $\omega=\alpha\circ\beta$ is an analytic family of 
$\ms H$-generic projective Anosov homomorphisms parameterized by $U$.
The  Hessian of the pull-back of the renormalized intersection gives
rise to an analytic 2-tensor, again called the pressure form,
on $\ms T U$.  Suppose that  $v\in\ms T_z \tilde U$ has pressure norm  zero.
Then Corollary \ref{typ0b} implies that   ${\rm D}\omega(v)$ is trivial in $H^1_{\omega(z)}(\Gamma, \mk h)$
where $\mk h$ is the Lie algebra of $\ms H$.
Since $\alpha$ is an immersion,
$$\alpha_*:H^1_{\beta(z)}(\Gamma,\mk{g})\to H^1_{\omega(z)}(\Gamma,\mk h)$$
is an isomorphism. Since  $\beta_*$ identifies $\ms T_z U$ with
$H^1_{\beta(z)}(\Gamma,\mk{g})$ this implies that $v=0$, so the pressure form on $\ms T U$ is non-degenerate.
Therefore, the pressure form is an analytic Riemannian metric on $\widehat{\mathcal Z}(\Gamma;\ms G,\ms P)$.
Since the pressure form is invariant under the action of $\ms G/\ms G_0$ it descends to a 
Riemannian metric on $\widehat{\mathcal Z}(\Gamma;\ms G,\ms P)$
This completes the proof.

\section{Appendix}

We used the following lemma in the proof of Corollary \ref{psl2c}.

\begin{lemma}
\label{pro:path-metric}
Let $M$ be a smooth manifold and let $W$ be a submanifold of $M$. Suppose that 
$g$ is a smooth non negative symmetric 2-tensor $g$ such that 
\begin{itemize}
\item $g$ is positive definite on $\ms T_xM$ if $x\in M\setminus W$,
\item  the restriction of $g$ to $\ms T_xW$ is positive definite if $x\in W$.
\end{itemize}
Then the path pseudo metric defined by $g$ is a metric.
\end{lemma}

\begin{proof} 
It clearly suffices to show that if $x\in M$, then there exists an open neighborhood $U$ of $M$ such that the restriction
of $g$ to $U$ gives a path metric on $U$. If $x\in M\setminus W$, then we simply choose a neighborhood $U$ of $x$
contained in $M\setminus W$ and the restriction of $g$ to $U$ is Riemannian, so determines a path metric.

If $x\in W$ we can find a neighborhood $U$ which is identified with a ball $B$ in $\mathbb R^n$ so that $W\cap U$ is identified
with $B\cap (\mathbb R^k\times \{0^{n-k}\})$.
Possibly after restricting to a smaller neighborhood, we can assume that there exists $r>0$ so that if
$v\in \ms T_zB$ and $v$ is tangent to $\mathbb R^k\times\{(z_{k+1},\ldots,z_n)\}$, then
$g(v,v)\ge r^2 ||v||^2$, where $||v||$ is the Euclidean norm of $v$.
If $z,w\in B$, $z\ne w$ and one of them, say $z$, is contained in $M\setminus W$, then $g$ is Riemannian in a neighborhood
of $z$, so $d_{U,g}(z,w)>0$ where $d_{U,g}$ is the path pseudo-metric on $U$ induced by $g$.
If $z,w\in W$, then the estimate above implies that
$d_{U,g}(z,w)\ge r d_B(z,w)$ where $d_B$ is the Euclidean metric on $B$. Therefore, $d_{U,g}$ is a metric on $U$ and
we have completed the proof.
\end{proof}


\begin{thebibliography}{100}

\bibitem{abramov} L.M. Abramov, ``On the entropy of a flow," {\em Dokl. Akad. Nauk. SSSR} {\bf 128}(1959), 873--875.

\bibitem{anderson-rocha} J.W. Anderson and A. Rocha, ``Analyticity of Hausdorff dimension of limit sets of Kleinian groups,''
{\em Ann. Acad. Sci. Fenn.} {\bf 22}(1997),  349--364.

\bibitem{limite} Y. Benoist, ``Propri{\'e}t{\'e}s asymptotiques des groupes lin{\'e}aires,''
{\em Geom. Funct. Anal.} {\bf 7}(1997), 1--47.

\bibitem{benoist2} Y. Benoist, ``Propri{\'e}t{\'e}s asymptotiques des groupes lin{\'e}aires II,''
{\em Adv. Stud. Pure Math.} {\bf 26}(2000), 33--48.

\bibitem{benoist-divisible0} Y. Benoist, ``Automorphismes des c\^ ones convexes,''
{\em Invent. Math.} {\bf 141}(2000) 149--193.

\bibitem{benoist-divisible1} Y. Benoist, ``Convexes divisibles I,'' in {\em Algebraic groups and arithmetic},
Tata Inst. Fund. Res. Stud. Math. {\bf 17}(2004), 339--374.

\bibitem{benoist-divisible3} Y. Benoist, ``Convexes divisibles III,'' {\em Ann. Sci. de l'E.N.S.} {\bf 38}(2005), 793--832.

\bibitem{bers-spaces} L. Bers, ``Spaces of Kleinian groups,'' in
{\em Maryland conference in Several Complex Variables I},
Springer-Verlag Lecture Notes in Math, No. 155(1970), 9--34.

\bibitem{bridgeman-pressure} M. Bridgeman, ``Hausdorff dimension and the Weil-Petersson extension
to quasifuchsian space,''
{\em Geom. and Top.} {\bf 14}(2010),  799--831.

\bibitem{bridgeman-taylor} M. Bridgeman and E. Taylor, ``An extension of the Weil-Petersson metric to quasi-Fuchsian space,''
{\em Math. Ann.} {\bf 341}(2008),  927--943. 

\bibitem{bonahon} F. Bonahon, ``The geometry of Teichm\"uller space via geodesic currents,''
{\em Invent. Math.} {\bf 92}(1988),  139--162. 

\bibitem{bowen1} R. Bowen, ``Periodic orbits of hyperbolic flows," {\em Amer. J. Math.} {\bf 94}(1972), 1--30.

\bibitem{bowen2} R. Bowen, ``Symbolic dynamics for hyperbolic flows," {\em Amer. J. Math.} {\bf 95}(1973), 429--460.

\bibitem{quasicircles} R. Bowen, ``Hausdorff dimension of quasi-circles,'' {\em Publ. Math. de l'I.H.E.S.} {\bf 50}(1979),
11--25.

\bibitem{BowenRuelle} R. Bowen \and D. Ruelle, ``The ergodic theory of axiom A flows,'' {\em Invent. Math.} {\bf 29}(1975),
181--202.

\bibitem{burger} M. Burger, ``Intersection, the Manhattan curve and Patterson-Sullivan theory in rank 2," {\em Internat. Math. Res. Notices} {\bf 7}(1993), 217--225.


\bibitem{champetier} C. Champetier, ``Petite simplification dans les groupes hyperboliques,''
{\em Ann. Fac. Sci. Toulouse Math.} {\bf 3}(1994),  161--221.


\bibitem{cooper-delp} D. Cooper and K. Delp,  ``The marked length spectrum of a projective manifold or orbifold,'' {\em Proc. of the A.M.S.} {\bf 138}(2010), 3361--3376.


\bibitem{coornaert-papadopoulos} M. Coornaert and A. Papadopoulos, ``Symbolic coding for the geodesic flow
associated to a word hyperbolic group,''
{\em Manu. Math.} {\bf 109}(2002),  465--492. 


\bibitem{corlette-iozzi} K.Corlette and A. Iozzi, ``Limit sets of discrete groups of isometries of
exotic hyperbolic spaces,'' {\em Trans. A.M.S.} {\bf 351}(1999), 1507--1530. 

\bibitem{dalbo-kim} F. Dal'Bo and I. Kim, ``A criterion of conjugacy for Zariski dense subgroups,''
{\em Comptes Rendus Math.} {\bf 330} (2000),  647--650.

\bibitem{delzant} T. Delzant, O. Guichard, F. Labourie and S. Mozes,
``Displacing representations and orbit maps,'' in {\em Geometry, rigidity and group actions},
Univ. Chicago Press, 2011, 494--514.

\bibitem{dreyer} G. Dreyer, ``Length functions for Hitchin representations,'' {\em Algebraic and Geometric Topology}, to appear,
preprint available at: arXiv:1106.6310.

\bibitem{goldman-darvishzadeh} M. Darvishzadeh and W. Goldman, ``Deformation spaces of convex real
projective structures and hyperbolic structures,'' {\em J. Kor. Math. Soc.} {\bf 33}(1996), 625--639.

\bibitem{gromov} M. Gromov, ``Hyperbolic groups,'' in {\em Essays in Group Theory},
MSRI Publ. 8 (1987), 75--263.

\bibitem{guichard-wienhard} O. Guichard and A. Wienhard, ``Anosov representations: Domains of discontinuity and
applications,'' {\em Invent. Math.} {\bf 190}(2012), 357--438.

\bibitem{guichard-or} O. Guichard, {\em oral communication}.

\bibitem{gundogan} H. G\"undo\u gan. ``The component group of the automorphism group
of a simple Lie algebra and the splitting of the corresponding short exact sequence,''
{\em J. Lie Theory} {\bf 20}(2010), 709--737.


\bibitem{hps} M. W.  Hirsch, C. C.  Pugh and M. Shub,
{\em Invariant manifolds}
Lecture Notes in Mathematics, Vol. 583, 1977

\bibitem{hitchin} N. Hitchin, ``Lie groups and Teichm\"uller space,''
{\em Topology} {\bf 31}(1992), 449--473. 

\bibitem{hubbard} J. Hubbard, {\em Teichm\"uler theory and applications to geometry, topology and dynamics. Vol 1.}, Matrix Editions, Ithaca.

\bibitem{humphreys} J.E. Humphreys \emph{Linear algebraic groups}, Graduate Text in Mathema-
tics {\bf 21}, Springer Verlag , New York, 1981.

\bibitem{johnson-millson} D. Johnson and J. Millson, ``Deformation spaces associated to compact hyperbolic
manifolds,'' in {\em Discrete Groups and Geometric Analysis}, Progress in Math., vol. 67(1987), 48--106.

\bibitem{kapovich}  M. Kapovich, \emph{Hyperbolic manifolds and discrete groups},
Progr. Math. {\bf 183},  Birkh{\"a}user, 2001.

\bibitem{KKPW} A. Katok, G. Knieper, M. Pollicott and H. Weiss, ``Differentiability and analyticity of topological entropy
for Anosov and geodesic flows,''
{\em Invent. Math.} {\bf 98}(1989), 581--597. 

\bibitem{KKW} A. Katok, G. Knieper, and H. Weiss, ``Formulas for the derivative and critical points of topological 
entropy for Anosov and geodesic flows,'' 
{\em Comm. Math. Phys.} {\bf 138}(1991), 19--31. 

\bibitem{kim-ergodic} I. Kim, ``Ergodic theory and rigidity on the symmetric space of non-compact type,''
{\em Erg. Thy. Dyn. Sys.} {\bf 21}(2001), 93--114. 


\bibitem{kim-JDG}I. Kim, ``Rigidity and deformation space of strictly convex real projective structures,'' {\em J. Differential Geom.} {\bf 58}(2001), 189--218.

\bibitem{kim-zhang} I. Kim and G. Zhang, ``K\"ahler metric on Hitchin component'', preprint available at  \texttt{http://arxiv.org/pdf/1312.1965v1.pdf}.

\bibitem{koszul} J.L. Koszul, ``D\'eformation des connexions localement plates,''
{\em Ann. Inst. Four.} {\bf 18}(1968), 103--114.

\bibitem{labourie-anosov} F. Labourie, ``Anosov flows, surface groups and curves in projective space,''
{\em Invent. Math.} {\bf 165}(2006), 51--114.

\bibitem{cross} F. Labourie. ``Cross Ratios, Surface Groups, $\sf{SL}_n(\Real)$ and Diffeomorphisms of the Circle,"
{\em Publ. Math. de l'I.H.E.S.} {\bf 106}(2007), 139--213.

\bibitem{labourie-cubic} F. Labourie, ``Flat projective structures on surfaces and cubic differentials,''
{\em P.A.M.Q.} {\bf 3}(2007), 1057--1099.

\bibitem{labourie-ens} F. Labourie, ``Cross ratios, Anosov representations and the energy functional on 
Teichm\"uller space,''
{\em Ann. Sci. E.N.S.} {\bf 41}(2008),  437--469. 

\bibitem{labourie-cyclic} F. Labourie ``Cyclic surfaces and Hitchin components in rank 2'', preprint available at
\texttt{http://arxiv.org/pdf/1406.4637.pdf}.

\bibitem{labourie-wentworth} F. Labourie, R. Wentworth ``The pressure metric along the fuchsian locus" in preparation.

\bibitem{Lbook} F. Labourie, {\em Lectures on representations of surface groups}, Zurich Lectures in Advanced 
Mathematics, 2013, 145 pages.

\bibitem{li} Q. Li, ``Teichm\"uller space is totally geodesic in Goldman space,''
preprint, available at:  \texttt{http://front.math.ucdavis.edu/1301.1442}

\bibitem{livsic}A.N. {Liv\v sic}, ``Cohomology of dynamical systems," {\em Math. USSR Izvestija} {\bf 6}(1972).

\bibitem{loftin} J. Loftin, ``Affine spheres and convex $\mathbb{RP}^2$ structures,'' {\em Amer. J. Math.} {\bf 123}(2001), 255--274.

\bibitem{lubotzky-magid} A. Lubotzky and A. Magid, {\em Varieties of representations of finitely generated groups},
{\em Mem. Amer. Math. Soc.} {\bf 58}(1985), no. 336.


\bibitem{manhe} R. Ma{\~n}\'e, {\em Teoria Erg\'odica}, Projeto Euclides, IMPA, 1983.

\bibitem{mcmullen} C. McMullen, ``Thermodynamics, dimension and the Weil-Petersson metric,''
{\em Invent. Math.} {\bf 173}(2008), 365--425.

\bibitem{mineyev} I. Mineyev, ``Flows and joins of metric spaces,'' 
{\em Geom. Top.} {\bf 9}(2005), 403--482. 

\bibitem{parry-pollicott} W. Parry and M. Pollicott, {\em Zeta functions and the periodic orbit 
structure of hyperbolic dynamics}, {\em Ast\'erisque} {\bf 187-188}(1990).

\bibitem{patterson} S. Patterson, ``The limit set of a Fuchsian group,''
{\em Acta Math.} {\bf 136}(1976),  241--273. 

\bibitem{pollicottcoding} M. Pollicott, ``Symbolic dynamics for Smale flows,'' {\em Am. J. of Math.} {\bf 109}(1987), 183--200.

\bibitem{pollicott-sharp} M. Pollicott and R. Sharp, ``Length asymptotics in higher Teichm\"uller theory,''
{\em Proc. A.M.S.} {\bf 142}(2014), 101--112.

\bibitem{razovin} D. Ragozin, ``A normal subgroup of a semisimple Lie group is closed.''
{\em Proc. A.M.S.} {\bf 32}(1972), 632--633.
 
\bibitem{ruelle-hd} D. Ruelle, ``Repellers for real analytic maps,''
{\em Ergodic Theory Dynamical Systems} {\bf 2}(1982),  99--107. 

\bibitem{ruelle} D. Ruelle, {\em Thermodynamic Formalism}, Addison-Wesley, London.

\bibitem{sambarino-thesis} A. Sambarino, {\em Quelques aspects des repr\'esentations lin\'eaires
des groupes hyperboliques}, Ph.D. Thesis, Paris Nord, 2011.

\bibitem{sambarino-quantitative} A. Sambarino, ``Quantitative properties of convex representations,''
{\em Comm. Math. Helv.}, {\bf 89}(2014), 443--488.

\bibitem{growth} A. Sambarino, ``Hyperconvex representations and exponential growth,"  {\em Ergod.
Th. Dynam. Sys.} {\bf 34}(2014), 986--1010.

\bibitem{shub} M. Shub, {\em Global Stability of Dynamical Systems}, Springer-Verlag, 1987.

\bibitem{sullivan} D. Sullivan, ``The density at infinity of a discrete group of hyperbolic motions,''
{\em Inst. Hautes \'Etudes Sci. Publ. Math.} {\bf 50}(1979), 171--202. 

\bibitem{tapie} S. Tapie, ``A variation formula for the topological entropy of convex-cocompact manifolds,''
{\em Erg. Thy. Dynam. Sys.} {\bf 31}(2011),  1849--1864. 

\bibitem{wolpert} S. Wolpert, ``Thurston's Riemannian metric for Teichm\"uller space,''
{\em J. Diff. Geom.} {\bf 23}(1986), 143--174. 

\bibitem{yue} C. Yue, ``The ergodic theory of discrete isometry groups on manifolds of variable negative curvature,'' 
{\em Trans. A.M.S.} {\bf 348}(1996),  4965--5005. 

\end{thebibliography}
\end{document}